\documentclass{amsart}
\usepackage{amsmath,amssymb,amscd,amsthm,verbatim,alltt,amsfonts,array}
\usepackage[english]{babel}
\usepackage{latexsym}
\usepackage{amssymb}
\usepackage{euscript}
\usepackage{graphicx}
\usepackage[all]{xy}
\usepackage{hyperref}
\usepackage{enumerate}

\newtheorem{theorem}{Theorem}
\theoremstyle{plain}

\newtheorem{algorithm}{Algorithm}

\newtheorem{corollary}{Corollary}

\newtheorem{definition}{Definition}
\newtheorem{example}{Example}

\newtheorem{lemma}{Lemma}

\newtheorem{problem}{Problem}
\newtheorem{proposition}{Proposition}
\newtheorem{remark}{Remark}

\numberwithin{equation}{section}

\usepackage{algorithm}
\usepackage{algorithmic}

\newcommand{\INIT}{\item[\algorithmicinit]}
\newcommand{\algorithmicinit}{\textbf{Initialization:}}

\usepackage{url,doi}

\newcommand{\lhom}[3]{\mathsf{hom}_{#1}(#2,#3)}
\newcommand{\Ima}{\mathrm{Im\;}}
\newcommand{\field}[1]{\mathsf{#1}}
\newcommand{\cuerpo}[1]{\mathsf{#1}}

\newcommand{\End}[2]{\mathrm{End}_{#1}(#2)}

\def\overset#1#2{
 \mathrel{\mathop{\kern 0pt#2}\limits^{#1}}}

\def\underset#1#2{
 \mathrel{\mathop{\kern 0pt#2}\limits_{#1}}}

\newcommand{\supp}{\operatorname{supp}}

\newcommand{\Newton}{\mathcal{N}}

\newlength{\resto}
\newcommand{\lres}[2]%
 {%
 \settoheight{\resto}{$#1$}%
 \addtolength{\resto}{-1pt}%
 \raisebox{\resto}{\scriptsize $#2$}\overline{#1}%
 }
\newcommand{\rres}[2]%
 {%
 \settoheight{\resto}{$#1$}%
 \addtolength{\resto}{-1pt}%
 \overline{#1}\raisebox{\resto}{\scriptsize $#2$}%
 }

\newcommand{\lm}{\operatorname{lm}}

\newcommand{\lc}[1]{\operatorname{lc}(#1)}
\newcommand{\Exp}{\operatorname{Exp}}
\newcommand{\SP}{\operatorname{SP}}
\newcommand{\level}{\operatorname{level}}

\newcommand{\Ext}{\operatorname{Ext}}

\newcommand{\abs}[2]{|#2|_{\gordo{#1}}}

\newcommand{\bilin}[2]{\langle #1, #2 \rangle}
\newcommand{\GKdim}{\operatorname{GKdim}}
\newcommand{\card}{\operatorname{card}}

\newcommand{\rgcd}{\operatorname{r-gcd}}

\newcommand{\llcm}{\operatorname{l-lcm}}
\newcommand{\lrem}{\operatorname{l-rem}}

\newcommand{\N}{\mathbb{N}}
\newcommand{\Nn}{\mathbb{N}^n}
\newcommand{\Nnm}{\mathbb{N}^{n, (m)}}

\newcommand{\BbbN}{{\mathbb{N}}}

\newcommand{\BbbC}{{\mathbb{C}}}

\newcommand{\sexp}[1]{\mathrm{sexp}(#1)}
\renewcommand{\exp}[1]{\mathrm{exp}(#1)}

\newcommand{\gordo}[1]{\boldsymbol{#1}}

\newcommand{\xx}{\gordo{x}}
\newcommand{\aalpha}{\gordo{\alpha}}
\newcommand{\ggamma}{\gordo{\gamma}}
\newcommand{\eepsilon}{\gordo{\epsilon}}
\newcommand{\bbeta}{\gordo{\beta}}

\newcommand{\lend}[2]{\mathrm{End}_{#1}(#2)}
\newcommand{\mono}[2]{\gordo{#1}^{\gordo{#2}}}

\DeclareMathOperator{\diag}{diag}
\DeclareMathOperator{\Ann}{Ann}
\DeclareMathOperator{\ann}{ann}
\DeclareMathOperator{\GL}{GL}
\DeclareMathOperator{\gr}{gr}
\DeclareMathOperator{\Syz}{Syz}

\begin{document}

\title[Module Theory]{Basic Module Theory over Non-Commutative Rings with Computational Aspects of Operator Algebras}


\author{Jos\'e G\'omez-Torrecillas}
\address{Departamento de \'Algebra,
Universidad de Granada,
E-18071 Granada, Spain}
\email{\tt gomezj@ugr.es}

\thanks{Partially supported by the Spanish Ministerio de 
 Ciencia en Innovaci\'on and the European Union --- grant MTM2010-20940-C02-01. The author wishes to thank Thomas Cluzeau, Viktor Levandovskyy, Georg Regensburger, and the anonymous referees for their comments that lead to improve this paper. }


\maketitle

\begin{abstract}
The present text surveys some relevant situations and results where basic Module Theory interacts with computational aspects of operator algebras. We tried to keep a balance between  constructive and algebraic aspects. 

\end{abstract}

\tableofcontents

\section*{Introduction}
Let $R$ be any unital ring (in the sense, e.g.,  of \cite{Anderson/Fuller:1992}). A left $R$-module is just an additive group over which the elements of $R$ act as linear operators. Thus, results on left modules over general rings (like the existence of free resolutions, or Jordan-H\"older and Krull-Schmidt theorems for modules of finite length) are of interest for operator algebras. This kind of general results are part of Module Theory (see \cite{Anderson/Fuller:1992,Stenstrom:1975} for two expositions with different orientations).  On the other hand, the rings appearing as operator algebras (in their algebraic version) are rather concrete (rings of differential-difference operators being the prototype). When some basic results from Module Theory are interpreted for modules over some of these operator algebras, both theories benefit from the interaction: for instance, Jordan-H\"older theorem gives a unique factorization theorem for polynomials in a single Ore extension $\field{D}[ x ;\sigma,\delta]$ of a skew field $\field{D}$, or Krull-Schmidt theorem would be seen under the perspective of the searching of canonical forms for pseudo-linear operators over vector spaces. 

From an algorithmic or constructive point of view, many interesting operator algebras are polynomial algebras, well understood that the ``variables'', that should represent operators, will not commute among them or even with the coefficients,  which  sometimes represent ``functions''. For instance, left ideals and modules over iterated Ore extensions and other non-commutative polynomial rings play a role in the algebraic modeling of ``real'' situations, from cyclic convolutional codes \cite{Gluesing/Schmale:2004} to linear control systems \cite{Chyzak/Quadrat/Robertz:2005}, apart from the well-known use in Algebraic Analysis of the rings of differential operators (again, non-commutative polynomials), see \cite{Bjork:1979},  and \cite{Quadrat:2013} for a constructive approach. 

The present text surveys some relevant situations and results where basic Module Theory interacts with computational aspects of operator algebras. We tried to keep a balance between  constructive and algebraic aspects. 

Section \ref{sec:modules} fixes basic notations and notions, and explains how to represent finitely presented left modules, morphisms between them, and also their kernels and images, by means of matrices with entries in the base ring $R$. This is made over a general ring, and the included material is very elementary.  We hope, however, that our concise presentation would be useful when dealing with concrete situations. 

In Section \ref{sec:Ore} our ring will be an Ore extension $R = \field{D}[ x ;\sigma,\delta]$ of a skew field  $\field{D}$. We will discuss the structure of the finitely generated left modules over such an $R$. The included results  are inspired by the theory of non-commutative principal ideal domains developed in  the excellent monograph \cite{Jacobson:1943} by N. Jacobson. Our exposition highlights the constructive aspects of the theory, connecting the aforementioned general results for modules of finite length with some algorithms of factorization of non-commutative polynomials (see \cite{Singer:1996,Giesbrecht:1998,Gomez/alt:unp}) in certain special cases of Ore extensions. The existence of normal forms for matrices with coefficients in $R$ are also discussed in the light of the structure of left $R$-modules. 

Sections \ref{sec:LPBW} and \ref{sec:Buchberger} deal with a generalization of $R = \field{D}[ x ;\sigma,\delta]$ to several variables, namely, the left PBW rings from \cite{Bueso/Gomez/Lobillo:2001a}. As in the commutative case, a main tool for handling these rings from the effective point of view  are  Gr\"obner bases for submodules of free left modules. The definition of left PBW ring tries to capture the essential property that makes the multivariable division algorithm, and Buchberger's  algorithm  work: the exponent (i.e., the ``multi-degree'' of the leading term) of a product of polynomials is the sum of the exponents of the factors. The class of all left PBW rings contains some interesting subclasses,  like  the solvable polynomial algebras from \cite{Kandri-Rody/Weispfenning:1988}, solvable polynomial rings \cite{Kredel:1993}, or Ore algebras \cite{Chyzak/Salvy:1998}. Many algorithms working for some of these rings are implemented on \textsc{SINGULAR} (see \cite{Levandovskyy:2005}) or \textsc{Maple} (see \cite{Chyzak/Quadrat/Robertz:2007}). Any differential operator ring $\field{D}[x_1, \delta_1] \cdots [x_n, \delta_n]$ is a (left and right) PBW ring not covered,  in general,  by the aforementioned classes. 

An alternative approach to non-commutative Gr\"obner bases is to work on factors of finitely generated free algebras over commutative fields (see the survey \cite{Mora:1994}). In our opinion, interesting as it is, it does not fit to the examples of rings of operators as well as the left PBW do. One of the reasons is that the latter are  often non finitely generated algebras  over commutative fields (this is the case, for instance, of differential operator algebras over rational function fields), so they cannot be written as factor algebras of a finitely generated free algebra.  

Section \ref{sec:GKdimPBW} discusses an algorithm for the computation of the Gelfand-Kirillov dimension of a left $R$-module, where $R$ is a PBW algebra (i.e., a polynomial solvable algebra from \cite{Kandri-Rody/Weispfenning:1988} or a $G$-algebra from \cite{Levandovskyy:2005}).  This objective serves as an excuse to characterize PBW algebras within the class of  all  filtered algebras. Since the transfer of properties from the associated graded algebra to the filtered one is a very well developed theory (see, among others, \cite{Bjork:1979,Gomez/Lenagan:2000,Krause/Lenagan:2000,Lorenz:1988,McConnell/Robson:1988,McConnell/Stafford:1989}), this characterization allows to have many good properties for any PBW algebra for free \cite[Theorem 4.1]{Bueso/Gomez/Lobillo:2001b}.   In addition,  some other computational aspects of these algebras will be discussed.

The overview ends with an appendix on computer algebra systems by Viktor Levandovskyy. I am most grateful to him for writing it.

\section{Modules over a Non-Commutative Ring}\label{sec:modules}
First, we will fix some notations and recall some definitions on rings and their modules. For all undefined notions (as ring, (left) ideal, submodule, factor module, etc.) we refer the reader to \cite{Stenstrom:1975},  \cite{Jacobson:1980}, or \cite{Hazewinkel/alt:2004}. 

All rings will be assumed to have unit, that is, a neutral element, denoted by $1$, for the multiplication.  All homomorphisms of rings are assumed to be unital.  

The \emph{center} of a ring $R$ is the commutative subring defined by
$$C(R) = \{r \in R : rr' = r'r \; \forall \, r' \in R \}\,.$$
By a $\field{k}$-algebra, where $\field{k}$ is a field, we understand a ring $R$ such that $C(R)$ contains the field $\field{k}$. 

 A \emph{skew field} (or division ring) is a nontrivial  ring $\field{D}$ such that every nonzero element has a multiplicative inverse.

Given any ring $R$, by $R^{op}$ we denote its \emph{opposite ring}, that is, $R^{op}$ coincides with $R$ as an additive group, but it is endowed with the new product defined by $r \cdot s = sr$ for all $r, s \in R$.   

\subsection{On the Notion of Module, and other Basic Concepts} 
Basic examples of  non-commutative  rings appear as endomorphism rings of abelian groups.  Concretely, let  $M$ be an abelian group. We use additive notation $+$ for its group operation. Then the set $\mathrm{End}(M)$ of all  group endomorphisms of $M$   is an abelian group with the operation, denoted also by $+$, defined in the obvious way. In $\mathrm{End}(M)$ there exists a second operation: the composition of maps. With these two operations, $\mathrm{End}(M)$ becomes an associative ring with unit (the identity map, of course). 

\begin{definition}\label{modulo}
Let $M$ be any abelian group, and $R$ be a ring. We say that $M$ is a \emph{left $R$-module} if there exists a homomorphism of rings $\lambda : R \to \mathrm{End}(M)$. We say also that $M$ has the structure of a left $R$-module given by $\lambda$. Different homomorphisms of rings $R \to \mathrm{End}(M)$ lead to different left $R$-module structures on the same $M$.
\end{definition}

Given a left $R$-module $\lambda : R \to \mathrm{End}(M)$, we define a map $R \times M \to M$ sending a pair $(r, m) \in R \times M$ to the element 
$rm := \lambda(r)(m)$ of $M$. The following properties hold for every $m, m' \in M, r, r' \in R$.
\begin{enumerate}
\item $r(m + m') = rm + rm'$.
\item $(r + r') m = rm +  r' m$.
\item $r(r'm) = (rr')m$.
\item $1m = m$.
\end{enumerate}
Conversely, an abelian group $M$ with a map $R \times M \to M$ that sends each $(r,m) \in R  \times M$ to  some element of $M$, denoted by    $rm$, and satisfying the properties above, gives a homomorphism of rings $\lambda : R \to \mathrm{End}(M)$, by means of the rule $\lambda (r)(m) : = rm$, which makes $M$ a left $R$-module. Definition \ref{modulo}  stresses  the fact that elements of a ring $R$ are interpreted as linear operators on any left $R$-module. We sometimes write ${}_RM$ to emphasize the left action of $R$ on $M$.  A \emph{right} $R$-module is, by definition, a left module over $R^{op}$.

 Submodules of a left module ${}_RM$ are defined in the obvious way. The set of submodules of ${}_RM$ is a (modular and pseudo-complemented) lattice (see \cite[Chapter 3]{Stenstrom:1975}), being the lower bound given by the intersection and the upper bound by the sum of submodules. 

The \emph{annihilator of} ${}_RM$ is the two-sided ideal of $R$ defined as 
\[
\Ann_R(M) = \ker \lambda = \{ r \in R : rm = 0 \; \forall m \in M \},
\]
where the symbol $\ker$ means, as usual, the kernel of an additive map.  That $\Ann_R(M)$ is a two-sided ideal of $R$ is immediately deduced from the fact that $\lambda$ is a homomorphism of rings. 
Obviously, $M$ becomes a left $R/\Ann_R(M)$-module, and the lattice of left $R$-submodules of $M$ is the same than that of left $R/\Ann_R(M)$-submodules.

A \emph{homomorphism}  (or \emph{morphism}) of left $R$-modules is a map $h : M \to N$ such that $h(rm+r'm') = rh(m) + r'h(m')$ for all $r, r' \in R, m, m' \in M$. The homomorphism is said to be an \emph{isomorphism} if $h$ is bijective. If $h$ is an isomorphism, then its inverse map $h^{-1}$ is also an isomorphism of modules. If $M$ and $N$ are connected by an isomorphism, then we say that $M$ and $N$ are \emph{isomorphic}.   

Any homomorphism of left $R$-modules $h : M \to N$ encodes an isomorphism $\tilde{h} : M/\ker h \to \Ima h$, defined by $\tilde{h}(m + \ker h) = h(m)$. This is the first Noether's isomorphism Theorem.  In particular, if $m$ is any element of a left $R$-module $M$, then we have a homomorphism of left $R$-modules $\rho_m : R \to M$ defined by $\rho_m(r) = rm$ for all $r \in R$. The image of $\rho_m$ is $Rm$, the cyclic $R$-submodule of $M$ generated by $m$, while the kernel of $\rho_m$ is the \emph{annihilator} of $m$, namely, the left ideal of $R$ 
\[
\ann_R(m) = \{ r \in R : rm = 0\}\,.
\]
First Noether's isomorphism Theorem yields an isomorphism of left $R$-modules $Rm \cong R/\ann_R(m)$.

When $R$ is an algebra over a field $\field{k}$, then any left $R$-module $M$ becomes, by an obvious restriction of scalars, a vector space over $\field{k}$. A straightforward computation shows that the image of $\lambda :  R  \to \mathrm{End}(M)$ is included in the subring $\lend{\field{k}}{M}$ of all $\field{k}$-linear endomorphisms of $M$. That is, a left $R$-module is a $\field{k}$-vector space $M$ with a homomorphism of $\field{k}$-algebras $\lambda : R \to \lend{\field{k}}{M}$.

  If ${}_RM, {}_RN$ are modules, then $\lhom{R}{M}{N}$ denotes the set of all homomorphisms of left $R$-modules from $M$ to $N$. This set is an additive group with the obvious sum of homomorphisms, and it is a vector space over $\field{k}$ if $R$ is a $\field{k}$-algebra. In contrast with the commutative case, $\lhom{R}{M}{N}$ is not in general a (left or right) $R$-module.  When $M = N$, we use the notation $\lend{R}{M} = \lhom{R}{M}{M}$, and $\lend{R}{M}$ is considered as a ring with multiplication defined as the opposite of the composition of maps.

\subsubsection{Direct Sums.}
Given $R$-modules ${}_RM$ and ${}_RN$, we may endow the cartesian product $M \times N$ with the structure of a left $R$-module with the sum defined componentwise, and the left action of $R$ given by $r(m,n) = (rm,rn)$ for all $(m, n) \in M \times N, r \in R$. This new left $R$-module is called the external \emph{direct sum} of $M$ and $N$, and it is denoted by $M \oplus N$. This notation is also used for the decomposition of a module as \emph{internal} direct sum of two submodules. In this case, if we have a module ${}_RL$ and two submodules $M, N$ of $L$,  then  we say that $L$ is the internal direct sum of $M$ and $N$ if every $x \in L$ decomposes in a unique way as $x = m + n$ with $m \in M$ and $n \in N$. It happens that the map $x \mapsto (m,n)$ is an isomorphism of left $R$-modules from $L$ to the external direct sum $M \oplus N$. This identification, up to isomorphisms, of the internal and the external direct sum is often assumed. 

We may analogously form the direct sum of finitely many modules. In particular, given a positive integer $t$, we may consider the direct sum of $t$ copies of $R$, and thus the left $R$-module  $R^t = \{(r_1, \dots, r_t) : r_1, \dots, r_t \in R\}$ for each $t \geq 1$.  

One may conceive the direct sum of the modules of an infinite family $\{ M_i : i \in I \}$ of left $R$-modules. As a set, $\bigoplus_{i \in I} M_i$ is the subset of the cartesian product $\prod_{i \in I}M_i$ whose elements are the $I$-tuples with finitely many non zero components. Symbolically,
\[
\bigoplus_{i \in I} M_i = \{ (m_i)_{i \in I} : m_i \in M_i \text{ for all } i \in I, m_i \neq  0 \text{ for finitely many } i \in I \}\,.
\]
This set is a left $R$-module, called the \emph{direct sum of the family $\{ M_i : i \in I \}$} with the operations
\[
(m_i)_{i \in I} + (m'_i)_{i \in I} = (m_i+m_i')_{i \in I}, \qquad r(m_i)_{i \in I} = (rm_i)_{i \in I}\,.
\]

If, for any set $I$, we put $M_i = R$ for every $i \in I$, then we may form the direct sum $\bigoplus_{i \in I} M_i  $, which will be denoted by $R^{(I)}$. 

\subsubsection{Finitely Generated Modules, Free Modules and Bases.}
Recall that a subset $\{m_1, \dots, m_t \}$ of a left $R$-module $M$ is said to be a \emph{set of generators} of $M$ if for every $m \in  M$ there exist $r_1, \dots, r_t \in R$ such that $m = r_1 m_1 + \cdots + r_t m_t$. In such a case, $M$ is said to be \emph{finitely generated}. The left $R$-module $R^t$ is finitely generated for any $t \geq 1$, being a set of generators $\{\gordo{e}_1, \dots, \gordo{e}_t \}$, where $\gordo{e}_i$ denotes the $t$-tuple with a unique component equal to $1$ at the $ i $-th position, and the rest of entries equal to $0$. If $M$ is any other left $R$-module with a set of  generators $\{m_1, \dots, m_t \}$, then the map $\varphi : R^t \to M$ defined as $\varphi (\sum_ir_i\gordo{e}_i) = \sum_ir_im_i$ for all $\sum_ir_i\gordo{e}_i \in R^t$ is a surjective homomorphism of left $R$-modules. Thus, by Noether's first isomorphism theorem, $M$ is isomorphic to the quotient left $R$-module $R^t/\ker \varphi$. Observe that $\ker \varphi = \{ 0 \}$ if and only if the set $\{ m_1, \dots, m_t \}$ is $R$-linearly independent. In this case, $M$ is isomorphic to $R^t$ and $\{ m_1, \dots, m_t \}$ is a \emph{basis} of $M$. We say then  that $M$ is a free left $R$-module  on (or with)  the basis $\{m_1, \dots, m_t \}$.  We will often denote by $\mathbf{F}_t$ a free left $R$-module with a basis of cardinal $t$. 

We may consider free modules with non necessarily finite bases; they are needed for some constructions,  like Ore extensions.  Thus, given a left $R$-module $M$ and a set $\{ x_i : i \in I \}$ of elements of $M$ indexed by a set $I$ (finite or not), we say that $\{x_i : i \in I \}$ is a \emph{set of generators} of ${}_RM$ if for every $m \in M$, there exist finitely many $r_{i_1}, \dots, r_{i_s} \in R$,  such that $m = r_{i_1} x_{i_1} + \cdots + r_{i_s}x_{i_s}$, for $i_1, \dots, i_s \in I$. If the elements of the set of generators $\{ x_i : i \in I \}$ of $M$ are $R$-linearly independent, then we say  that ${}_RM$ is free  on  the basis $\{ x_i : i \in I \}$. In such a case, $M$ is isomorphic to $R^{(I)}$.

\subsection{A Motivating Example of Module}
We recall the basic example that connects linear differential operators and modules. 

\subsubsection{Differential Operators.}
Consider a homogeneous ordinary linear differential equation
\begin{equation*}
a_n(t)\frac{d^ny(t)}{dt^n} + \cdots + a_1(t)\frac{dy(t)}{dt}+
a_0(t)y(t) = 0,
\end{equation*}
where the $a_i(t)$'s are functions in some field (e.g. the field
$\mathbb{C}(t)$  of rational functions over the complex numbers.) Consider the linear differential operator
\begin{equation}\label{L}
L = a_n(t)\frac{d^n}{dt^n} + \cdots + a_1(t)\frac{d}{dt}+ a_0(t)
\end{equation}
acting on some (commutative)  $\mathbb{C}$-algebra $\mathcal{F}$ of functions containing  $\mathbb{C}(t)$  as a subalgebra (e.g. $\mathcal{F}$ could be  the algebra of all meromorphic functions). 
This allows
 $\mathbb{C}(t)$  act on $\mathcal{F}$ by multiplication. Thus,
 $\mathbb{C}(t) \cup \{ d/dt \} \subseteq
\End{\mathbb{C}}{\mathcal{F}}$  generates a
 $\mathbb{C}$-subalgebra, say $R$, of the (huge)  non-commutative  algebra 
$\End{\mathbb{C}}{\mathcal{F}}$ of all linear endomorphisms of $\mathcal{F}$. Obviously, $L \in R$ and the rule
\[
L \cdot y(t) = L(y(t))
\]
endows $\mathcal{F}$ with the structure of a left $R$-module, and
 equation \eqref{L} becomes $L \cdot y(t) = 0$.

Let $y(t) \in \mathcal{F}$ be a solution of $L \cdot y(t) = 0$.
Then the map
\[
R/RL \rightarrow \mathcal{F} \qquad (r + RL \mapsto r \cdot y(t))
\]
is a homomorphism of left $R$-modules. Here,
$RL = \{ rL ~|~ r \in R \}$ is the left ideal of $R$ generated by $L$.

Conversely, every homomorphism of left $R$-modules $\varphi :
R/RL \rightarrow \mathcal{F}$ provides a solution $y(t) :=
\varphi (1 + RL)$ of our differential equation. Therefore, the
generator $1 + RL$ of the left $R$-module $R/RL$ may be viewed as a
``generic solution''
of the differential equation $L(y(t)) = 0$.

The left $R$-module $R/RL$ contains relevant
information about the differential equation. For instance,  if two
differential  equations have
isomorphic associated left $R$-modules, then their sets of solutions are tightly related. More concretely, let $L_1, L_2 \in R$ be linear differential operators such that there exists an isomorphism of left $R$-modules $h : R/RL_1 \to R/RL_2$. 
Write $h (1 + RL_1) = F + RL_2$ for a suitable $F \in R$. Then $f \in \mathcal{F}$ is a solution of the diffential equation $L_2 \cdot  u = 0$ if and only if $F \cdot f$ is a solution of the differential equation $L_1 \cdot v =0$. In this way, isomorphic modules give equations whose sets of solutions are equivalent in a precise way.

More generally, a \emph{system} of  homogeneous  ordinary differential linear equations is identified with a finitely generated left
$R$-module, which is of the form $R^m/K$, where $K$ is a
submodule of  a finitely generated free left $R$-module $R^m$.

This gives, for instance, a safe
framework to declare when two systems are equivalent.

\subsubsection{Which Kind of Ring is our $R$?}

By Leibniz's rule, for every $a(t) \in  \mathbb{C}(t)$,  we get the  equality of operators in $\lend{\mathbb{C}}{\mathcal{F}}$ 
\begin{equation}\label{Leibniz}
d/dt \circ a(t)  = a(t) \circ d/dt  + \frac{d a(t)}{dt} ,
\end{equation}
whence $R$ is a non-commutative ring (commutative rings rarely
appear in nature). It follows from \eqref{Leibniz} that every differential operator $L
\in R$ may be written as in \eqref{L}.  On the other hand, the powers $(\frac{d}{dt})^n =
\frac{d^n}{dt^n}$ are clearly linearly independent over $ \mathbb{C}(t) $.
Therefore, the elements of $R$ are already polynomials in the ``variable'' $x:= \frac{d}{dt}$ with coefficients \emph{on the left} in the field $ \mathbb{C}(t)$. The non-commutative multiplication is completely determined by \eqref{Leibniz}, that is, $x a = a x + da/dt$ for all $a \in  \mathbb{C}(t)$.  This is a  basic  example of \emph{Ore extension} of a field \cite{Ore:1933}. 
The fundamental property of the ring $R$ (and of more general Ore extensions) is the existence of (left and right) Euclidean
division algorithms, which makes possible the computation of canonical forms of matrices with entries in $R$ and, as a consequence, of a structure theorem for finitely generated left $R$-modules (fully developed in \cite{Jacobson:1943}).

\subsubsection{The Abstract Setting: Systems of Linear Equations over a
Non-Com\-mu\-ta\-tive Ring.}

Let $\mathcal{F}$ be a vector space over a field $\field{k}$ and $R
\subseteq \End{\field{k}}{\mathcal{F}}$ a subalgebra (under composition)
of linear operators. Thus, $\mathcal{F}$ is a left $R$-module
whose elements play the r\^{o}le of ``functions'', on which the elements of $R$ operate. A system of linear equations over $R$ is
a left $R$-module $R^m/ K $, where $ K$ is a finitely generated
submodule of the free left $R$-module $ R^m$, that is, a finitely presented left $R$-module. 
The solutions of the system should be  then  found in  $\mathcal{F}^m$, since 
\[
\lhom{R}{R^m/K}{\mathcal{F}} \subseteq \lhom{R}{R^m}{\mathcal{F}} \cong \mathcal{F}^m,
\]
as vector spaces over $\field{k}$. 
Thus, a central problem is the computational treatment of
(finitely presented) modules over $R$.

Let us focus our attention in a very elementary problem: Given
$r_1, \dots, r_s \in R$ and $r \in R$, is $r$ linearly dependent
of $r_1, \dots, r_s$? In other words, we want to solve the
equation
\[
r = g_1r_1 + \cdots + g_sr_s \qquad (g_1, \dots, g_s \in R)\,.
\]
This is the ``membership problem'', i.e., is $r$ an element of the
left ideal $Rr_1 + \cdots + Rr_s$ generated by $r_1, \dots, r_s$?

 \subsubsection{A Canonical Example.}  Let $A = \mathbb{C}[t_1, \dots, t_n]$ be the ring of polynomials with complex coefficients
in the (commuting) variables $t_1, \dots, t_n$.  Let
$\mathcal{F}$ be an algebra of functions containing $A$ such that
the partial  derivatives  $\partial_i = \frac{\partial}{\partial t_i}$
``make sense'' on $\mathcal{F}$. Consider $R$ the subring of
$\End{\mathbb{C}}{\mathcal{F}}$ generated by $A$ and $\partial_1,
\dots, \partial_n$. So, every element of $R$ is a linear
differential operator with polynomial coefficients acting on
$\mathcal{F}$ (this is the $n$-th complex Weyl algebra).

The product in $R$ is built (from Leibniz's rule) on the
commutation relations $\partial_i t_i - t_i \partial_i = 1$
(think of $t_i$ as operators!) and any other pair among the
$t_i$'s and $\partial_j$'s commute.

An interpretation from the point of view of the theory of
differential equations of the membership problem here is to know
whether a given linear differential equation in a system can be dropped
because it is ``linearly dependent'' of the rest.

 We cannot expect an algorithm to solve the membership
problem for any ring $R$. However, some basic algorithms can be developed
for certain non-commutative polynomial rings (see  Section \ref{sec:LPBW})  that include differential operator rings and many difference operator rings. 

\subsection{Linear Algebra over a Non-Commutative Ring}\label{LA}
Linear algebra  rests  on the  assignment  of coordinates to each vector, once a basis is fixed in a vector space. In this way, linear transformations are represented by matrices, and their composition by the multiplication of matrices. It is possible to extend these ideas  to handle with homomorphisms of modules over a general ring $R$. 

\subsubsection{Morphisms between Free Modules and Matrices.}
Let $\mathbf{F}_s$ and $\mathbf{F}_t$ be finitely generated free left $R$-modules with bases $\{ \mathbf{u}_1, \dots, \mathbf{u}_s \}$ and $\{ \mathbf{e}_1, \dots, \mathbf{e}_t \}$, respectively. Let $R^{s \times t}$ denote the set of all matrices with $s$ rows and $t$ columns with  entries  in $R$. Any homomorphism of left $R$-modules $\psi: \mathbf{F}_s \to \mathbf{F}_t$ is represented by a matrix $A_{\psi} \in R^{s \times t}$ in the usual way: the matrix $A_{\psi} = (a_{ij})$ is defined by the conditions $\psi (\mathbf{u}_i) = \sum_j a_{ij}\mathbf{e}_j$ for $i = 1, \dots, s$.  The homomorphism $\psi$ is easily recovered from $A_{ \psi}$. Explicitly, if $\mathbf{u} =\sum_i x_i \mathbf{u}_i$ for  $\mathbf{x} = (x_1, \dots, x_s) \in R^s$, then $\psi (\mathbf{u}) = \sum_j y_j\mathbf{e}_j$, where $\mathbf{y} = (y_1,\dots,y_t) \in R^t$ is given by the matrix product 
\[
\mathbf{y} = \mathbf{x} A_{\psi}\,.
\] 
A straightforward computation shows that if $\xymatrix{\mathbf{F}_s \ar^-{\psi}[r] & \mathbf{F}_t \ar^-{\phi}[r] & \mathbf{F}_u}$ are homomorphism of free left $R$-modules, then, once fixed bases, $A_{\phi \psi} = A_{\psi} A_{\phi}$. This, in particular, shows that the endomorphism ring $\lend{R}{\mathbf{F}_s}$ of a free left $R$-module $\mathbf{F}_s$ is  isomorphic   to the matrix ring $R^{s \times s}$  (recall that the product of the ring $\lend{R}{\mathbf{F}_s}$ is the opposite of the composition of maps).

The group of units of $R^{s \times s}$, that is, of invertible matrices, is denoted by $\GL_s(R)$. 

\subsubsection{Presentations of Finitely Generated Modules.}
Given $M$, a left $R$-module generated by finitely many elements $m_1,\dots,m_t \in M$, and the surjective homomorphism of left $R$-modules $\varphi : \mathbf{F}_t \to M$ defined by $\varphi (\sum_i r_i \mathbf{e}_i) = \sum_i r_i m_i$, where $r_1, \dots, r_t \in R$, there is no reason to expect that $\ker \varphi$ is a finitely generated left $R$-submodule of $\mathbf{F}_t$. Assume, however, that $\ker \varphi$ is finitely generated as a left $R$-module ($M$ is then said to be \emph{finitely presented}). If $k_1, \dots, k_s$ are generators of $\ker \varphi$, then $k_i = \sum_j a_{ij}\mathbf{e}_j$ for some coefficients $a_{ij} \in R$. The matrix $A_{ \psi} = (a_{ij}) \in R^{s \times t}$ defines a homomorphism $\psi : \mathbf{F}_s \to \mathbf{F}_t$. Since, by Noether's first isomorphism theorem, $M$ is isomorphic to the factor module $\mathbf{F}_t/\ker \varphi$, we see that $M$ is determined, up to isomorphisms, by the $s \times t$ matrix $A_{ \psi} = (a_{ij})$. 
 We say then that 
\begin{equation}\label{finpres}
\xymatrix{\mathbf{F}_s \ar^-{\psi}[r] & \mathbf{F}_t \ar^-{\varphi}[r] & M \ar[r] & 0}
\end{equation}
is a \emph{finite presentation} of $M$.  

\begin{definition}
A ring $R$ is said to be \emph{left noetherian} if every finitely generated left $R$-module has a finite presentation as in \eqref{finpres}. 
\end{definition}

The precedent discussion shows that a finitely generated left module over a left noetherian ring $R$ is essentially the same  as  a matrix with coefficients in $R$. However, this leads to the computational problem of deciding when two matrices represent isomorphic left $R$-modules. This is a difficult problem, tightly related to the existence of canonical or normal forms for matrices with coefficients in $R$. 

A well known characterization says that $R$ is left noetherian if and only if every left ideal of $R$ is finitely generated (see, e.g., \cite[Corollary 3.6]{Stenstrom:1975}). From the effective point of view, left noetherian rings are good because every finitely generated left $R$-module $M$ may be represented by a matrix with entries in $R$, as before. In fact, $M$ is isomorphic to $R^t/row (A_{\psi})$, where $row (A_{\psi})$ is the $R$-submodule of $R^t$ generated by the rows of $A_{ \psi}$.   

\subsubsection{Presentations of Homomorphisms.}
What about morphisms? A homomorphism $h : M \to N$ between finitely presented left $R$-modules leads to a diagram of homomorphisms of left $R$-modules
\begin{equation}\label{morfismos}
\xymatrix{\mathbf{F}_s \ar^-{\psi}[r] \ar^-{p}[d] & \mathbf{F}_t \ar^-{\varphi}[r] \ar^-{q}[d] & M \ar[r] \ar^{h}[d] & 0 \\
\mathbf{F}_{s'} \ar^-{\psi'}[r] & \mathbf{F}_{t'} \ar^-{\varphi'}[r] & N \ar[r] & 0 \rlap{\,.} }
\end{equation}
Here, the  bottom  row is a finite presentation of $N$, and $p, q$ are homomorphisms of left $R$-modules satisfying $h \varphi = \varphi' q$ and $q \psi = \psi' p$ (the diagram is then said to be \emph{commutative}).  

The construction of the morphisms $p, q$ from $h$ goes as follows. If $n_1, \dots, n_{t'}$ are generators of $N$, then $h(m_i) = \sum_{j} q_{ij}n_j$ for some $q_{ij} \in R$. The morphism $q$ is then defined by the $t \times t'$ matrix $(q_{ij})$.  Since $q$ induces, by restriction, a morphism from $\ker \varphi$ to $\ker \varphi'$, a similar procedure defines $p$. Conversely, given morphisms $p, q$ as in diagram \eqref{morfismos}, with $q \psi = \psi' p$, then $h : M \to N$ is well-defined by the rule $h(\varphi (v)) = \varphi' ( q (v))$ for every $v \in \mathbf{F}_t$. 

As a consequence, once fixed bases in $\mathbf{F}_s, \mathbf{F}_t, \mathbf{F}_{s'}$ and $\mathbf{F}_{t'}$, the morphisms $p, q$ are determined by matrices $A_p, A_q$ with coefficients in $R$. We get thus, as in \cite[Corollary 2.1]{Cluzeau/Quadrat:2008}, that a morphism $h : M \to N$ is defined by a pair of matrices $Q, P$ such that $A Q = P A' $, where $A = A_{\psi}$, $A'=A_{\psi'}$, $Q = A_q$ and $P = A_p$.

\subsection{Syzygies}
The effective treatment of modules and their homomorphisms has been developed over different kinds of polynomial rings (commutative or not)  in  many places (see, among others, \cite{Adams/Loustaunau:1994,Apel/Lassner:1988,Bueso/Gomez/Verschoren:2003,Kredel:1993,Levandovskyy:2005}). Our next aim is to distill the essence,   for a general ring $R$, of these algorithmic approaches.  This philosophy of abstracting the categorical component of a constructive approach to modules from the specific shape of the base ring is present in the proposal of   the  ``meta-package''  \emph{homalg} in \cite{Barakat/Robertz:2008}.  A similar substratum underlies other works, like \cite{Cluzeau/Quadrat:2008}.

Let $\psi : \mathbf{F}_t \to \mathbf{F}_m$ be a homomorphism of finitely generated free left $R$-modules. A basic problem is the computation of a presentation of the kernel of $\psi$. This kernel is the so called \emph{module of syzygies} of $\psi$. 

Let $A_{\psi} \in R^{t \times m}$ be the matrix representing $\psi$ with respect to the bases $B_t = \{\gordo{u}_1, \dots, \gordo{u}_t\}$ and $B_m=\{\gordo{e}_1, \dots, \gordo{e}_m\}$ of $\mathbf{F}_t$ and $\mathbf{F}_m$, respectively.  In coordinates with respect to $B_t$, $\ker \psi$ is encoded by  the matrix equation $$0 = (h_1, \dots, h_t)A_{\psi},$$ whose set of solutions is a left $R$-submodule of $R^t$, called the \emph{module of syzygies of $A_{\psi}$}, and denoted by $\Syz(A_{\psi})$. If $\{\gordo{s}_1, \dots, \gordo{s}_s \}$ is a set of generators of this module, then  its elements   are the coordinates with respect to $B_t$ of a finite set of generators of $\ker \psi$, and if we build the matrix $\mathsf{syz}(A_{\psi})\in R^{s \times t}$ whose rows are the vectors $\gordo{s}_i$, then this matrix gives a presentation of $\Ima \psi$.  

Assume that there is another homomorphism $\hat{\psi} : \mathbf{F}_{\hat{t}} \to \mathbf{F}_m$, and morphisms (always of left $R$-modules) $\hat{p} : \mathbf{F}_t \to \mathbf{F}_{\hat{t}}, p : \mathbf{F}_{\hat{t}} \to \mathbf{F}_t$ such that $\hat{\psi} \hat{p} = \psi$ and $\psi p = \hat{\psi}$. Then it is immediate that $\Ima (id - p\hat{p}) \subseteq \ker \psi$ and that, for $x \in \mathbf{F}_t$, $x \in \ker \psi$ if and only if $\hat{p}(x) \in \ker \hat{\psi}$. From these facts, one easily derives that 
\begin{equation}\label{nucleo}
\ker \psi = p( \ker \hat{\psi}) + \Ima (id - p\hat{p})\,.
\end{equation}
The expression \eqref{nucleo} may be used to compute $\ker \psi$ if $\ker \hat{\psi}$ is explicitly given (for instance,  if $\hat{\psi}$ is a sort of normal form of $\psi$). 
Working in coordinates with respect to bases in the different free modules involved, we obtain 

\begin{equation}\label{syzcomp}
\mathsf{syz}(A_{\psi}) = \left( \begin{array}{c} \mathsf{syz}(A_{\hat{\psi}})A_{p} \\ I_t-A_{\hat{p}}A_{p} \end{array} \right),
\end{equation}
where $A_{\hat{p}} \in R^{t \times \hat{t}}$ and $A_{p} \in R^{\hat{t} \times t}$ are such that $A_{\psi} = A_{\hat{p}}A_{\hat{\psi}}$ and $A_{\hat{\psi}} = A_{p}A_{\psi}$.
We will have the opportunity of applying \eqref{syzcomp} to obtain algorithms for the computation of the submodule of syzygies  in  some more concrete situations later.  Of course, one may expect more efficient alternatives to this general scheme (after all, $R$ is here \emph{any} ring) in specific situations. 

\subsection{Images and Kernels}
Let $h: M \to N$ be a homomorphism of left $R$-modules. Our next aim is to obtain a presentation of $\Ima h = \{ h(m) : m \in M\}$. We will first deal with the particular case where $M = \textbf{F}_m$ is a finitely generated free left module. The presentation of a general $h$ given in \eqref{morfismos} may then be simplified to
\[
\xymatrix{\langle 0 \rangle \ar^-{0}[r] \ar^-{0}[d] & \textbf{F}_m \ar@{=}[r] \ar^-{q}[d] & \mathbf{F}_m \ar[r] \ar^{h}[d] & 0 \\
\mathbf{F}_{s'} \ar^-{\psi'}[r] & \mathbf{F}_{t'} \ar^-{\varphi'}[r] & N \ar[r] & 0\rlap{\,.} }
\]
Then $h = \varphi' q$, and we need just to compute a presentation of the kernel of $\varphi' q$. 

\begin{lemma}\label{kernelcomp}
A presentation of the kernel of $\xymatrix{\mathbf{F}_m \ar^{q}[r] & \mathbf{F}_{t'} \ar^{\varphi'}[r] & N }$ is given by the first $m$ columns of the matrix
\[
\mathsf{syz} \left( \begin{array}{c} A_q \\ A_{\psi'}\end{array}\right).
\]
\end{lemma}
\begin{proof}
 Given  $y \in \mathbf{F}_{m}$ we have that $\varphi' q (y) = 0$ if and only if $q(y) \in \ker \varphi' = \Ima \psi '$. Thus, $y \in \ker \varphi' q$ if and only if there is $x \in \mathbf{F}_{s'}$ such that $ q (y) = \psi' (x)$.  Let $\langle q, -\psi' \rangle : \mathbf{F}_m \times \mathbf{F}_{s'} \to \mathbf{F}_{t'}$ be defined by 
\[
\langle  q , -\psi' \rangle (y, x) = q(y) - \psi'(x), \qquad \text{ for all } (y,x) \in \mathbf{F}_m \times \mathbf{F}_{s'}\,.
\]
We thus deduce that
\[
\ker \varphi' q = \{ y \in \mathbf{F}_m : \exists x \in \mathbf{F}_{s'}  \hbox{ with } (y,x) \in \ker \langle q, - \psi' \rangle \},
\]
and, since $(y,x) \in \ker \langle q, -\psi' \rangle$ if and only if $(y, -x) \in \ker \langle q, \psi' \rangle$, we get
\[
\ker \varphi' q = \{ y \in \mathbf{F}_m : \exists x \in \mathbf{F}_{s'}  \hbox{ with } (y,x) \in \ker \langle q, \psi' \rangle \}\,.
\]
Thus, any set of generators $\{(y_1, x_1), \dots, (y_r,x_r) \}$ of $\ker  \langle q, \psi' \rangle$ will give a set of generators $\{y_1, \dots, y_r \}$ of $\ker \varphi' q$. This finishes the proof.  
\end{proof}

Let us turn now to a general $h : M \to N$ with presentation
\begin{equation}\label{morfismosbis}
\xymatrix{\mathbf{F}_s \ar^-{\psi}[r] \ar^-{p}[d] & \mathbf{F}_t \ar^-{\varphi}[r] \ar^-{q}[d] & M \ar[r] \ar^{h}[d] & 0 \\
\mathbf{F}_{s'} \ar^-{\psi'}[r] & \mathbf{F}_{t'} \ar^-{\varphi'}[r] & N \ar[r] & 0\rlap{\,.} }
\end{equation}

As a consequence of Lemma \ref{kernelcomp} we obtain a procedure to compute a presentation of the left $R$-module $\Ima h = \{ h(m) : m \in M \}$, from the presentation of $h$ given in \eqref{morfismosbis}. 

\begin{proposition}\label{prop:Im}
A presentation of $\Ima h$ is given by the first $t$ columns of the matrix
\[
\mathsf{syz}\left( \begin{array}{c} A_q \\ A_{\psi'}\end{array}\right).
\]
\end{proposition}
\begin{proof}
Since $h \varphi = \varphi' q$, and $\varphi$ is surjective, we get that $\Ima \varphi' q = \Ima h$. The result follows from Lemma \ref{kernelcomp}.
\end{proof}

Finally, let us describe the kernel of a general morphism of left $R$-modules.

\begin{proposition}\label{prop:ker}
Let $h : M \to N$ be presented as in \eqref{morfismosbis}, and $S \in R^{r \times t}$ the matrix formed by the $t$ first columns of $\mathsf{syz}\left( \begin{array}{c} A_q \\ A_{\psi'}\end{array}\right)$. A presentation of $\ker h$ is given by the first $r$ columns of the matrix
\[
\mathsf{syz}\left( \begin{array}{c} S \\ A_{\psi}\end{array}\right).
\]
\end{proposition}
\begin{proof}
Since $\ker \varphi \subseteq \ker h \varphi$, we get that the restriction of $\varphi$ to $\ker h \varphi$ defines a surjective homomorphism of left $R$-modules $\overline{\varphi}: \ker h\varphi \to \ker h$. Moreover, $\ker \overline{\varphi} = \ker \varphi$.  On the other hand, $\ker h \varphi = \ker \varphi' q$, and, by Lemma \ref{kernelcomp}, the matrix $S$ gives a surjective homomorphism of left $R$-modules $s: \mathbf{F}_r \to \ker \varphi' q = \ker h \varphi$.
In resume, a presentation of $\ker h$ is computed as soon as the kernel of the surjective homomorphism
\[
\xymatrix{\textbf{F}_r \ar^-{s}[r] & \ker h\varphi \ar^{\overline{\varphi}}[r] & \ker h}
\]
is computed. Finally, $\ker \overline{\varphi} = \ker \varphi = \Ima \psi$, and the proof of Lemma \ref{kernelcomp} is easily adapted to describe the coordinates of a set of generators of $\ker \overline{\varphi} s$ as the first $r$ rows of the matrix $\mathsf{syz}\left( \begin{array}{c} S \\ A_{\psi}\end{array}\right)$. 
\end{proof}

\section{Modules over $\field{D}[ x ;\sigma,\delta]$}\label{sec:Ore}

Modules over an Ore extension (see the definition below) may be understood  as  pseudo-linear operators (see \cite{Leroy:1995} and \cite{Bronstein/Petkovsek:1991}). Thus, the study of the structure of these modules is important for the understanding and algorithmic treatment of pseudo-linear operators and, in particular, for (linear) ordinary differential operators.   

\subsection{Ore Extensions}

We first recall a basic construction in Ring Theory, namely that of Ore extension of a given ring.  These rings  were introduced by \O. Ore in \cite{Ore:1933}. 

\subsubsection{Definition of an Ore Extension.} Let $A$ be any ring, $\sigma : A \to A$ an endomorphism of rings, and $\delta: A \to A$ is a $\sigma$-derivation, that is, for all $a, b \in A$, 
\begin{equation}\label{dtorcida}
\delta (a+b) = \delta(a) + \delta(b), \quad \delta(ab) = \sigma(a)\delta(b) + \delta(a)b\,.
\end{equation}
 The construction of the \emph{Ore extension $R = A[ x ;\sigma,\delta]$ of $A$ by $(\sigma,\delta)$} goes as follows:
\begin{itemize}
\item $R$ is a free left $A$-module on the basis $\{  x ^n : n \geq 0 \}$. Thus, the elements of $R$ are \emph{left polynomials} of the form $a_0 + a_1 x  + \cdots + a_n x ^n$, with $a_i \in A$. 
\item The sum of polynomials is as usual. 
\item The product  of  $R$ is based on the following product rules: $ x ^n x ^m =  x ^{n+m}$, for $m, n \in \mathbb{N}$, and $ x  a = \sigma(a) x  + \delta(a)$ for $a \in A$.  This product is extended recursively to $R$. 
\end{itemize}

\begin{remark}
That the product just defined on $A[x;\sigma,\delta]$ is associative is not completely obvious. A proof may be found in \cite[Ch. 1]{Bueso/Gomez/Verschoren:2003}, for instance.  
\end{remark}

\begin{remark}
In our definition of $A[x;\sigma,\delta]$ there is a choice of the side, since we assume that $\{x^n : n \geq 0 \}$ is a basis of the Ore extension as a \emph{left} $A$-module. Of course, one may prefer to work by assuming that $\{x^n : n \geq 0\}$ is a basis as a right $A$-module (this is, for instance, the choice in \cite{McConnell/Robson:1988} or \cite{Cohn:1971}). In such a case, the skew derivation $\delta$ should satisfy $\delta(ab) = \delta(a)\sigma(b) + a \delta(b)$ for all $a,b \in A$, instead of \eqref{dtorcida}. Both choices lead to equivalent theories,  by replacing $A$ by its opposite ring $A^{op}$. However, once a side is fixed, say our ``left'' choice, we cannot assume that the ``monomials'' $x^n$ form a basis of the Ore extension as a right $A$-module (therefore, its elements cannot be understood as right polynomials), unless $\sigma$ is an autormorphism (see \cite[Proposition 3.9, Ch. 1]{Bueso/Gomez/Verschoren:2003}).
\end{remark}

Let $f \in A[x;\sigma,\delta]$ be a nonzero element, and consider its unique expression as a left $A$-linear combination of the elements of the basis $\{x^n : n \geq 0 \}$,
\[
f = a_0 + a_1 x  + \cdots + a_d  x^d,
\]
with $a_0,a_1, \dots, a_d \in A$ and $a_d \neq 0$. The \emph{degree} of $f$ is defined as $\deg(f) = d$.

If $A$ is a domain (that is, $ab = 0$ with $a, b \in A$ implies $a= 0$ or $b = 0$), and $\sigma$ is injective, then $\deg(fg) = \deg(f) + \deg(g)$ for every $f, g \in A[ x ;\sigma,\delta]$. We are assigning $\deg{0} = - \infty$, with the usual conventions for the symbol $-\infty$ with respect to the ordering and addition of integers. The \emph{leading monomial} $\lm(f)$ and \emph{leading coefficient} $\lc{f}$ are defined in the obvious way. 

Two special cases of Ore polynomials are of interest. If $\delta = 0$,  then  it is usually written $R = A[x;\sigma]$, and if $\sigma$ is the identity, it is omitted, and we denote $R = A[x;\delta]$.  The latter, when $(A,\delta)$ is a commutative \emph{differential ring} with derivation $\delta$ (see, e.g. \cite{VanDerPut/Singer:2003}), gives the connection between the Module Theory, and the differential modules. 

 In this section, we will study modules over an Ore extension $\field{D}[x;\sigma,\delta]$, where $\field{D}$ is a skew field.  

\subsubsection{Euclidean Pseudo-Division.}

A (non-commutative) domain $A$ is a \emph{left Ore domain} if $Aa \cap Ab \neq 0$ for all nonzero $a, b \in A$. Left Ore domains have a left ring of fractions  which is  a skew field (\cite[Example II.3]{Stenstrom:1975}). Every left noetherian domain is left Ore \cite[Proposition II.1.7]{Stenstrom:1975}. The fundamental property of Ore extensions of skew fields is that they have a left Euclidean Division Algorithm. As in the commutative case, it may be deduced from the following left Pseudo-Division. We include its easy proof because we find interesting to see that the Ore condition is already needed.  
 
\begin{proposition}{[}Left Pseudo-Division{]}\label{Pseudodivision} Let
$A$ be a left Ore domain and $f,g\in A[ x ;\sigma,\delta]$. If $g\neq0$,
and $\sigma$ is injective, then there exist a nonzero element $a\in A$
and polynomials $q,r\in A[ x ;\sigma,\delta]$ such that $af=qg+r$
and $\deg(r)<\deg(g)$. \end{proposition} \begin{proof} If $\deg(f)<\deg(g)$,
then put $a=1,q=0,r=f$. So, let us assume $\deg(f)\geq\deg(g)$,
and write $m=\deg(f)-\deg(g)$. We will prove that there exist $a_{1},b_{1}\in A \setminus\{0\}$
such that 
\begin{equation}
\deg(a_{1}f-b_{1} x ^{m}g)<\deg(f)\label{psdiv}\,.
\end{equation}
 Then the result follows by induction on $\deg(f)$. To prove \eqref{psdiv},
write 
\[
f=\lc{f} x ^{\deg(f)}+\underline{f},\quad g=\lc{g} x ^{\deg(g)}+\underline{g}
\]
 with $\deg(\underline{f})<\deg(f)$ and $\deg(\underline{g})<\deg(g)$.
By the  left  Ore condition there exist nonzero elements $a_{1},b_{1}\in A$
such that $a_{1}\lc{f}=b_{1}\sigma^{m}(\lc{g})$. On the other hand,
$ x ^{m}\lc{g}=\sigma^{m}(\lc{g}) x ^{m}+h$ with $\deg(h)<m$. Therefore,
\begin{multline*}
a_{1}f-b_{1} x ^{m}g=a_{1}\lc{f} x ^{\deg(f)}-b_{1} x ^{m}\lc{g} x ^{\deg(g)}+a_{1}\underline{f}-b_{1}\underline{g}= \\ a_{1} \underline{f} -b_{1}h-b_{1}g,
\end{multline*}
 and this last polynomial has degree strictly less than $\deg(f)$.
\end{proof}

\begin{remark} We have seen in the proof of Proposition \ref{Pseudodivision}
that in order to compute a  left  pseudo-division in $A[ x ;\sigma,\delta]$
we need to make effective the left Ore  condition   on $A$. This means that given
$c,d\in A$ we should be able to compute nonzero $a,b\in  A $ such
that $ac-bd=0$. This is equivalent to the computation of some nontrivial
element in the kernel of the left $ A$-linear map 
\[
 \varphi:A^{2}\longrightarrow A,\qquad\varphi(a,b)=ac-bd \,.
\]
Thus, we need an algorithm to compute  a nonzero element in  the syzygy module $\Syz\begin{pmatrix}c \\ -d \end{pmatrix}$.  If such an algorithm is available, then we  obtain Algorithm \ref{pseudodivalg} for the computation of the pseudo-division. 
\end{remark}

\begin{algorithm}
\caption{Left Pseudo-Division}\label{pseudodivalg}
\begin{algorithmic}
\REQUIRE$f, g \in A[ x ;\sigma,\delta]$ with $g \neq 0$. 
\ENSURE $0 \neq a \in A, q, r \in A[ x ;\sigma,\delta]$ such that $af = qg + r$ and $\deg (r) < \deg (g)$. 
\INIT  a:= 1, q:= 0, r:= f
\WHILE{$\deg (g) \leqslant \deg (r)$}
\STATE with $ (0,0) \neq  (a_1,b_1) \in \Syz \begin{pmatrix}\lc {r}  \\ - \sigma^{\deg (r) - \deg (g)}(\lc {g}) \end{pmatrix}$,
\STATE $a := a_1 a$,  $q:= q + b_1 x ^{\deg (r) - \deg (g)}$, $r:= a_1r - b_1  x ^{\deg (r) - \deg (g)}g$ \ENDWHILE
\end{algorithmic}
\end{algorithm}

\begin{example}
The  Left  Pseudo-Division  Algorithm  applied  to $f =  x ^3 - t x  + 1, g = t x  - 1 \in \mathbb{C}[t][ x ;d/dt]$ gives
the exact division $t^2f = ( t   x ^2- x -t^2)g$. 
\end{example}

The  Left   Euclidean Division Algorithm for polynomials in $\field{D}[ x ;\sigma,\delta]$, where $\field{D}$ is a skew field, is easily deduced from Proposition \ref{Pseudodivision}.   Explicitly,   for any $f, g \in \field{D}[ x ;\sigma,\delta]$ with $g \neq 0$, there exist uniquely determined polynomials $q, r \in \field{D}[ x ;\sigma,\delta]$ such that $f =qg + r$ and $\deg(r) < \deg(g)$. We will use the notation $r = \lrem(f,g)$ for the \emph{left remainder} $r$ of the division.  Note that $\sigma$ is necessarily injective, since it is an endomorphism of the skew field $\field{D}$. 

A remarkable consequence of the Left Euclidean Division  Algorithm  is that $R = \field{D}[ x ;\sigma,\delta]$ is a left principal ideal domain (left PID), that is, every left ideal of $R$ is principal. Therefore,  if $f, g \in R$, then $Rf + Rg = Rd$, where $d$ is the \emph{right greatest common divisor} of $f$ and $g$, determined up to multiplication on the left by a nonzero element of $\field{D}$. We will use the notation $d = \rgcd(f,g)$. Analogously, the \emph{left least common multiple} of $f, g$ as a polynomial $m \in R$ such that $Rf \cap Rg = Rm$ (notation $m = \llcm(f,g))$. Both $\rgcd(f,g)$ and $\llcm(f,g)$ may be computed by the corresponding non-commutative version of the Extended Euclidean Algorithm (see, e.g.,  \cite[Section I.4]{Bueso/Gomez/Verschoren:2003}).

\subsection{Jordan-H\"older Theorem and Factorization}
We will derive that $\field{D}[ x ;\sigma,\delta]$ is a (non-commutative) unique factorization domain, for $ \field{D}$ a skew field, $\sigma : \field{D} \to \field{D}$ any ring endomorphism, and $\delta : \field{D} \to \field{D}$ a $\sigma$-derivation, from the Jordan-H\"older theorem for modules of finite length. To this end, let $R$ denote any ring, and let us first recall that a left $R$-module $M$ is \emph{simple} if $M \neq \{ 0 \}$ and the only submodules of $M$ are $M$ and $\{ 0 \}$. A left $R$-module $M$ has \emph{finite length}  if there exist a chain of submodules 
\begin{equation}\label{compserie}
\{ 0 \} = M_0 \subset M_1 \subset \cdots \subset M_n = M
\end{equation}
 such that $M_{i}/M_{i-1}$ is simple for every $i= 1, \dots, n$. The sequence \eqref{compserie} is then called \emph{a composition series} of $M$, and the simple factors $M_{i}/M_{i-1}$ are the \emph{composition factors} of the series \eqref{compserie}. Jordan-H\"older Theorem is a standard result in any basic course on Module Theory (see \cite{Anderson/Fuller:1992} for a detailed proof), and asserts that if 
 \[
 \{ 0 \} = N_0 \subset N_1 \subset \cdots \subset N_p = M
 \] 
is another composition series of $M$, then $p = n$ and there exists a permutation $\pi : \{1, \dots, n \} \to \{1, \dots, n \}$ such that $M_{i}/M_{i-1} \cong N_{\pi (i)}/N_{\pi (i) -1}$ for every $i = 1, \dots, n$. Thus, the composition factors are unique up to reordering and isomorphisms, and  they  are called composition factors of the module $M$. The number $n$ is the \emph{length} of $M$.  

It is possible to derive from Jordan-H\"older Theorem a factorization theorem over (two-sided) PID's (see \cite[Theorem 1.2.9]{Jacobson:1996}). Let us illustrate how to apply this idea to $\field{D}[x;\sigma,\delta]$, which is only a left PID (unless $\sigma$ is an automorphism, see \cite[Proposition 1.1.14]{Jacobson:1996}). 

A polynomial  of positive degree  $f \in \field{D}[  x   ; \sigma,\delta]$ is called \emph{irreducible} if for any factorization $f = ab$, then either $a \in \field{D}$ or $b \in \field{D}$. Write $R =\field{D}[  x  ; \sigma,\delta]$. If $f, g \in R$, then $Rf \subseteq Rg$ if and only if $f = f'g$ for some $f' \in R$.  The inclusion is strict if and only if $f' \notin \field{D}$.   This immediately gives that $f$ is irreducible if and only if $Rf$ is a maximal left ideal of $R$ or, equivalently, $R/Rf$ is simple as  a  left $R$-module. Before  stating   the factorization theorem in $\field{D}[  x  ; \sigma,\delta]$, let us give an example that prevents the reader against any  naive   approach to the uniqueness of non-commutative factorizations.

\begin{example}\label{ejemplofactor}
Let $R=\mathbb{C}(t)[ x ;d/dt]$ the differential polynomial ring over the field $\mathbb{C}(t)$ of complex rational functions, where $d/dt$ denotes the usual derivation of polynomials in the variable $t$. There exist infinitely many different factorizations of $ x ^2$ into monic irreducible factors,
\begin{equation}\label{muchasfact}
 x ^2 = \left( x + \frac{1}{t+z}\right)\left( x  - \frac{1}{t+z}\right), \qquad (z \in \mathbb{C})\,.
\end{equation} 
 Thus, the uniqueness in the factorization of Ore polynomials can not be understood in the same sense as it is done in a commutative setting, since none of the monic polynomials $f_z =  x  + 1/(t+z)$ and $g_z =  x  - 1/(t+z)$ differ from $ x $ by multiplication by a unit of $R$. 

Nevertheless, observe that we have isomorphisms of left $R$-modules
\[
\xymatrix@R=5pt{
R/R x  \ar[r] & R/Rf_z, &   r(t) + R x  \ar@{|->}[r] & (t+z)r(t) + Rf_z \\
R/R x  \ar[r] & R/Rg_z, &   r(t) + R x   \ar@{|->}[r]  & r(t)/(t+z) + Rg_z }
\]
This property leads to the following definition.
\end{example}

\begin{definition}
\cite[Chapter 3]{Jacobson:1943} We say that $f, g \in R$ are \emph{similar},  $f \sim g$, if there is an isomorphism of left $R$-modules $R/Rf \cong R/Rg$.
\end{definition}

\begin{theorem}\label{JHfactorization}
Every polynomial $f \in \field{D}[x;\sigma,\delta]$ of positive degree factorizes as $f = f_1 \cdots f_t$, where $f_i \in \field{D}[x;\sigma,\delta]$ is irreducible for every $i=1, \dots, t$. If $f = g_1 \cdots g_s$ is any other such a factorization of $f$, then $s = t$ and there exists a permutation $\pi : \{1, \dots, t \} \to \{ 1, \dots, t \}$ such that $f_i \sim g_{\pi(i)}$ for every $i = 1, \dots, t$. 
\end{theorem}
\begin{proof}
Write $R = \field{D}[x;\sigma,\delta]$. By the Left Division Algorithm,  the left $R$-module $R/Rf$ is a left $\field{D}$-vector space of dimension $\deg f$. This clearly implies that $R/Rf$ is a left $R$-module of finite length. We induct on this length $t$. If $t = 1$, then $R/Rf$ is simple and $f$ is irreducible. If $t >1$, then there exits a left maximal ideal $Rf_t$ of $R$ such that $Rf \subset Rf_t \subset R$.  We know that $f = f'f_t$ for some $f' \in R$. Since $Rf_t/Rf = Rf_t/Rf'f_t \cong R/Rf'$, we deduce from the Jordan-H\"older Theorem that the length of $R/Rf'$ equals $t-1$. By induction hypothesis, $f' = f_1 \cdots f_{t-1}$ for some irreducible  polynomials  $f_1, \dots, f_{t-1} \in R$. Therefore, $f = f_1 \cdots f_{t-1} f_t$ and $f_t$ is irreducible since $R/Rf_t$ is simple. This proves the existence of the factorization. As for the uniqueness concerns, the factorizations lead to the composition series for $M = R/Rf$ given by  $M_{i-1} = Rf_i \cdots f_tf_{t+1}/Rf$, and $N_{j-1} = Rg_j \cdots g_sg_{s+1}/Rf$, for $i =1, \dots, t$, $j = 1, \dots, s$, where $f_{t+1} = g_{s+1} = 1$.  The corresponding composition factors are then
\[
M_{i}/M_{i-1} \cong Rf_{i+1}\cdots f_tf_{t+1}/Rf_{i} \cdots f_tf_{t+1} \cong R/Rf_{i},
\]
and 
\[
N_{j}/N_{j-1} \cong Rg_{j+1}\cdots g_sg_{s+1}/Rg_{j} \cdots g_sg_{s+1} \cong R/Rg_{j}\,.
\]
The uniqueness is  directly deduced from the Jordan-H\"older Theorem. 
\end{proof}

\begin{remark}\label{compseriesfactor}
A careful reading of the proof of Theorem \ref{JHfactorization} shows that, for a given monic polynomial $f \in \field{D}[ x ;\sigma,\delta]$ of positive degree, there exists a bijection between the set of composition series of the left $R$-module $R/Rf$ and the set of factorizations of $f$ as a product of irreducible monic polynomials.  
\end{remark}

\begin{remark}
 Since $R = \field{D}[x;\sigma,\delta]$ is a domain,  similarity of polynomials in $R$ is independent on the side. Concretely, given $f, g \in R$, then $R/Rf \cong R/Rg$  as left $R$-modules if and only if $R/fR \cong R/gR$ as right $R$-modules  (see \cite[Lemma 4.11, Ch. 1]{Bueso/Gomez/Verschoren:2003}).
\end{remark}

\subsubsection{Similarity: a Computational Problem.}
From the computational point of view, the following problem arises after Theorem \ref{JHfactorization}. 
\begin{problem}\label{problemasimilar}
How to decide whether a given pair of polynomials $f, g \in R  = \field{D}[x;\sigma,\delta]$ are similar? 
\end{problem}
Of course, we may assume that $f$ and $g$ are monic polynomials of the same degree. Thus, in the commutative case ($\field{D}$ commutative, $\sigma = id_\field{D}$, $\delta = 0$), Problem \ref{problemasimilar} is just to decide whether $f = g$. In the general non-commutative case, things are different. For instance, in the linear case,  $ x -a \sim  x  -b$ if and only if there exists $c \in \field{D}\setminus \{0\}$ such that $b = \sigma(c)ac^{-1} + \delta(c)c^{-1}$. More generally, for polynomials $f, g$ of degree $n\geq 1$ with companion matrices $A$ and $B$, respectively, we have \cite[Proposition 2.4]{Leroy:1995}:
\begin{equation}\label{conjugadas}
f \sim g \iff \exists \; C \in \GL_n( \field{D}) \hbox{ such that } B = \sigma(C)AC^{-1} + \delta(C)C^{-1}\,.
\end{equation}
Similarity of matrices $A, B$ in the sense of \eqref{conjugadas} may be reduced to the corresponding commutative problem in some particular cases, and, henceforth, to the computation of canonical forms of matrices over commutative fields. For instance, N. Jacoboson made such a reduction when $\field{D}$ is finite dimensional over its center  for $\delta = 0$ and $\sigma^m$ an inner automorphism for some power of $\sigma$ (see \cite[Ch. 4, Theorem 4.34]{Jacobson:1943}). The diagonalization of algebraic pseudo-linear transformations, and their associated matrices, is discussed in \cite{Leroy:1995}. These mathematical results deserve, in our opinion, a computational approach. In general, the definition and computation of (e.g. rational) canonical forms of matrices with respect to condition \eqref{conjugadas} seems to be an open problem, tightly related to the election of a canonical representative of the isomorphism class of a module of the form $R/Rf$. 

\subsection{Eigenrings and Factorization}
The second obvious problem from the effective perspective is

\begin{problem}\label{factorizacion}
How to compute a factorization of a given polynomial $f \in R  = \field{D}[x;\sigma,\delta]$? How to test whether $f$ is irreducible?
\end{problem}

Of course, this would depend heavily on the division ring $ \field{D}$.  When $\field{D} = \mathbb{F}$ is a finite field, Problem \ref{factorizacion} has been addressed in \cite{Giesbrecht:1998} for $R = \mathbb{F}[ x ;\sigma]$,  and for $R = \mathbb{F}(t)[ x ;\sigma,\delta]$ in \cite{Giesbrecht/Zhang:2003}. The factorization of differential operators (this is to mean, polynomials over $\mathbb{K}[ x ;\delta]$ for a computable differential field $\mathbb{K}$ with derivation $\delta$) has a long tradition (see the references of \cite{Giesbrecht/Zhang:2003}, \cite{VanDerPut/Singer:2003} on this topic).  However, this problem seems not to be completely  solved even for $\mathbb{K} = \mathbb{Q}(t)$  (the factorization in $\overline{\mathbb{Q}}(t)[x;\delta]$, where $\overline{\mathbb{Q}}$ denotes the algebraic closure of $\mathbb{Q}$, is addressed in \cite{VanHoeij:1997}).

Let us explain how some basic module theory may help to find partial solutions to Problem \ref{factorizacion}.  

Let $f \in R = \field{D}[ x ;\sigma,\delta]$  be  a polynomial of positive degree. We have seen that the factorizations of $f$ are encoded in the structure of the lattice of submodules of the left $R$-module $R/Rf$. Some information on this lattice can be extracted from the ring $\lend{R}{R/Rf}$ of all  left $R$-module endomorphisms of $R/Rf$.  For instance, if $R/Rf$ is simple, then (Schur Lemma) $\lend{R}{R/Rf}$ is a skew field, because every nonzero endomorphism of $R/Rf$ has to be surjective with zero kernel. Unfortunately, the converse is, in general, false (see Example \ref{indescomponible} below).

However, there are examples of Ore extensions $R = \field{D}[ x ;\sigma,\delta]$ for which a left $R$-module $M$ is simple if and only if $\lend{R}{M}$ is a skew field. Remarkably, this is the case  when  $\field{D} = \mathbb{F}$ is a finite field, $\delta = 0$ and, henceforth, $\sigma$ is an autormorphism of $\mathbb{F}$. M. Giesbrecht developed in \cite{Giesbrecht:1998} a factorization algorithm for polynomials in $\mathbb{F}[ x ;\sigma]$ ultimately based in the fact that $f \in \mathbb{F}[ x ;\sigma]$ is irreducible if and only if $\lend{\mathbb{F}[ x ;\sigma]}{\mathbb{F}[ x ;\sigma]/\mathbb{F}[ x ;\sigma]f}$ is a field (see \cite[Theorem 3.3]{Giesbrecht:1998}). Moreover, when $\lend{\mathbb{F}[ x ;\sigma]}{\mathbb{F}[ x ;\sigma]/\mathbb{F}[ x ;\sigma]f}$ is not a skew field, then it contains some zero divisor that serves to find a factorization of $f$.

Some algorithms that use the eigenring for factoring pseudo-linear operators in the context of differential and difference equations were implemented in \textsc{Maple} in a package called \textsc{ISOLDE} (see \cite{ISOLDE}). These algorithms are described in \cite{Barkatou:1999,Barkatou:2007,Barkatou/Pflugel:1998}. 

\subsubsection{Fitting's Lemma and Zero Divisors.}%
Let $R$ be any ring. A left  $R$-module $M$ is said to be \emph{indecomposable}  if  no decomposition $M = N\oplus L$ with $N, L$ nonzero submodules is possible. Otherwise, we say that $M$ is decomposable. The endomorphism ring $\lend{R}{M}$ of a decomposable module $M$ has always zero divisors: from a nontrivial decomposition $M = N \oplus L$ we get that the endomorphism $e : M \to M$ defined by $e(n + l) = n$, for $n \in N, l \in L$ is an \emph{idempotent} of the ring $\lend{R}{M}$ (i.e., $e^2 = e$) different from $0$ and $1$, and, hence, the obvious equality $e(1-e) = 0$ gives that $e$ is a nontrivial zero divisor.  

Fitting's  Lemma  (see \cite[Proposition 3.1.8]{Hazewinkel/alt:2004}) says that if $h : M \to M$ is an endomorphism of a left $R$-module of finite length, then $M = \ker h^n \oplus \Ima h^n$ for some positive integer $n$. As a consequence, if $M$ is indecomposable of finite length and $h \neq 0$, then either $h$ is an automorphism or $h^n = 0$ for some $n >0$ (and, hence, $h$ is a zero divisor of $\lend{R}{M}$).  We so far deduce

\begin{lemma}
If $M$ is a left module of finite length over a ring $R$, then $\lend{R}{M}$ is a skew field if and only if $\lend{R}{M}$ has no nontrivial zero divisors.
\end{lemma}

The previous discussion also suggests what kind of non simple module of finite length $M$ could satisfy that $\lend{R}{M}$ is a skew field. The simplest situation is to assume $M$ indecomposable of length $2$, with a composition series with non isomorphic composition factors. The fact that $\lend{R}{M}$ is then a skew field is obvious after some training in abstract module theory. We include a detailed reasoning for the general reader: In fact, such a module contains a unique simple submodule $S_1$ (because $M$ is indecomposable of length $2$).  Moreover, $M/S_1$ is simple and  not  isomorphic to $S_1$. Now, if $h : M \to M$ is a nonzero endomorphism, then either $\ker h = 0$, and thus $h$ must be an automorphism, or $\ker h$ is a simple submodule of $M$ (because $M$ has length $2$). In the second case, $\ker h = S_1$ and, by Noether's first isomorphism theorem, $M/ \ker h$ is isomorphic to a simple submodule of $M$, which must be $S_1$. Thus the second option is excluded by the structure of the composition series of the module $ M$, and, hence, $h$ must be an automorphism. The following example illustrates the  situation just  described. 

\begin{example}\label{indescomponible}
Let $R = \mathbb{Q}(t)[ x ;d/dt]$ be the differential operator ring associated to the differential field $(\mathbb{Q}(t), d/dt)$, and consider the reducible polynomial 
\begin{equation}\label{factorizacionunica}
f =  x ^2 -t x  - 1 =  x ( x -t) \in R \,.
\end{equation} 
Let us see that \eqref{factorizacionunica} gives the unique factorization of $f$ with monic irreducible factors. By Remark \ref{compseriesfactor}, this will imply that the left $R$-module $R/Rf$ has a unique composition series with composition factors isomorphic to $R/R( x -t)$ and $R/R x $, respectively. 

If $ x ^2 -t x  - 1 = ( x -v)( x -u)$, with $v, u \in \mathbb{Q}(t)$, then we get, equating coefficients, that $u$ satisfies the Riccati equation $u' = -u^2 + tu +1$. Solving it, we see that the only rational solution is $u = t$. This implies the uniqueness of the factorization \eqref{factorizacionunica}. 

Finally, if $h : R/R x  \to R/R( x -1)$ is any homomorphism of left $R$-modules, then $h(1 + R x ) = q + R( x -1)$ for some $q \in \mathbb{Q}(t)$ such that $ x  q \in R( x -1)$. A straightforward argument shows that $q=0$ and, thus, $h = 0$. Therefore, the composition factors of the indecomposable module $R/Rf$ are not isomorphic. As a consequence, $f$ is not irreducible in $\mathbb{Q}(t)[ x ;d/dt]$ but  $\lend{R}{R/Rf}$ is a skew field.  
\end{example}

\subsubsection{Bounded Indecomposable Polynomials and Factorization.} Chapter 3 of \cite{Jacobson:1943} is concerned with the arithmetic and the structure of finitely generated modules over a left and right principal ideal domain. This is the case of $R = \field{D}[ x ;\sigma,\delta]$, whenever we  assume that $\sigma$ is an automorphism. A polynomial $f \in R$ is  called \emph{indecomposable} if $R/Rf$ is indecomposable as a left $R$-module, and $f$ is said to be \emph{bounded} if $\Ann_R(R/Rf) \neq \{0 \}$.  If $f$ is bounded, then  there exists a polynomial $0 \neq f^* \in R$ such that $Rf^* = f^*R$ is the largest two-sided ideal of $R$ contained in $Rf$. This polynomial $f^*$ is determined by $f$ up to multiplication (e.g. on the left) by a nonzero element of $\field{D}$, and it is called \emph{the bound} of $f$. By \cite[Theorem 11, Ch. 3]{Jacobson:1943}, the bound $Rf^* = f^*R$ is also the largest two sided ideal of $R$ contained in $fR$. 

It follows from the theory developed in paragraphs 8 and 9 of Chapter 3 of \cite{Jacobson:1943} (specially, theorems 20, 21 and 24) that if $f$ is bounded and $R/Rf$ is indecomposable as a left $R$-module, then $R/Rf$  has a unique composition series, and all its composition factors are isomorphic. 

\begin{lemma}
Let $R = \field{D}[ x ;\sigma,\delta]$, and assume that $\sigma$ is an automorphism of the skew field $\field{D}$. If $f \in R$ is a bounded polynomial of positive degree, then $f$ is irreducible if and only if $\lend{R}{R/Rf}$ is a skew field. 
\end{lemma}
\begin{proof}
Let us prove that if the length of $M = R/Rf$ is at least $2$, then $\lend{R}{M}$ has nontrivial zero divisors. We know from the previous paragraph that this is already the case when $M$ is decomposable, so we assume that $M$ is indecomposable, and let $\{ 0 \} = M_{0} \subset M_1 \subset \cdots \subset M_{n} = M$ be its unique composition series. Then $M/M_{n-1}$ is simple and isomorphic to $M_1$. Let $h : M \to M$ be the composition of the canonical projection $M \to M/M_{n-1}$ followed by $M/M_{n-1} \cong M_1 \subseteq M$. Certainly, $h \neq 0$, but $h^2 = 0$, which finishes the proof.
\end{proof}

\subsubsection{The Eigenring.}
Given $h \in \lend{R}{R/Rf}$,   where $R = \field{D}[ x ;\sigma,\delta]$, 
\begin{equation}\label{endo}
h(1 + Rf) = q + Rf
\end{equation}
 for some $q \in R$ such that $\deg{q} < \deg{f}$. Obviously, $q$ determines $h$, since $h(u + Rf) = uq+Rf$ for any $u + Rf \in R/Rf$. On the other hand, $q$ cannot be arbitrary. In fact, $q$ will define an endomorphism of $R$-modules via \eqref{endo} if and only if $fq \in Rf$ or, equivalently, $\lrem(fq,f) = 0$. In resume, \eqref{endo} provides a bijection between $\lend{R}{R/Rf}$ and 
 \begin{equation}\label{eigenring}
 \mathfrak{E}(f) = \{ q \in R : \deg{q} < \deg{f} \text{ and } \lrem(f q ,f) = 0 \}\,.
 \end{equation}
 A straightforward computation shows that this bijection is a homomorphism of additive groups and that the product of $\lend{R}{R/Rf}$ is transferred to the product $
 q' \cdot q = \lrem(q'q,f)$ in $\mathfrak{E}(f)$,
 which makes it a ring isomorphic to $\lend{R}{R/Rf}$ (recall that the product in this endomorphism ring is the opposite of the composition of maps).  Any (non trivial) zero divisor $q \in \mathfrak{E}(f)$ gives  the  proper factor $\rgcd(q,f)$ of $f$, for if $\rgcd(q,f) = 1$, then the endomorphism $h$ defined by \eqref{endo} is surjective and, since $R/Rf$ is of finite length, an isomorphism. 
 
 \begin{proposition}\label{eigenfactor}
 Given $ f \in R = \field{D}[ x ;\sigma,\delta]$, each nontrivial zero divisor $q$ of $f$ in $\mathfrak{E}(f)$ gives a proper factor $\rgcd(q,f)$ of $f$. Moreover, if $f$ is bounded,  and $\sigma$ is an automorphism,  then $f$ is irreducible if and only if $\mathfrak{E}(f)$ has no nontrivial zero divisors. 
 \end{proposition}
 
Proposition \ref{eigenfactor} reduces the problem of finding proper factors, if possible, of a bounded polynomial $f \in R = \field{D}[ x ;\sigma,\delta]$, to the search of zero divisors of the eigenring $\mathfrak{E}(f)$.  Even if $f$ is not bounded, if the eigenring is presented as a finite-dimensional algebra over some subfield $K$, then the  idea  of finding zero divisors in $\mathfrak{E}(f)$ may be used to obtain factorizations of $f$ (even complete factorizations, e. g., if $R/Rf$ is semisimple as  a left $R$-module, that is, it is a sum of simple modules). This is the case of the method used in \cite{Singer:1996} for linear differential operators. 

The ``opposite'' situation, where all polynomials are bounded, embodies the factorization method based also in the search of zero divisors of $\mathfrak{E}(f)$ presented in \cite{Giesbrecht:1998} for $f \in \mathbb{F}[ x ; \sigma]$, where $\mathbb{F}$ a finite field. In this example, and in many others, there exists a subfield $\field{k}$ of the center of $\field{D}$ such that the eigenring $\mathfrak{E}(f)$ of every polynomial $f \in R = \field{D}[ x ;\sigma,\delta]$ is a finite dimensional algebra over $\field{k}$ (in the case of $\mathbb{F}[ x ;\sigma]$, $\field{k}$ may be chosen as the invariant subfield of $\mathbb{F}$ under $\sigma$). Under these circumstances, $\mathfrak{E}(f)$ may be embedded as a subalgebra of a full matrix algebra over $\field{k}$ and try to use Linear Algebra methods to find non zero divisors in $\mathfrak{E}(f)$, or to prove that they do no exist. For $f \in \mathbb{F}[ x ;\sigma]$, efficient  algorithms do exist \cite{Ronyai:1987} and \cite{Giesbrecht:1998}. 

Recently, some alternative methods of factorization have been investigated. For the specific case $R = \mathbb{F}[ x ;\sigma]$, there is the approach \cite{Caruso/LeBorgne} based on the use of non trivial properties of Azumaya algebras over finite fields. An alternative method, combining the use of elementary properties of the modules over the Artinian ring $R/Rf^*$ with the search of zero-divisors in eigenrings of semisimple modules, is proposed in \cite{Gomez/alt:unp}. The latter may be applied to more general cases of Ore extensions  that  are finitely generated as modules  over their centers. 

The aforementioned algorithms give \emph{one} factorization of a given polynomial $f \in R  = \field{D}[ x ;\sigma,\delta]$. In view of Example \ref{ejemplofactor}, the following question makes perfectly sense.

\begin{problem}
Given a factorization of $f \in R$ as a product of irreducible polynomial,  is it  possible to compute (or describe as explicitly as possible) all the other factorizations of $f$?
\end{problem}

\subsection{Matrices and Structure of Modules}
Let $ \field{D}$ be a skew field, $\sigma :  \field{D}  \to  \field{D} $ a ring endomorphism, and $\delta :  \field{D}  \to  \field{D}$ a $\sigma$-derivation, and consider $R = \field{D}[x;\sigma,\delta]$ the associated Ore extension of $ \field{D}$. As we have seen, a  Left Division Algorithm  is available, and this implies that $R$ is a left principal ideal domain (left  PID,   for short). This division algorithm may be used for the computation of the matrix $\mathsf{syz}(A_{\psi})$ associated to a homomorphism of left  $R$- modules $\psi : \mathbf{F}_t \to \mathbf{F}_m$  from $A_{\psi} = (a_{ij}) \in R^{t \times m}$.  Recall that $A_{\psi}$ denotes the matrix representing $\psi$ in coordinates with respect to some fixed bases in the free left $R$-modules $\mathbf{F}_t$ and $\mathbf{F}_m$).  As in the commutative case, the key is the existence of an equivalence between elementary operations on the rows of a matrix and left multiplication by elementary matrices. 

To this end, given $i, j \in \{ 1, \dots, t \}$, let $E_{ij}$  denote  the matrix obtained from the identity matrix $I_{ t} \in R^{t \times t}$ by interchanging the  $i$-th and $j$-th rows.   We know that $E_{ij}A_{\psi}$ is the resulting matrix of interchanging the rows $i$-th and $j$-th in $A_{\psi}$. Analogously, given $q \in R$, let $E_{i+qj}$ denote the matrix obtained from $ I_t$ by adding to the $i$-th row the result of multiplying the $j$-th row by $q$ on the left. The matrix $E_{i + qj}A_{\psi}$ is obtained from $A_{\psi}$ in the same way. It is obvious that $E_{ij}^2 =  I_t  = E_{i + qj}E_{i - qj}$, henceforth these matrices, called \emph{row elementary matrices}, are invertible.  

Now, it is clear that, due to the  Left Division Algorithm,   we can make a finite sequence of elementary row operations on $A_{\psi}$ to obtain a matrix $ A_{\hat{\psi}}  \in R^{t \times m}$ of the form
\begin{equation}\label{smith}
A_{\hat{\psi}} = \begin{pmatrix} b_{11} & b_{21} & \cdots & b_{1r} & \cdots & b_{1m} \\ 
0 & b_{22}& \cdots & b_{2r} & \cdots &  b_{2m} \\
 \vdots & & & \vdots & & \vdots \\
 0 & 0 & \cdots & b_{rr} & \cdots & b_{rm} \\
 0 & 0 & \cdots & 0& \cdots & 0 \\
  \vdots & \vdots & & \vdots & & \vdots \\ 
  0 & 0 & \cdots & 0& \cdots & 0  \end{pmatrix},
\end{equation}
where $r \leqslant t, m$, and the first $r$ rows of $B$ are nonzero, but $b_{ji} = 0$ if $j > i$. If $\hat{\psi} : \mathbf{F}_t \to \mathbf{F}_m$ is defined by $A_{\hat{\psi}}$, then $\ker \hat{\psi}$ is generated by the last $t-r$ vectors of the basis of $\mathbf{F}_t$, that is, the matrix $ \mathsf{syz}(A_{\hat{\psi}})$ in \eqref{syzcomp} is 
\[
 \mathsf{syz}(A_{\hat{\psi}}) = \begin{pmatrix} 0_{ (t-r) \times r} & \vline & I_{t-r}  \end{pmatrix}.
\]
Let $P \in R^{t \times t}$ be the invertible matrix obtained as a product of elementary matrices such that $ A_{\hat{\psi}}  = PA_{\psi}$, and $\hat{P}$ its inverse. If we define $p : \mathbf{F}_t \to \mathbf{F}_t$ and $\hat{p} : \mathbf{F}_t \to \mathbf{F}_t$ by the conditions $A_p = P$, $A_{\hat{p}} = \hat{P}$, then we  may  compute the matrix $  \mathsf{syz}(A_{\psi}) \in R^{(2t-r) \times t}$ according to \eqref{syzcomp}  resulting that its first $t-r$ rows are the last $t-r$ rows of $P$, and its last $t$ rows are zero.  In resume, the kernel of $\psi$ is generated by last  $t-r$ rows of $P$. Since $P$ is computed from $I_t$ by the same sequence of row elementary operations  used to get $A_{\hat{\psi}}$  from $A_{\psi}$, we  may  resume the information so far obtained in the following theorem.

\begin{theorem}\label{leftPID}
Given $\psi : \mathbf{F}_t \to \mathbf{F}_m$, a homomorphism of free left modules over $R = \field{D}[ x ;\sigma,\delta]$,  there is an algorithm that computes $P \in \GL_t(R)$ such that 
\[
P A_{\psi}  = A_{\hat{\psi}}\,,
\]
where $A_{\hat{\psi}}$ is of the form displayed in \eqref{smith}. Moreover, a basis for the kernel of $\psi$ is given by the last $t-r$ rows of $P$, while a basis of the image of $\psi$ is obtained from the first $r$ rows of $A_{\hat{\psi}}$. 
\end{theorem}

Observe that Theorem \ref{leftPID} says in particular that both the kernel and the image of the homomorphism $\psi : \mathbf{F}_t \to \mathbf{F}_m$ are finitely generated free left $R$-modules. Le us record this relevant fact.

\begin{corollary}
Let $R =  \field{D}[ x ;\sigma,\delta]$ be an Ore extensions of a skew field $ \field{D}$. Every left $R$-submodule of a finitely generated free left $R$-module is (finitely generated) free. 
\end{corollary}
\begin{proof}
Given a submodule $L$ of the free left $R$-module $\mathbf{F}_m$, it suffices to apply Theorem \ref{leftPID} with $\psi$ equal to the inclusion map from $L$ to $\mathbf{F}_m$. 
\end{proof} 

\begin{example}\label{ann1}
Let $f \in R  = \field{D}[ x ;\sigma,\delta]$ be an Ore polynomial, and $h : R \to R/Rf$ a homomorphism of left $R$-modules. Then $h(1) = q + Rf$ for some $q \in R$. Let us show how to apply Theorem \ref{leftPID} to this particular case for the computation of a presentation of $\Ima h$. Consider the sequence of remainders $r_{-1} = q$, $r_0 := -f$, $r_{i+1} =\lrem(r_{i-1},r_i)$, for $i > 0$, such that $r_n \neq 0$ and $r_{n+1} = 0$ for some $n \geq 0$. Write, for $i = 1, \dots, n$,  $r_{i-1} = c_ir_i + r_{i+1}$, where $c_i \in R$ is the quotient of the Left Euclidean Division of $r_{i-1}$ by $r_i$. We have 
\[
\begin{pmatrix} 0 & 1 \\
1 & - c_n \end{pmatrix} \begin{pmatrix} 0 & 1 \\
1 & - c_{n-1} \end{pmatrix} \cdots \begin{pmatrix} 0 & 1 \\
1 & - c_0 \end{pmatrix}\begin{pmatrix} q \\ -f \end{pmatrix} = \begin{pmatrix} r_n \\ 0 \end{pmatrix} .
\]
Thus, the $1 \times 1$-matrix given by Theorem \ref{leftPID} is the entry $c$ at position $(2,1)$ of the matrix
\[
P = \begin{pmatrix} 0 & 1 \\
1 & - c_n \end{pmatrix} \begin{pmatrix} 0 & 1 \\
1 & - c_{n-1} \end{pmatrix} \cdots \begin{pmatrix} 0 & 1 \\
1 & - c_0 \end{pmatrix} .\]

We get, therefore, the presentation $\Ima h \cong R/Rc$. Observe that we have $Rc = \ker (h) = \ann_R(q + Rf)$. 
\end{example}

\begin{example}
The idea developed in Example \ref{ann1} may be used for computing annihilators of elements in a general finitely presented left $R$-module. Consider a left $R$-module $M$ with a presentation $\xymatrix{\textbf{F}_{s} \ar^{\psi}[r] & \textbf{F}_t \ar^{\varphi}[r] & M \ar[r] & 0}$. Given $m \in M$, consider the morphism of left $R$-modules $\rho_m : R \to M$ defined as $\rho_m(r) = rm$, whose kernel is $\ann_R(m)$. To compute this annihilator, we write $m = \varphi(q)$, for some $q \in \mathbf{F}_t$. This vector $q$ also defines a morphism $q : R \to \mathbf{F}_t$ which encodes $\rho_m$.  
\end{example}

\subsection{The Structure of a Finitely Generated Module}
The structure of finitely generated modules over a (left and right) PID is given in \cite{Jacobson:1943}. In fact, Jacobson's approach is constructive, as it is based on the reduction of matrices to a diagonal form by means of elementary row and column operations. In this section, we assume that, in the Ore extension $\field{D}[ x ;\sigma,\delta]$, $\sigma$ is an automorphism of $\field{D}$. Therefore, $R$ has also a  Right  Division Algorithm (see Algorithm \ref{rightdivalg}). 

\begin{algorithm}
\caption{Right Euclidean Division}\label{rightdivalg}
\begin{algorithmic}
\REQUIRE $f, g \in \field{D}[ x ;\sigma,\delta]$ with $g \neq 0$.
\ENSURE $q, r \in \field{D}[ x ;\sigma,\delta]$ such that $f = gq + r$ and $\deg (r) < \deg (g)$. 
\INIT q:= 0, r:= f
\WHILE{$\deg (g) \leqslant \deg (r)$} 
\STATE with $a = \sigma^{-\deg(g)}(\lc{g}^{-1}\lc{r})$,
\STATE $q:= q + a x ^{\deg (r) - \deg (g)}$, $r:=r - g a  x ^{\deg (r) - \deg (g)}$ 
\ENDWHILE
\end{algorithmic}
\end{algorithm}

 In general, if $\sigma$ is an automorphism, then $A[x;\sigma,\delta]^{op} \cong A^{op}[x;\sigma^{-1},-\delta\sigma^{-1}]$, as rings (here $A$ is any ring). Thus, the Right Euclidean Algorithm on $ \field{D}[x;\sigma,\delta]$ is deduced from the Left Euclidean Algorithm of $\field{D}^{op}[x;\sigma^{-1},-\delta\sigma^{-1}]$, and vice-versa. This idea may be also used to change of side when dealing with modules.

\subsubsection{Diagonalization.}
Besides the elementary operations on the rows of a matrix with coefficients in $R$, we may use elementary operations on the columns. Thus,  if  $A \in R^{t \times s}$, then $AE_{ij}$ is the matrix obtained by interchanging the $i$-th and the $j$-th columns of $A$. Now, let $E_{i+jq}$ denote the matrix obtained from $ I_m$ by adding to the $i$-th column the result of multiplying the $j$-th column by $q$ on the right. The matrix $AE_{i + jq}$ is obtained from $A$ in the same way. If $R = \field{D}[ x ;\sigma,\delta]$ possesses also a  Right Division Algorithm  (i.e., $\sigma$ is an automorphism of $\field{D}$), then we may apply to any matrix $A \in R^{t \times m}$  suitable sequences of elementary row operations and elementary column operations in order to compute invertible matrices $P \in R^{t\times t}$ and $Q \in R^{m \times m}$ such that 
\[
PAQ = \begin{pmatrix} b_{1} & 0 \\
0 & A' \end{pmatrix}
\]
for some matrix $A' \in R^{(t-1) \times (m-1)}$. We obviously may obtain this diagonal form with $\deg b_1$ less or equal than the degrees of all nonzero entries of $A$. By induction, we obtain:

\begin{proposition}\label{diagonal}
Let $A \in R^{t \times m}$ be any matrix with coefficients in $R = \field{D}[ x ;\sigma,\delta]$. Assume that $\sigma$ is an automorphism. There exists an algorithm that computes $P \in \GL_t(R)$ and $Q \in \GL_m(R)$ such that 
\begin{equation}\label{eq:diagonal}
PAQ = \Delta : = \diag \{ b_{1}, \dots, b_{r}, 0, \dots, 0 \},
\end{equation}
that is, a diagonal matrix of size $t \times m$. 
\end{proposition}

Assume now that the matrix $A$ represents a left $R$-module $M$, that is, $A = A_{\psi}$ for a presentation  $\xymatrix{ R ^t \ar^{\psi}[r] &   R^m \ar^{\varphi}[r] & M  \ar[r] & 0}$. Let $P, Q$ invertible matrices such that $PA_{\psi}Q = \Delta$, where $\Delta = \diag \{b_{1}, \dots, b_{r}, 0, \dots, 0 \}$. Take morphisms in the following diagram such that $A_{p} =  P,   A_{q} = Q^{-1}, A_{\delta} = \Delta$, and $h : N \to M$  is defined by $q$ and $p$.  
\[
\xymatrix{ R^t \ar^{\delta}[r] \ar^{p}[d] &   R^m \ar^{\varphi'}[r] \ar^{q}[d] & N  \ar^{h}[d] \ar[r] & 0 \\
 R^t \ar^{\psi}[r] &   R^m \ar^{\varphi}[r] & M  \ar[r] & 0 \rlap{\,.}}
\]
 Since $p$ and $q$ are isomorphisms, we get that $h$  is an isomorphism.  Thus, $M \cong N \cong R^{m}/row(\Delta) \cong R^{m-r} \oplus R/Rb_{1} \oplus \cdots \oplus R/Rb_{r}$. 

\begin{theorem}\cite{Jacobson:1943}\label{torsiontorsionfree}
Every finitely generated left module $M$ over $R = \field{D}[ x ;\sigma,\delta]$ is a direct sum of finitely many cyclic modules. More precisely, $M$ is a direct sum of a free left $R$-module of finite rank and finitely many cyclic left $R$-modules of finite length. 
\end{theorem}

\subsubsection{Krull-Schmidt Theorem and Elementary Divisors.}
We deduce from Theorem \ref{torsiontorsionfree} that any finitely generated indecomposable left module of finite length over $R = \field{D}[ x ;\sigma,\delta]$ with $\sigma$ bijective, is isomorphic to $R/Rf$ for some  nonzero  $f \in R$, which is said to be \emph{indecomposable}. This, in conjunction with \emph{Krull-Schmidt Theorem}, lead to a complete classification of the finitely generated left $R$-modules of finite length. Recall (from \cite{Anderson/Fuller:1992}, for instance) that, for any ring $R$, if $M$ is a left $R$-module of finite length, and $M \cong  X_1 \oplus \cdots \oplus  X_s \cong Y_1 \oplus \cdots \oplus Y_t$ are two decompositions of $M$ as a direct sum of indecomposable modules, then $s = t$ and, after an eventual reordering, $X_i \cong Y_i$ for $i = 1, \dots, s$. As a consequence, $M$ is determined,  up to isomorphism , by finitely many indecomposable left $R$-modules of finite length. The classification of these indecomposable modules is one of the central problems in the Representation Theory of $R$. 

When applied to $R = \field{D}[ x ;\sigma,\delta]$,  with $\sigma$ bijective,   we get that a finitely generated left $R$-module of finite length $M$ is determined, up to isomorphisms, by a finite sequence $\{ e_1, \dots, e_s \}$ of indecomposable polynomials such that $M \cong R/Re_1 \oplus \cdots \oplus R/Re_s$. The polynomials $e_1, \dots, e_s$ are unique up to similarity and eventual reordering, and  they  are called the \emph{elementary divisors of $M$}.  

The computation of the elementary divisors of $M$ could be done from a presentation of $M$, by computing a diagonal form $\Delta$ as in \eqref{eq:diagonal}, and making the decomposition as a sum of indecomposable modules of each $R/Rb_{i}$. This last amounts to the computation of a decomposition of  $b_{i}$ as a  left  least common multiple of coprime indecomposable polynomials. Such a decomposition is described in the ``Third decomposition Theorem'' of \cite{Ore:1933}, and an algorithm for its computation in the case $R = \mathbb{F}[ x ;\sigma]$ was presented in \cite{Giesbrecht:1998}. 

As we have already seen, the elementary divisors of $M$ are an invariant, up to similarity, of the isomorphism class of $M$. Thus, in order to decide whether two given finitely generated left $R$-modules of finite length are isomorphic, we should compute their elementary divisors and compare if they give lists of similar polynomials. Even if both lists have been computed, which will depend greatly from the kind of Ore extension we are dealing with, a second problem arises: given indecomposable polynomials $f, g \in R$ of the same degree, how to decide if they are similar? When $f, g$ are indecomposable and bounded, then, by \cite[Theorem 20, Ch. 3]{Jacobson:1943}, $f$ and $g$ are similar if and only if  they have the same bound. General procedures for computing the bound of a polynomial in the case that $R = \field{D}[ x ;\sigma,\delta]$ is finitely generated as a module over its center are described in \cite{Gomez/alt:unp}. 

\subsubsection{Jacobson Normal Form and Invariant Factors.}
A different list of invariants of a finitely generated module of finite length over $R = \field{D}[ x ;\sigma,\delta]$,  with $\sigma$ bijective,  is that of invariant factors. They arise from the computation of a refinement of the diagonalisation algorithm of matrices in $R^{t\times m}$ described above. To be more precise, given $A \in R^{t \times m}$, the problem is the computation of $P \in \GL_t(R)$ and $Q \in \GL_m (R)$ such that $PAQ = J$, where 
\begin{equation}\label{JacobsonD}
J= \diag \{ f_{1}, \dots, f_r, 0, \dots, 0 \}
\end{equation}
 with the additional condition that $f_j$ is a total divisor of $f_{j+1}$ for all $j = 1, \dots, r-1$. Recall from \cite{Jacobson:1943} that $f \in R$ is a \emph{total divisor} of $g \in R$ if there exists a two sided ideal $I$ of $R$ such that $Rg \subseteq I \subseteq Rf$. An easy argument, taking that $Rf^* = f^*R \subseteq fR$ into account, shows that this is equivalent to the condition $Rg \subseteq fR$. A third equivalent condition is to require $RgR \subseteq Rf \cap fR$.  
 
 The diagonal matrix $J$ of  \eqref{JacobsonD}  is often called \emph{Jacobson normal form} of $A$. As for its existence concerns, we can assume, by Proposition \ref{diagonal}, that the matrix $A$ has been already reduced to a diagonal form $\Delta$ and, by an obvious inductive argument, we may assume that 
 \[
\Delta = \begin{pmatrix} b_1 & 0 \\
0 & b_2 \end{pmatrix}
 \]
with $\deg b_1  \leqslant   \deg b_2$.  Now, always following \cite{Jacobson:1943}, if $b_1$ is not a total divisor of $b_2$, then there exists $b \in R$ such that $bb_2 \notin b_1R$. By the  Right  Division Algorithm, $bb_2 = b_1 q + r_1$, where $q, r_1 \in R$ are nonzero polynomials and $\deg r_1 < \deg b_1$. Now, adding the second row of $\Delta$ multiplied on the left by $b$ to the first one, and substracting the first column of the resulting matrix multiplied on the right by $q$ to the second one, we reduce $\Delta$ to
\[
\begin{pmatrix} b_1 & r_1 \\ 
0 & b_2 \end{pmatrix} .
\]
Now, the discussion previous to Proposition \ref{diagonal} shows that this last matrix can be reduced by a sequence of row and column elementary operations to $\diag \{ b_1', b_2' \}$, with $\deg b_1' \leqslant \deg r_1 < \deg b_1$. After finitely many steps (in number less or equal than $\deg b_1$) we will arrive to a diagonal matrix $\diag  \{f_1, f_2 \}$ such that $f_1$ is a total divisor of $f_2$. We thus obtain

\begin{proposition}\cite{Jacobson:1943}\label{Jacobson}
Given $A \in R^{t \times m}$ there exist $P \in \GL_t(R)$ and $Q \in \GL_m(R)$ such that $PAQ = J = \diag \{ f_1, \dots, f_r, 0, \dots, 0 \}$, where $f_i$ is a total divisor of $f_{i+1}$ for all $i = 1, \dots, r-1$. 
\end{proposition}

Observe that, in contrast with Proposition \ref{diagonal}, Proposition \ref{Jacobson} does not claim that there exists an algorithm to compute $J$. The problem is that in the reduction from $\diag (b_1, b_2)$ to $\diag (f_1,f_2)$ described before, the first step seems not to be in general constructive: how could we check whether $Rb_2 \subseteq b_1R$? If not, how do we compute $b \in R$ such that $bb_2 \notin b_1 R$? This makes ``difficult'' the computation of the Jacobson normal form of a matrix with coefficients in $R = \field{D}[ x ;\sigma,\delta]$.  We see, thus, that the following problem is interesting from the effective point of view.

\begin{problem}
Given $b_1, b_2 \in R = \field{D}[ x ;\sigma,\delta]$, how to decide whether $Rb_2 \subseteq b_1R$? If not, how to compute $b \in R$ such that $bb_2 \notin b_1R$? 
\end{problem}

However, in some circumstances such an algorithm  does  exist. For instance, if $R$ is finitely generated as a module over its center, then an algorithm for the computation of $b$ such that $bb_2 \notin b_1R$, if it exists, may be constructed in the spirit of the  algorithms  given in \cite{Gomez/alt:unp} for the computation of the bound of an Ore polynomial. This opens the possibility of the implementation of an algorithm for the computation of the Jacobson normal form over this kind of Ore extensions, that include the case $\mathbb{F}[ x ;\sigma]$,  with $\mathbb{F}$ finite. 

In the ``opposite'' case, namely, when $R = \field{D}[ x ;\sigma,\delta]$ is a simple ring, that is, the only proper two-sided ideal of  $R$  is $\{ 0 \}$, the Jacobson normal form of any finitely generated left $R$ module is of the form $\diag \{ 1, \dots, 1, f, 0, \dots, 0 \}$. This means that any finitely generated left $R$-module of finite length is cyclic, thus, generated by a \emph{cyclic vector}. A probabilistic algorithm for computing cyclic presentations (and, hence, Jacobson normal forms)  has been proposed recently \cite[Remark 5.2]{Levandovskyy/Schindelar:2012}  (see \cite{Churchill/Kovacic:2002} for a rather complete study of the differential field case, where $R = \mathbb{K}[x;\delta]$ for a commutative field $\mathbb{K}$). 

 Some intermediate cases have been also explored. Thus, in \cite{Giesbrecht/Heinle:2012}, an algorithm for computing the Jacobson normal form is proposed for $R = \field{k}(z)[x;\sigma]$, where $\sigma$ is the \emph{shift operator} on the field of fractions $\field{k}(z)$ of a (commutative) polynomial ring $\field{k}[z]$ over a commutative field $\field{k}$ of characteristic $0$, that is, $\sigma (z) = z + 1$.  This is an example of a non simple Ore extension of a field which is not finitely generated as a module over its center $\field{k}$. A systematic study of the Jacobson normal form over this kind of ``centerless''  Ore extensions, by means of a classification of the so called two-sided elements, has been developed in \cite{Foldenauer:2012}. In particular, the shape of the Jacobson normal form  of a matrix with coefficients in $\mathbb{K}[x;\sigma]$,  where $\sigma$ is an autormophism of infinite order of a commutative field $\mathbb{K}$, is given in \cite[Theorem 2.88]{Foldenauer:2012}.  

The elements $f_1, \dots, f_r$ appearing in the diagonal of the Jacobson normal form of $A$,  see \eqref{JacobsonD},  are called \emph{the invariant factors} of $A$.  Even the simplest situation shows that their uniqueness has to take a weak form: in $\mathbb{Q}(t)[ x ;d/dt]$ we have $t^{-1}  x  t =  x  + t^{-1}$. Thus, even for $1 \times 1$ matrices, the expected uniqueness of their ``invariant factors'' is far from the familiar uniqueness in the commutative case. The best one can obtain is  to look  at $A$ as the matrix of a presentation of a finitely generated left module: Making use of the Krull-Schmidt theorem, Jacobson obtained (see \cite[Theorem 31, Ch. 3]{Jacobson:1943}) the following result credited to Nakayama.

\begin{theorem}\cite{Jacobson:1943}
Let $M$ be a non zero left module over $R = \field{D}[ x ;\sigma,\delta]$. If $M$ is finitely generated and of finite length, then $M \cong R/Rf_1 \oplus \cdots \oplus R/Rf_r$, for $f_1, \dots, f_r \in R$ polynomials of positive degree such that $f_{j}$ is a total divisor of $f_{j +1}$ for $j = 1, \dots, r-1$. The polynomials $f_1, \dots, f_r$, called \emph{the invariant factors} of $M$, are unique up to similarity. 
\end{theorem}

Again, in order to decide whether two modules are isomorphic, we see how central is the problem of deciding if two given polynomials are similar.   For instance, if $R$ is a simple ring (i.e., there are no nontrivial two sided ideals), then the Jacobson normal form of any matrix is $\diag \{ 1, \dots, 1, f, 0, \dots, 0 \}$. So, any finitely generated left $R$-module of finite length is cyclic, and it is isomorphic to $R/Rf$, where $f$ is ``its'' invariant factor. Thus, the problem of deciding whether two given finitely generated and indecomposable modules are isomorphic reduces, whenever the computation of Jacobson normal form is possible, to the problem of checking whether  two  given polynomials are similar.

\section{Left PBW Rings}\label{sec:LPBW}

From the classical Poincar\'e-Birkhoff-Witt Theorem (see, e.g. \cite{Dixmier:1977}) we know that given a finite-dimensional Lie algebra $\mathfrak{g}$ over a field $\field{k}$, with an ordered $\field{k}$-basis $x_1, \dots, x_n$, the standard monomials $x_1^{\alpha_1} \cdots x_n^{\alpha_n}$ for $(\alpha_1, \dots, \alpha_n) \in \Nn$ form a basis of the enveloping universal algebra $U(\mathfrak{g})$ as a  vector space  over $\field{k}$. Thus, the elements of $U(\mathfrak{g})$ resemble commutative polynomials in the usual polynomial ring $\field{k}[x_1, \dots,x_n]$, even though that the multiplication is not commutative: 
in $U(\mathfrak{g})$ we have $x_j x_i = x_ix_j + [x_j,x_i]$, where $[-,-]$ denotes the Lie bracket of $\mathfrak{g}$. This suggests that Gr\"obner basis methods might be adapted to develop effective algorithms in $U(\mathfrak{g})$. The Gr\"obner bases theory for $U(\mathfrak{g})$ was already introduced in  \cite{Lassner:1985}, and further developed in \cite{Apel/Lassner:1988},  and very soon extended to the more general settings of \emph{solvable polynomial algebras} \cite{Kandri-Rody/Weispfenning:1988} and \emph{solvable polynomial rings} \cite{Kredel:1993}. The common feature of all these non-commutative rings is that the usual multivariable Division Algorithm and the formulation of Buchberger's Theorem for commutative rings  hold  with minor changes. At the heart of the algorithms running for commutative polynomials over a field is the fact that the leading monomial (with respect to some term ordering) of a product of polynomials is the product of the leading  monomials  of the factors.  Since in the non-commutative setting the product of two monomials is not longer a monomial, we cannot expect a direct translation from the commutative case.  Looking at the ``exponents'' of the leading monomials, what is preserved under all these non-commutative generalizations is that the exponent of a product is the sum of the exponents of the factors. This point of view is very explicit in \cite{Castro:1984,Galligo:1983,Lejeune:1984}, and also in \cite{Bueso/Castro/Gomez/Lobillo:1998,Lobillo:1998}.  Going on with this idea, left PBW rings were introduced in \cite{Bueso/Gomez/Lobillo:2001a}. 

(Left) PBW rings cover a wide range of examples, from differential operator rings over a skew field (see Corollary \ref{cor:diffop}) to the aforementioned universal enveloping algebras of finite dimensional Lie algebras $U(\mathfrak{g})$.  Observe that the latter are not in general iterated Ore extensions of the base field $\field{k}$.  Moreover, besides a good algorithmic theory (see Section \ref{sec:Buchberger}), PBW rings have nice algebraic properties (see the last subsection of the paper).

\subsection{(Left) PBW Rings}
Let $R$ be a ring containing a skew field $\field{D}$. Given $x_1, \dots, x_n \in R$ and $\gordo{\alpha} = (\alpha_1, \dots, \alpha_n) \in \Nn$, we use the notation $\gordo{x}^{\aalpha} = x_1^{\alpha_1} \cdots x_n^{\alpha_n}$.  In particular,  $\gordo{x}^{\gordo{0}} = 1$.  Here, $\Nn$ denotes the additive monoid of all vectors $(\alpha_1, \dots, \alpha_n)$ with $\alpha_i$ a nonnegative integer for every $i= 1, \dots, n$. The ring $R$ is said to be \emph{left polynomial} over $\field{D}$ in $x_1, \dots, x_n \in R$ if  every $f \in R$ has a unique \emph{standard representation}
\begin{equation}\label{standard}
f = \sum_{\aalpha \in \Nn} c_{\aalpha}\xx^{\aalpha}, \qquad (c_{\aalpha} \in \field{D}) \, .
\end{equation}
Obviously, we understand that the set $\Newton (f) := \{ \aalpha \in \Nn : c_{\aalpha} \neq 0 \}$ is finite. In other words, $\{ \xx^{\aalpha} : \aalpha \in \Nn \}$ is a basis of $R$ as a left vector space (or module) over $\field{D}$. 

\begin{definition}\label{A.1.4.1}
We say that a total order $\preceq$ on $\Nn$ is \emph{admissible} if 
\begin{enumerate}[(1)]
\item $\gordo{0}\preceq \gordo{\alpha}$ for every $\gordo{\alpha}\in
\Nn$;
\item  $\gordo{\alpha} + \gordo{\gamma} \preceq \gordo{\beta} + \gordo{\gamma}$ for all
$\gordo{\alpha}, \gordo{\beta}, \gordo{\gamma} \in \Nn$ with $\gordo{\alpha} \preceq \gordo{\beta}$.
\end{enumerate}
\end{definition}

Given a left polynomial ring $R$ over $\field{D}$ in $x_1, \dots, x_n \in R$, and an admissible ordering $\preceq$ on $\Nn$, we may define \emph{the exponent} of $f \in R$ as \begin{equation*} \exp{f} = \max_{\preceq} \Newton (f),\end{equation*} if $f \neq 0$, and $\exp{0} = - \infty$.  Here, the symbol $- \infty$ is assumed to behave properly with respect to the addition and the ordering of $\Nn$. 

By $\eepsilon_i$ we denote the vector in $\Nn$ all of whose entries are $0$ except for a value $1$ in the $i$-th component.

\begin{theorem}\label{LPBW} \cite[Theorem 1.2]{Bueso/Gomez/Lobillo:2001a}
Let $R$ be a left polynomial ring over a skew field $\field{D}$ in $x_1, \dots, x_n$, and $\preceq$ be an admissible ordering on $\Nn$. The following statements are equivalent:
\begin{enumerate}[(a)]
\item $\exp{fg} = \exp{f} + \exp{g}$ for all $f, g \in R$;
\item \begin{enumerate}[1.]
\item\label{Q} for every $1 \leqslant i < j \leqslant n$ there exist $0 \neq q_{ji} \in D$ and $p_{ji} \in R$ such that
\begin{equation*}
x_j x_i = q_{ji} x_ix_j + p_{ji}, \qquad  \exp{p_{ji}} \prec \eepsilon_i + \eepsilon_j \,;
\end{equation*}
\item\label{Q'} for every $1 \leqslant j \leqslant n$ and every $0 \neq a \in \field{D}$  there exist $0 \neq q_{ja} \in \field{D}$ and $p_{ja} \in R$ such that 
\begin{equation*}
x_j a = q_{ja} x_j + p_{ja}, \qquad \exp{p_{ja}} \prec \eepsilon_j \, ;
\end{equation*}
\end{enumerate}
\item \begin{enumerate}[1.]
\item for every $\aalpha, \bbeta \in \Nn$, there exist $0 \neq q_{\alpha, \beta} \in \field{D}$ and  $p_{\aalpha,\bbeta} \in R$ such that
\begin{equation*}
\xx^{\aalpha}\xx^{\bbeta} = q_{\aalpha,\bbeta}\xx^{\aalpha + \bbeta} + p_{\aalpha, \bbeta}, \qquad \exp{p_{\aalpha,\bbeta}} \prec \aalpha + \bbeta \, ;
\end{equation*}
\item for every $\aalpha \in \Nn$ and every $0 \neq a \in \field{D}$ there exist $0 \neq a^{\aalpha} \in \field{D}$, $p_{\aalpha, a} \in R$ such that 
\begin{equation*}
\xx^{\aalpha} a = a^{\aalpha} \xx^{\aalpha} + p_{\aalpha,a}, \qquad \exp{p_{\aalpha, a}} \prec \aalpha \, .
\end{equation*}
\end{enumerate}
\end{enumerate}
\end{theorem}

\begin{definition}\label{def:LPBW}
A left polynomial ring $R$ satisfying the equivalent conditions  of  Theorem \ref{LPBW} is said to be a \emph{left PBW ring} with respect to $\preceq$. We will use the notation
\[
R=\field{D}\{x_1,\ldots,x_n;  Q,Q',\preceq\},
\]
where
\[
Q=\{x_{j} x_{i} = q_{ji} x_{i} x_{j} + p_{ji};\ 1 \leqslant i < j \leqslant n \}
\]
and
\[
Q'=\{x_{j} a = q_{ja} x_{j} + p_{ja};\ 1\leqslant j \leqslant n, 0 \neq a \in \field{D} \} \, .
\]
\end{definition}

There are some pertinent remarks concerning Definition \ref{def:LPBW}. 

\begin{remark}
Solvable polynomial rings from \cite{Kredel:1993} are left PBW with the additional requirements that $\exp{p_{ja}} = \gordo{0}$ for every $a \in \field{D}$, and $ j \in \{ 1, \dots, n \}$,  and that $\{\xx^{\aalpha} : \aalpha \in \Nn \}$ is also a basis of $R$ as a \emph{right} vector space over $\field{D}$. Thus, they are examples of (twosided) PBW rings in the sense of \cite{Bueso/Gomez/Lobillo:2001a}. In fact, once fixed $\field{D}$, $x_1, \dots, x_n \in R$ and $\preceq$ on $\Nn$, it can be proved \cite[Theorem 1.9]{Bueso/Gomez/Lobillo:2001a} that $R$ is right polynomial and left PBW if and only if $R$ is left polynomial and right PBW. The ring is said to be a \emph{PBW ring with respect to $\preceq$} if it satisfies these equivalent conditions. Thus, every solvable polynomial ring is a PBW ring, but the converse fails, as \cite[Example 1.6]{Bueso/Gomez/Lobillo:2001a} shows (see also Proposition \ref{triangular} below). 
\end{remark}

\begin{remark}
Inspired by the commutative graded structures from \cite{Robbiano:1986}, a rather general class of algebras, called G-algebras (a more general notion than that of $G$-algebra from \cite{Levandovskyy:2005}), was introduced in \cite{Apel:1992}. Essentially, these G-algebras are filtered by an ordered semigroup, and the idea is to lift Gr\"obner bases from the associated graded algebra to the G-algebra. Of course, the strategy of this approach  is  to find a simpler  associated graded algebra, and  it is understood  that the lifting property from it to the filtered algebra is feasible. Thus, further requirements are to be assumed in order to make this technique constructive. For example, PBW rings fit in this scheme since they have an $(\Nn,\preceq)$-filtration such that the associated $\Nn$-graded ring (see \cite{Gomez:1999}) is isomorphic to a crossed product $\field{D} \ast \Nn$ (see the proof of \cite[Proposition 1.10]{Bueso/Gomez/Lobillo:2001a}). However, the advantages of lifting computations and Gr\"obner bases from $\field{D} \ast \Nn$ to the PBW ring are not clear, since, anyway, an exponent with respect to $\preceq$ has to be used. 
\end{remark}

\begin{remark}
A basic example of PBW ring (with respect  to  any $\preceq$) is the usual polynomial ring $\field{D}[x_1, \dots, x_n]$ in some variables $x_1, \dots, x_n$, where the multiplication is prescribed by the rules $x_ia = a x_i$, and $x_jx_i = x_ix_j$ for all $a \in  \field{D}$ and $i, j \in \{1, \dots, n \}$. In \cite{Kredel:1993} (following the setup of \cite{Kandri-Rody/Weispfenning:1988}) solvable polynomial rings are conceived as a kind of deformations of $\field{D}[x_1, \dots,x_n]$, in the sense that the polynomials are multiplied accordingly a ``new'' non-commutative multiplication $*$. By \cite[Proposition 3.2.5]{Kredel:1993},  $f*g = c fg + h$ for uniquely determined $0 \neq c \in \field{D}$ and $h \in R$ with $\exp{h} \prec \exp{f} + \exp{g}$, where, for a moment, $fg$ denotes the ``commutative'' multiplication (it is not really commutative, unless $\field{D}$ is assumed to be commutative). But, once $*$ is assumed to exist, the process may be reversed, and, hence, the ``commutative'' multiplication can be recovered from $*$, since $fg = c^{-1}(f*g) - c^{-1}h$. We see, thus, that the pre-existence of the ``commutative'' multiplication on the  left  $\field{D}$-vector space with basis $\{ \xx^{\aalpha} : \aalpha \in \Nn \}$ plays no essential role. 
\end{remark}

\begin{corollary}\label{cor:estrellapunto}
Let $R$ be a left polynomial ring over $\field{D}$ in $x_1, \dots, x_n$. Let $*$ denote the multiplication of $R$.  Define a new multiplication on the standard monomials of the left $ \field{D}$-vector space $R$ as $a\xx^{\aalpha} b \xx^{\bbeta} = ab \xx^{\aalpha + \bbeta}$ for $a, b \in \field{D}, \aalpha, \bbeta \in \Nn$. Since  $\{\xx^{\aalpha} : \aalpha \in \Nn \}$ is a basis of $R$ as a left $\field{D}$-vector space, this multiplication extends to a new multiplication on $R$ (denoted here by juxtaposition) making of $R$ a ring isomorphic to $\field{D}[x_1, \dots, x_n]$.  Given an admissible ordering $\preceq$ on $\Nn$, the following statements are equivalent:
\begin{enumerate}[(a)]
\item $R$ is a left PBW ring with respect to $\preceq$;
\item\label{estrellapunto} for every nonzero $f, g \in R$, $f*g = c fg + h$  for some  $0 \neq c \in \field{D}$ and $h \in R$ with $\exp{h} \prec \exp{f} + \exp{g}$. 
\end{enumerate}
\end{corollary}
\begin{proof}
If $R$ is a left PBW ring, then $\exp{f*g} = \exp{f} + \exp{g} = \exp{fg}$. Taking $c = \lc{f}\lc{g}^{\exp{f}}q_{\exp{f},\exp{g}}\lc{g}^{-1}\lc{f}^{-1} \in \field{D}$, we get from \cite[Corollary 2.10, Ch. 2]{Bueso/Gomez/Verschoren:2003} that $\exp{f*g - cfg} \prec \exp{f} + \exp{g}$. Conversely, if $R$ satisfies the second condition, then putting $f= \xx^{\aalpha}, g = \xx^{\bbeta}$ we get the first condition in the second equivalent statement in Theorem \ref{LPBW}, while the second one is obtained with $f = \xx^{\aalpha}, g = a$. 
\end{proof}

 Note that $c$ and $h$ in Corollary \ref{cor:estrellapunto}.\eqref{estrellapunto} are uniquely determined by $f$ and $g$. 

\begin{remark}
If $R = \field{D}\{x_1,\ldots,x_n;  Q,Q',\preceq\}$ is a left PBW ring, then the elements $q_{ji}, p_{ji}, q_{ja}, p_{ja}$ appearing in $Q$ and $Q'$ are far from being arbitrary. Thus, for instance,  the map $(-)^{\aalpha} : \field{D} \to \field{D}$ sending $a$ to $a^{\aalpha}$ (see Theorem \ref{LPBW}.(c)) is a ring endomorphism of $\field{D}$ for each $\aalpha \in \Nn$. To see this, just use the equality $(\xx^{\aalpha}a)b = \xx^{\aalpha}(ab)$ for all $a, b \in \field{D}$. In particular, the map $(-)^{\eepsilon_j} : \field{D} \to \field{D}$ sending $a \mapsto a^{\eepsilon_j} = q_{ja}$ gives a ring endomorphism for each $j = 1, \dots, n$. In general, the constants $q_{ji}, p_{ji}, q_{ja}, p_{ja}$ are subject to constraints of different types imposed by the associativity of the product of $R$ in conjunction with the linear independence over $\field{D}$ of the standard monomials $\xx^{\aalpha}$. An example of this kind of reasoning appears in the proof of the following proposition. 

\begin{proposition}\label{PBW}
Let $R=\field{D}\{x_1, \dots, x_n;\ Q,\ Q', \ \preceq\}$ be a left PBW ring. Then $R$ is a PBW ring if and only if $(-)^{\eepsilon_i}  : \field{D} \to \field{D}$ is an automorphism for every $i = 1, \dots, n$. 
\end{proposition}
\begin{proof}
For $i < j$, we get from Theorem \ref{LPBW} that, for all $a \in  \field{D}$, 
\begin{gather*}
\xx^{\eepsilon_i + \eepsilon_j}a = \xx^{\eepsilon_i}\xx^{\eepsilon_j}a = \xx^{\eepsilon_i}(a^{\eepsilon_j}\xx^{\eepsilon_j} + p_{\eepsilon_j,a}) = (a^{\eepsilon_j})^{\eepsilon_i} \xx^{\eepsilon_i}\xx^{\eepsilon_j} + p_{\eepsilon_i,a}\xx^{\eepsilon_j} + \xx^{\eepsilon_i}p_{\eepsilon_j, a},	
\end{gather*} 
and that
\[
\xx^{\eepsilon_i + \eepsilon_j}a = a^{\eepsilon_i + \eepsilon_j}\xx^{\eepsilon_i + \eepsilon_j} + p_{{\eepsilon_i + \eepsilon_j},a}= a^{\eepsilon_i + \eepsilon_j}\xx^{\eepsilon_i} \xx^{\eepsilon_j} + p_{{\eepsilon_i + \eepsilon_j},a} \,.
\]
Comparing the leading coefficients of the right hand standard polynomials of both expressions, we deduce that $(a^{\eepsilon_j})^{\eepsilon_i} = a^{\eepsilon_i + \eepsilon_j}$ for all $a \in \field{D}$. That is, $(-)^{\eepsilon_i} \circ (-)^{\eepsilon_j} = (-)^{\eepsilon_i + \eepsilon_j}$ for all $i < j$. A straightforward induction argument on the well ordered set $(\Nn,\preceq)$, always with the help of Theorem \ref{LPBW}, will lead us to
\begin{equation}\label{epsilongen}
(-)^{\aalpha} = [(-)^{\eepsilon_1}]^{\alpha_1} \circ \cdots \circ [(-)^{\eepsilon_n}]^{\alpha_n}
\end{equation}
for all $\aalpha \in \Nn$. By \cite[Theorem 1.9]{Bueso/Gomez/Lobillo:2001a}, $R$ is a PBW  ring  if and only if $(-)^{\aalpha}$ is an automorphism for every $\aalpha \in \Nn$. By \eqref{epsilongen}, this last condition is fulfilled if and only if $(-)^{\eepsilon_j}$ is an automorphism  for  every $j = 1, \dots, n$. 
\end{proof}

\end{remark}

\begin{remark}
By Theorem \ref{LPBW}, the multiplication of the left PBW ring $R$ is completely determined by the ``relations'' $Q$ and $Q'$. In practice, given two polynomials $f, g \in R$, what guarantees that the product $fg$ may be computed in finitely many steps?
The key is that, by Dickson's Lemma, $\preceq$ is a well ordering on $\Nn$ (see, e.g. \cite[Proposition 1.20, Ch. 2]{Bueso/Gomez/Verschoren:2003}). Thus, every strictly decreasing sequence $\aalpha_1 \succ  \aalpha_2 \succ \cdots  $ in $\Nn$ must be finite. Statement (c) of Theorem \ref{LPBW} may be used then to compute the product of two polynomials in finitely many steps. To see this quickly, assume that ``step'' means the application of one of the reduction rules in $Q$ or $Q'$. If $fg$ cannot be computed in finitely many steps, then we choose such a pair $f, g$ with $\exp{f} + \exp{g}$ minimal. A straightforward application of Theorem \ref{LPBW} will lead us to a contradiction. 
\end{remark}

\subsection{Differential and Difference Operator Rings}
An Ore extension $\field{D}[x;\sigma,\delta]$ of a skew field $\field{D}$ is a left PBW ring, where $Q$ is empty and $Q' = \{ xa = \sigma(a)x + \delta(a) : a \in \field{D} \}$. Moreover, it is a PBW ring if and only if $\sigma$ is an automorphism. Thus, in particular, both the differential operator case ($\sigma = id$) and the difference operator case ($\delta = 0$), are instances of (left) PBW rings.  What happens if we iterate the process? That is, under which circumstances is an iterated Ore extension $\field{D}[x_1; \sigma_1, \delta_1]\cdots [x_n; \sigma_n, \delta_n]$ a (left) PBW ring? Let us give a partial answer to a slightly more general question: If $R = \field{D}\{x_1, \dots, x_n; \ Q, \ Q', \ \preceq \}$ is a left PBW ring, under which conditions is an Ore extension $S = R[x; \sigma, \delta]$ of $R$ a left PBW ring?

The first pertinent observation is that the standard monomials 
$$\{ \xx^{\aalpha}x^i : (\aalpha, i) \in \mathbb{N}^{n+1} \}$$ form a basis of $S$ as a left vector space over $\field{D}$, thus $S$ becomes a left polynomial ring over $\field{D}$ in $x_1, \dots, x_n, x$.  Now it is rather natural to define $\exp{f} = (\exp{\operatorname{lc}_x(f)}, \deg_x(f))$, for $f \in S$, where we are considering the admissible order $\preceq_*$ on $\mathbb{N}^{n+1}$ given by $(\aalpha, i) \preceq_* (\bbeta,j)$ if $i < j$,  or $i = j$ and $\aalpha \preceq \bbeta$. A straightforward computation shows that $\exp{fg} = \exp{f} + \exp{g}$ for all $f, g \in S$ if and only if $\exp{\sigma(r)} = \exp{r}$ for all $r \in R$. This last condition, in conjunction with Theorem \ref{LPBW} and Proposition \ref{PBW}, may be used to deduce the following proposition (see \cite[Theorem 3.1, Ch. 2]{Bueso/Gomez/Verschoren:2003}).

\begin{proposition}\label{triangular}
Let $S = R[x; \sigma, \delta]$ be an Ore extension of a left  PBW  ring $R = \field{D}\{x_1, \dots, x_n; \ Q, \ Q', \ \preceq \}$. Keeping the previous notation, we get that $S$ is a left PBW ring in $x_1, \dots, x_n, x$ with respect to $\preceq_*$ if and only if $\sigma (\field{D}) \subseteq \field{D}$ and for every $i = 1, \dots, n$ there exist $0 \neq q_i \in \field{D}$ and $p_i \in R$ such that $\sigma(x_i) = q_i x_i + p_i$ and $\exp{p_i} \prec \eepsilon_i$. Moreover, in such a case, $S$ is a PBW ring if and only if $R$ is a PBW ring and the restriction of $\sigma$ to $\field{D}$ is an automorphism.  
\end{proposition}

If $S = R[x;\sigma,\delta]$ is known to be a (left) PBW ring accordingly to Proposition \ref{triangular}, then we have $S = \field{D}\{ x_1, \dots, x_n, x; \ Q_*, \ Q_*', \ \preceq_* \}$, where
\[
Q_* = Q \cup \{ xx_i = q_ix_ix + p_ix + \delta (x_i) ;\ 1 \leqslant i \leqslant  n \}
\]
and
\[
Q_*' = Q' \cup \{ xa = \sigma (a) x + \delta (a) ;\ a \in \field{D} \}\,.
\]

Our answer to the question on iterated Ore extensions of $\field{D}$ is the following consequence of Proposition \ref{triangular}.

\begin{corollary}\label{MasqueOrePoly}
Let $R = \field{D}[x_1; \sigma_1, \delta_1] \cdots [x_n; \sigma_n, \delta_n]$ be an iterated Ore extension of $\field{D}$, and $\preceq_{lex}$ be the lexicographical order on $\Nn$ with $\eepsilon_1 \prec_{lex} \cdots \prec_{lex} \eepsilon_n$. Then $R$ is a left PBW ring with respect to $\preceq_{lex}$ if and only if for every $1 < i <j \leqslant n$ there exist $0 \neq q_{ji} \in \field{D}$ and $f_{ji} \in \field{D}[x_1; \sigma_1, \delta_1] \cdots [x_{i-1}; \sigma_{i-1}, \delta_{i-1}]$  such that
\[
\sigma_j(x_i) = q_{ji}x_i + f_{ji}, \qquad (1  \leqslant  i < j \leqslant n),
\]
and  $\sigma_j (\field{D}) \subseteq \field{D}$ for every $j = 1, \dots, n$.
Moreover, in such a case, $R$ is a PBW ring if and only if $\sigma_j$ is an automorphism of $\field{D}$ for every $j = 1, \dots, n$. 
\end{corollary}

The class of rings described in Corollary \ref{MasqueOrePoly} contain that of polynomial Ore algebras, as defined in \cite[Definition 2.1]{Chyzak/Salvy:1998}  (In concrete examples, an iterated but flexible use of Proposition \ref{triangular} would lead to more general orderings than pure lexicographical ones).  It is noteworthy to mention that these rings were explicitly considered as rings of linear operators. A \textsc{Maple} implementation of many algorithms based on polynomial Ore algebras is available (see \cite{Chyzak/Quadrat/Robertz:2007,Cluzeau/Quadrat:2009}).

\begin{corollary}\label{cor:diffop}
Every differential operator ring $\field{D}[x_1, \delta_1] \cdots [x_n,\delta_n]$ is a PBW ring $R = \field{D}\{x_1, \dots, x_n; Q, Q', \preceq_{lex} \}$, where
\[
Q = \{ x_jx_i = x_ix_j + \delta_j(x_i): 1 \leqslant i < j \leqslant n \},
\]
and
\[
Q' = \{x_j a = a x_j + \delta_j (a) : 1 \leqslant j \leqslant n, a \in \field{D} \} \, .
\]
\end{corollary}

We stress that $\delta_j(a)$ needs not to belong to $\field{D}$ for $j >1$, so general differential operator rings are not covered by the notion of a solvable polynomial ring from \cite{Kredel:1993}.  

We have seen in Proposition \ref{triangular} that an Ore extension $R[x;\sigma,\delta]$ of a left PBW ring $R$ is a PBW ring with the obvious extension $\preceq_*$ of the ordering $\preceq$ of $R$ if and only if $\sigma$ satisfies suitable conditions. This leads naturally to the following general problem.

\begin{problem}\label{Rsigma}
Assume a ring $R$ having a good theory of left Gr\"obner bases, and let $S = R[x;\sigma,\delta]$ be an Ore extension. Under which conditions $S$ has a good theory of left Gr\"obner bases?
\end{problem}

For instance, if we require finiteness on left Gr\"obner bases to have a good theory, then $R$ has to be assumed to be left noetherian. In such a case, $S$ will be left noetherian whenever $\sigma$ is close to be an automorphism. For instance, if $\sigma : \field{k}[y] \to \field{k}[y]$ sends $f(y)$ to $f(y^2)$ ($\field{k}$ is a commutative field), then $\field{k}[y][x,\sigma]$ is not left nor right noetherian (see \cite[Example 2.11.(iii)]{McConnell/Robson:1988}). Even assuming that $\sigma$ is an automorphism, and $R$ a PBW ring, a general answer to Problem \ref{Rsigma}, beyond Proposition \ref{triangular}, seems to be not trivial.

\section{Algorithms for Modules over a Left PBW Ring}\label{sec:Buchberger}

Traditionally, the expositions of the theory of Gr\"obner bases for (commutative or not) polynomial rings deal first with (left) ideals and then with submodules of free  (left)  modules of finite rank. Since both developments are parallel, the latter refers continuously to the analogy with the first one. We prefer here  to  present a resume of the theory at the module level, which, apart of saving repetitions, allows the use of Schreyer's method \cite{Schreyer:1980} to simplify the proof of  Buchberger's theorem. V. Levandovskyy \cite{Levandovskyy:2005} used a similar point of view when dealing with $G$-algebras (that is, PBW algebras in the sense of Definition \ref{def:PBWalg}).

\subsection{The Division Algorithm and Gr\"obner Bases for Free Modules of Finite Rank}
When dealing with free left modules of rank $m$ over a left PBW ring, the standard monomials will have exponents indexed by $\Nnm = \Nn \times \{ 1, \dots, m \}$. There is an action of the monoid $\Nn$ on the set $\Nnm$, denoted by the symbol $+$, and defined as $(\aalpha, i) + \bbeta = (\aalpha + \bbeta, i)$ for all $\aalpha \in \Nn, (\bbeta, i) \in \Nnm$. A subset $E$ of $\Nnm$ is \emph{stable} under this action if $E + \Nn = E$, that is, if for every $ (\aalpha, i) \in E, \bbeta \in \Nn$, $(\aalpha + \bbeta, i) \in E$. Stable subsets of $\Nn$ (i.e., $m = 1$) are called \emph{monoideals}. As a consequence of Dikson's Lemma (see \cite[Lemma 1.10, Ch. 2]{Bueso/Gomez/Verschoren:2003}), every non empty stable subset $E \subseteq \Nnm$ is represented as
\begin{equation}\label{FiniteGen}
E = B + \Nn
\end{equation}
for a finite subset $B \subseteq E$, called a \emph{set of generators} of $E$. If we take $B$ minimal with respect to inclusion, then it is unique, and it will be called the \emph{basis} of $E$. 

Now, let $R = \field{D}\{x_1, \dots, x_n; \ Q, Q', \preceq \}$ be a left PBW ring, and $R^m$ a free left $R$-module with basis $\{ \gordo{e}_1, \cdots, \gordo{e}_m \}$. Since $\mathcal{B} = \{ \gordo{x}^{\gordo{\alpha}}; \, \gordo{\alpha} \in\Nn \}$ is a 
basis of $R$ as a left $\field{D}$-vector space, it follows that $\mathcal{B}_{m} = \{\gordo{x}^{\gordo{\alpha}}\gordo{e}_i; \, (\gordo{\alpha},i)
\in\Nnm \}$ is a basis of $R^m$ as a left $\field{D}$-vector space. Therefore, every element $\gordo{f} \in R^{m}$
has a unique \emph{standard representation}
\[
\gordo{f} = \sum_{(\gordo{\alpha},i) \in\Nnm} c_{(\gordo {\alpha},i)} \gordo{x}^{\gordo{\alpha}}\gordo{e}_i \, ,
\]
where $c_{(\gordo {\alpha},i)} \in \field{D}$ is non zero for finitely many $(\aalpha, i) \in \Nnm$. 
We define the \emph{Newton diagram} of $\gordo{f}$ to be
\[
\Newton(\gordo{f}) = \{ (\gordo{\alpha},i) \in\Nnm; \, c_{(\gordo{\alpha},i)} \neq0 \} \,.
\]

 The admissible ordering $\preceq$ on $\Nn$ may be extended to $\Nnm$ in different ways. For instance, we can require $(\aalpha, i) \preceq (\bbeta,j)$ if $\aalpha \prec \bbeta$, or $\aalpha = \bbeta$ and $i \leqslant j$ (this is the TOP ordering, from ``term over position''. There is also de POT order, defined in an obvious way).  
We thus may define the \emph{exponent} of $\gordo{f}$ to be
\[
\exp{\gordo{f}} = \max_{\preceq}{\Newton(\gordo{f})} \, .
\]
If $(\gordo{\alpha},i)$ is the exponent of $\gordo{f}$ then we call
$\gordo{\alpha}$ the \emph{scalar exponent}\index{scalar exponent}
of $\gordo{f}$, denoted by $\sexp{\gordo{f}}$ and we call $i$ the
\emph{level}\index{level} of $\gordo{f}$ denoted by
$\level(\gordo{f})$. Therefore, 
\[
\gordo{f} = c_{\exp{\gordo{f}}}\gordo{x}^{\sexp{\gordo{f}}}\gordo{e}_{\level(\gordo{f})} + \sum_{(\gordo{\beta},j)
\prec \exp{\gordo{f}}}c_{(\gordo{\beta},j)} \gordo{x}^{\gordo{\beta}}\gordo{e}_j \,.
\]
We will refer to $\lc{\gordo{f}} =c_{\exp{\gordo{f}}}$ as the \emph{leading coefficient} of $\gordo{f}$. The following proposition follows from Theorem \ref{LPBW}.

\begin{proposition}\label{aditivo}
For all $\gordo{f} \in R^m$ and $h \in R$ we have:
\begin{enumerate}[(1)]
\item $\exp{h\gordo{f}} = \exp{h} + \exp{\gordo{f}}$;
\item $\lc{h\gordo{f}} = \lc{h}\lc{{\gordo{f}}}^{\exp{h}}q_{\exp{h},\sexp{\gordo{f}}}$.
\end{enumerate}
\end{proposition}

It follows from Proposition \ref{aditivo} that, for any (nonzero) left $R$-submodule $L$ of $ R^m$, the set $\Exp(L) = \{\exp{\gordo{f}} : 0 \neq \gordo{f} \in L \}$ is a stable subset of $\Nnm$. Thus, applying \eqref{FiniteGen} to $E = \Exp(L)$, we get immediately:

\begin{proposition}\label{Groebnerexiste}
For any left  $R$-submodule $L$ of $R^m$ there exists  a finite subset $G = \{\gordo{g}_1, \dots, \gordo{g}_t \} \subseteq L$ such that $\Exp (L) = \bigcup_{i=1}^{t} (\exp{\gordo{g}_i} + \Nn)$. 
\end{proposition}

The set $G$ in Proposition \ref{Groebnerexiste} is said to be a \emph{Gr\"obner basis of $L$} (with respect to $\preceq$). Any Gr\"obner basis of $L$ turns out to be  a  (finite) set of generators of $L$ as a left $R$-module. To see this, the Division Algorithm for free left $R$-modules (Algorithm \ref{alg:5.1}) is needed. With Proposition \ref{aditivo}  at hand, the proof of the correctness of Algorithm \ref{alg:5.1} is straightforward. 

\begin{algorithm}\label{DivisionAlgo}
\caption{Division Algorithm on a free module over a left PBW ring $R = \field{D}\{x_1, \dots,x_n; Q, Q', \preceq \}$ } \label{alg:5.1}
\begin{algorithmic}
\REQUIRE $\gordo{f},\gordo{f}_1,\ldots,\gordo{f}_s\in R^m$ with $\gordo{f}_i\neq 0$ ($1\leqslant i\leqslant s$) \ENSURE
$h_1,\ldots,h_s \in R,\gordo{r} \in R^m$ such that $\gordo{f}= h_1\gordo{f}_1+\cdots+h_s\gordo{f}_s+\gordo{r}$ \\ and $\gordo{r}=\gordo{0}$ or
$\Newton(\gordo{r}) \cap \bigcup_{i=1}^s(\exp{\gordo{f}_i}+\Nn)=\emptyset$ and \\
$\max \{\exp{h_1}+\exp{\gordo{f}_1},\ldots,\exp{h_s}+\exp{\gordo{f}_s},
\exp{\gordo{r}}\}=\exp{\gordo{f}}$. 
\INIT
$h_1:=0,\ldots,h_s:=0,\gordo{r}:=\gordo{0}, \gordo{g}:=\gordo{f}$
\WHILE{$\gordo{g}\neq \gordo{0}$} \IF{
$\exp{\gordo{g}} \in \exp{\gordo{f}_i} + \Nn$ for some $i =1, \dots, s$} 
\STATE
$a_i:=\lc{\gordo{g}}(\lc{f_i}^{\sexp{\gordo{g}}-\sexp{\gordo{f}_i}}q_{\sexp{\gordo{g}}-\sexp{\gordo{f}_i}),
\sexp{\gordo{f}_i}})^{-1}$ \STATE
$h_i:=h_i+a_i\gordo{x}^{\sexp{\gordo{g}}-\sexp{\gordo{f}_i}}$ \STATE
$\gordo{g}:=\gordo{g}-a_i\gordo{x}^{\sexp{\gordo{g}}-\sexp{\gordo{f}_i}}\gordo{f}_i$
\ELSE \STATE $\gordo{r}:=\gordo{r}+\lm(\gordo{g})$ \STATE
$\gordo{g}:=\gordo{g}-\lm(\gordo{g})$ \ENDIF \ENDWHILE
\end{algorithmic}
\end{algorithm}

%%%%%%%%%%%%%%%%%%%%%%%%%%%%%%%%%%%%%
\begin{definition}\label{A.3.5.10}
%%%%%%%%%%%%%%%%%%%%%%%%%%%%%%%%%%%%%
The element $\gordo{r}$ obtained in Algorithm \ref{alg:5.1} is said
to be a \emph{remainder} of the division of $\gordo{f}$ by the set
$F = \{\gordo{f}_{1},\dots,\gordo{f}_{t}\}$. We will denote by
$\lres{\gordo{f}}{F}$ any remainder, that is, any  $\gordo{r} \in R^m$ satisfying the properties of the output of Algorithm  \ref{alg:5.1}.
\end{definition}%%%%%%%%%%%%%%%%%%%%%%%%%%%%%%%%%%%%%%%%%%%%%%%
%
%%%%%%%%%%%%%%%%%%%%%%%%%%%%%%%%%%%%
\begin{example}\label{A.3.5.9}
%%%%%%%%%%%%%%%%%%%%%%%%%%%%%%%%%%%%
Consider $\gordo{f}_1=(x,y)$, $\gordo{f}_2=(y,x)\in R^2$. We use the
lex ordering on $R=\BbbC\{x,y;\ yx=xy+1,\ \preceq_{lex}\}$ and TOP
with $\gordo{e}_{1}\prec\gordo{e}_{2}$ in $R^2$. Let $F =\{
\gordo{f}_{1}, \gordo{f}_{2}\}$ and $\gordo{f} = ( x^{2} + y^{2},1)$.
Then $\lres{\gordo{f}}{F}=(2 x^{2},0)$ since
\[
\gordo{f} = - x\gordo{f}_{1} + y\gordo{f}_{2}+(2 x^{2},0)\,.
\]
\end{example}%%%%%%%%%%%%%%%%%%%%%%%%%%%%%%%%%%%%%%%%%%%%%%%%%%
%
%%%%%%%%%%%%%%%%%%%%%%%%%%%%%%%%
As a consequence of the Division Algorithm, we get the following characterization of Gr\"obner bases,  which also solves the ``membership problem''. 

\begin{theorem}\label{Groebnerchar}
Let $G = \{\gordo{g}_1, \dots, \gordo{g}_t \}$ be a subset of a left $R$-submodule $L$ of $R^m$. The following statements are equivalent:
\begin{enumerate}[(a)]
\item $G$ is a Gr\"obner basis of $L$;
\item if $\gordo{f} \in L$, then $\lres{\gordo{f}}{G} = 0$;
\item if $\gordo{f} \in L$, then there exist $h_1, \dots, h_t \in R$ such that $f = \sum_{i=1}^t h_i\gordo{g}_i$ with 
$$\exp{\gordo{f}} = \max_{\preceq} \{ \exp{h_i} + \exp{\gordo{g}_i} : 1 \leqslant i \leqslant t \} \, .$$
\end{enumerate}
\end{theorem}

 Proposition \ref{Groebnerexiste} and  Theorem \ref{Groebnerchar} have the following relevant  consequence.

\begin{corollary}
Any Gr\"obner basis of $L$ generates $L$ as a left $R$-module. Therefore, every finitely generated left $R$-module is finitely presented, that is, $R$ is a left noetherian ring.  
\end{corollary}

\subsection{Buchberger's Theorem and Syzygies}\label{Buchberger}

\subsubsection{Buchberger's Theorem.}
Given $G = \{ \gordo{g}_1, \dots, \gordo{g}_t \} \subseteq R^m$, let us denote by ${}_R\langle G \rangle$ the left $R$-submodule of $R^m$ generated by $G$.   Buchberger's Theorem gives a criterion to check whether $G$ is a Gr\"obner basis of ${}_R\langle G \rangle$. Let us assume that $\gordo{g}_i$ is monic for $i = 1, \dots, t$. For $i \neq j$ such that $\level(\gordo{g}_i) = \level(\gordo{g}_j)$, define the monomial in $R$
\[
r_{ij} = q^{-1}_{\aalpha_i \vee \aalpha_j-\aalpha_i, \aalpha_i}\xx^{\aalpha_i \vee \aalpha_i-\aalpha_i},
\]
where $\aalpha_i = \sexp{\gordo{g}_i}$, and
\[
SP (\gordo{g}_i, \gordo{g}_j) = r_{ij} \gordo{g}_i - r_{ji}\gordo{g}_j \, .
\]

\begin{theorem}[Buchberger]\label{Buch}
The set $G$ is a Gr\"obner basis of ${}_R\langle G \rangle$ if and only if $\lres{SP (\gordo{g}_i, \gordo{g}_j)}{G} = 0$ for all $1 \leqslant i < j \leqslant t$ with $\level (\gordo{g}_i) = \level (\gordo{g}_j)$.
\end{theorem}
\begin{proof}
Let $I = \{ (i,j) : 1 \leqslant i < j \leqslant t, \level (\gordo{g}_i) = \level (\gordo{g}_j) \}$, and assume that $ \lres{SP (\gordo{g}_i, \gordo{g}_j)}{G} = 0$ for all $(i,j) \in I$. By the Division Algorithm
\[
SP(\gordo{g}_i,\gordo{g}_j) = \sum_{k=1}^t h_{ijk}\gordo{g}_k
\]
for some polynomials $h_{ijk} \in R$ with $(i,j) \in I, 1 \leqslant k \leqslant t$ such that 
\[
\max_{\preceq} \{ \exp{h_{ijk}} + \exp{g_k} : k =1, \dots, t\} = \exp{SP(\gordo{g}_i,\gordo{g}_j)}\,.
\]
Let $\{ \gordo{e}'_1, \dots, \gordo{e}'_t \}$ be the canonical basis of $R^t$ and set, for every $(i,j) \in I$, 
\[
\gordo{s}_{ij} = r_{ij}\gordo{e}_i' - r_{ji}\gordo{e}_j' - \sum_{k=1}^t h_{ijk}\gordo{e}'_k \,.
\]
Define, for $(\aalpha, i), (\bbeta,j) \in \mathbb{N}^{n,(t)}$,
\[
(\gordo{\alpha},i)\preceq_G(\gordo{\beta},j)\Longleftrightarrow\left \{
\begin{array}[c]{ll}
\gordo{\alpha}+\exp{\gordo{g}_i}\prec \gordo{\beta}+\exp{\gordo{g}_j}; &
\\   \text{or} & \\
\gordo{\alpha}+\exp{\gordo{g}_i}=\gordo{\beta}+\exp{\gordo{g}_j} & \mbox{ and } j \leqslant i \,.
\end{array}
\right.
\]
It turns out that $\preceq_G$ is an admissible order in $\mathbb{N}^{n,(t)}$ and that
\[
\operatorname{exp}_{\preceq_G}(\gordo{s}_{ij}) = (\aalpha_i \vee \aalpha_j - \aalpha_i, i),
\]
for all $(i,j) \in I$ (see \cite[Lemma 4.1, Lemma 4.5, Ch.6]{Bueso/Gomez/Verschoren:2003}). 

Now, let $\varphi : R^t \to R^m$ be the homomorhism of left $R$-modules defined by $\varphi(\gordo{e}'_i) = \gordo{g}_i$, $i = 1, \dots, t$. Given $f \in {}_R\langle G \rangle$, let $\gordo{g} = \sum_i a_i \gordo{e}'_i \in R^t$ such that $\gordo{f} = \varphi (\gordo{g})$. Now, divide (with Algorithm \ref{alg:5.1}) $\gordo{g}$ by the set $S = \{ \gordo{s}_{ij}: (i, j) \in I \}$. The outcome is
\[
\gordo{g} = \sum_{(i,j) \in I}a_{ij}\gordo{s}_{ij} + \gordo{h}, \qquad \Newton(\gordo{h}) \cap \left(\bigcup_{(i,j) \in I} (\operatorname{exp}_{\preceq_G}(\gordo{s}_{ij})+ \Nn)\right) = \emptyset
\]
for some $a_{ij} \in R$. Write $\gordo{h} = \sum_{i=1}^th_i\gordo{e}'_i$ for some $h_1, \dots, h_t \in R$. Then 
\begin{equation}\label{higi}
f = \varphi(\gordo{g}) = \varphi(\gordo{h}) = \sum_{i=1}^t h_i\gordo{g}_i \,.
\end{equation} 
Next, we will prove that, if $\gordo{h} \neq 0$,  in the expression \eqref{higi}, $\exp{h_i\gordo{g}_i} \neq \exp{h_j\gordo{g}_j}$ if $i \neq j$. To this end, assume that $\exp{h_i\gordo{g}_i} = \exp{h_j\gordo{g}_j}$ for $i < j$. Then $\exp{h_i} + \exp{\gordo{g}_i} = \exp{h_j} + \exp{\gordo{g}_j} = \exp{h_j\gordo{g}_j}$, which in particular implies that
$\exp{h_j\gordo{g}_j} \in (\exp{\gordo{g}_i} + \Nn) \cap (\exp{\gordo{g}_j} + \Nn) = (\aalpha_i \vee \aalpha_j, \level(\gordo{g}_i)) + \Nn$. Thus, there exists $\aalpha \in \Nn$ such that $\exp{h_j} + \aalpha_j = \aalpha_i \vee \aalpha_j + \aalpha$, whence $\exp{h_j} \in \aalpha_i \vee \aalpha_j - \aalpha_j + \Nn$.  Therefore, $(\exp{h_j},j) \in (\aalpha_i \vee \aalpha_j - \aalpha_i, j) + \Nn = \operatorname{exp}_{\preceq_G}(\gordo{s}_{ij}) + \Nn$; a contradiction. 

If $f \neq 0$, then $\gordo{h} \neq 0$ and $\max_{\preceq} \{ \exp{h_i\gordo{g}_i} : 1 \leqslant i \leqslant t \}$ has to be reached at a single value $j \in \{ 1, \dots, s \}$. We thus obtain from \eqref{higi} that $\exp{\gordo{f}} = \max_{\preceq} \{ \exp{h_i\gordo{g}_i} : 1 \leqslant i \leqslant t \}$, and $G$ is a  Gr\"obner  basis by Theorem \ref{Groebnerchar}. 
\end{proof}

As a consequence of Theorem  \ref{Buch},  in conjunction with Dickson's Lemma, we get that the well known ``commutative'' Buchberger's Algorithm, as formulated for example in \cite{Adams/Loustaunau:1994} or \cite{Becker/Weispfenning:1993}, holds true in the setting of left PBW rings (see Algorithm \ref{alg:5.2}). 

 For commutative polynomial rings, an alternative to Gr\"obner bases are Janet bases. In the non-commutative setting, Janet's algorithm was adapted to polynomial Ore algebras in \cite{Robertz:2007}.  The implementation based on
polynomial Ore algebras \cite{Chyzak/Quadrat/Robertz:2007}  can actually make use
of another \textsc{Maple} package called {\sc JanetOre} which
computes Janet bases for submodules of free
left modules over certain Ore algebras.  It would be interesting to investigate to what extent  these ideas can be extended to a general left PBW ring. 

 \begin{algorithm}
\caption{Gr\"obner Basis Algorithm for Left Modules over a left PBW ring $R = \field{D}\{x_1, \dots, x_n; Q, Q', \preceq \}$} \label{alg:5.2}
\begin{algorithmic}
\REQUIRE $F=\{\gordo{f}_1,\ldots,\gordo{f}_s\}\subseteq R^m$ with $\gordo{f}_i\neq
\gordo{0}$ ($1\leqslant i\leqslant s$) \ENSURE $G=\{\gordo{g}_1,\ldots,\gordo{g}_t\}$, a
Gr\"obner basis for $R\gordo{f}_1+\cdots+R\gordo{f}_s$. \INIT $G:=F$,
$B:=\{\{\gordo{f},\gordo{g}\};\ \gordo{f}, \gordo{g}\in G, \gordo{f} \neq \gordo{g}, \level(\gordo{g}) = \level(\gordo{f}) \}$ \WHILE{$B\neq
\emptyset$} \STATE Select $\{\gordo{f},\gordo{g}\}\in B$ \STATE $B:=B\setminus
\{\{\gordo{f},\gordo{g}\}\}$ 
 \STATE
$\gordo{h}:=\lres{\SP(\gordo{f},\gordo{g})}{G}$ \IF{$\gordo{h}\neq \gordo{0}$} \STATE
$B:=B\cup\{\{\gordo{p},\gordo{h}\};\ \gordo{p}\in G, \level(\gordo{p}) = \level(\gordo{h}) \}$ \STATE $G:=G\cup\{\gordo{h}\}$
\ENDIF \ENDWHILE
\end{algorithmic}
\end{algorithm}

\subsubsection{Computation of Syzygies.} 
Given a finite subset $F = \{\mathbf{f}_1, \dots, \mathbf{f}_s \} \subseteq R^m$,  whose elements have been implicitly ordered by their subscripts, we may consider the matrix 
\[
F = \begin{pmatrix} \mathbf{f}_1 \\ \vdots \\ \mathbf{f}_s  \end{pmatrix}\in R^{s \times m}\,.
\]
Thus, $F$ defines a homomorphism of left $R$-modules $\psi : R^s \to R^m$ by $\psi (\gordo{f}) = \gordo{f} F$ for $\gordo{f} \in R^s$ (in the notation of Section \ref{LA}, $F = A_\psi$).  Then $\ker \psi = \Syz(F)$. 

The following corollary of the proof of Theorem  \ref{Buch}  gives the key to compute $\Syz(F)$. We keep the notation introduced there. 

\begin{corollary}\label{cor:syz}
If $G$ is a Gr\"obner basis of ${}_R\langle G \rangle$, then the set $S = \{ \gordo{s}_{ij}: (i, j) \in I \}$ is a Gr\"obner basis of $\Syz(G)$ with respect to $\preceq_G$.
\end{corollary}
\begin{proof}
Given $\gordo{g}$ as in the proof of Theorem \ref{Buch}, then $\gordo{g} \in \Syz(G)$ if and only if $\gordo{f} = 0$. In this case, it follows that if in \eqref{higi} there is some nonzero summand, then taking $i$  such that $\exp{h_i\gordo{g}_i}$ is maximal, then $\exp{h_i\gordo{g}_i} = \exp{h_j\gordo{g}_j}$ for some $i \neq j$. But we have seen in the proof of Theorem \ref{Buch} that this is not possible. Thus, $h_i\gordo{g}_i = 0$ for every $i = 1, \dots, t$ and, hence, $\gordo{h} = 0$. Since $\gordo{h}$ is a  remainder  of a division of $\gordo{g}$ by $S$, we deduce from Theorem \ref{Groebnerchar} that $S$ is a Gr\"obner basis of $\Syz(G)$. 
\end{proof}

A finite set of generators of $\Syz(F)$ is given by the rows of the matrix $\mathsf{syz}(F)$. To compute this matrix, we need matrices $P \in R^{t \times s}$ and $\hat{P} \in R^{s \times t}$ such that $G = PF$ and $F = \hat{P}G$, where $G$ is a Gr\"obner basis of ${}_R\langle F \rangle$ (and, hence, ${}_R \langle G \rangle = {}_R \langle F \rangle$).  The matrix $P$ may be obtained by keeping track of the reductions performed during Algorithm \ref{alg:5.2}, while $\hat{P}$ is computed by using the Division Algorithm \ref{alg:5.1}. Let $S = \{\gordo{s}_1, \dots, \gordo{s}_r \}$ be the set of generators of $\Syz(G)$ from Corollary \ref{cor:syz}. Interpreting its elements as the rows of a matrix, $S = \mathsf{syz}(G) \in R^{r \times t}$.  We get from \eqref{syzcomp} that 
\begin{equation}\label{eq:syzF}
\mathsf{syz}(F) = \begin{pmatrix} \mathsf{syz}(G) P \\
I_s -\hat{P} P \end{pmatrix} \qquad (\text{with } G = PF, F = \hat{P}G)\,.
\end{equation}

\subsubsection{Remarks on Homological Computations.}

\begin{remark}
With \eqref{eq:syzF} at hand, the general Propositions \ref{prop:Im} and \ref{prop:ker} lead to algorithms for the computation of images and kernels of homomorphisms between finitely presented left modules over a left PBW ring, and, henceforth, free resolutions, pullbacks, pushouts, and potentially all the finite categorical constructions from Homological Algebra (see e.g. \cite{Hilton/Stammbach:1971} for these notions).  
\end{remark} 

\begin{remark}
As for the free resolutions concerns, let us remark that every finitely generated left module over $R = \field{D}\{x_1,\dots,x_n; Q,Q', \preceq \}$ has a free resolution of length at most $n$ (thus, $R$ has finite left global homological dimension). For a constructive proof, based on Schreyer's method, see \cite[Theorem 5.5, Ch. 6]{Bueso/Gomez/Verschoren:2003}. Analogous approaches were developed in \cite{Levandovskyy:2005} for $G$-algebras and in \cite{Adams/Loustaunau:1994} for commutative polynomial rings. 
\end{remark}

\begin{remark}\label{rem:jM}
The possibility of computing effectively free resolutions of finitely presented left modules opens, as in the case of left PBW rings, the chance of developing algorithms for the calculation of the  groups $\Ext_R^{i}(M,N)$.  This were already done in \cite{Bueso/Gomez/Lobillo:2001a,Bueso/Gomez/Verschoren:2003} in the case that $N$ is a centralizing bimodule, and in \cite{Lobillo:1998} when $N = R$ is a PBW algebra (see also \cite{Chyzak/Quadrat/Robertz:2005}, where homological computations over Ore algebras are related to linear control systems).  As a particular case, given a  nonzero  finitely presented left module $M$ over a left PBW ring $R$, we may compute the \emph{grade number}
\begin{equation}\label{eq:grade}
j(M) = \min \{ j \in \N : \Ext^j_{ R}(M,R) \neq 0 \}\,.
\end{equation}
This number plays a relevant role in the study of \emph{holonomic modules}, and of Bernstein Duality,  over rings of differential operators, namely, those finitely generated modules whose grade number equals the global homological dimension of the base ring (see \cite{Bjork:1979}). Many rings of differential operators are factors of differential operator rings. Recall from \ref{cor:diffop} that any differential operator ring $\field{D}[x_1,\delta_1]\cdots[x_n,\delta_n]$ over a skew field is a PBW ring.  Thus, the possibility of computing effectively the grade number for left PBW rings is of potential interest in Algebraic Analysis  (e.g., in the context of the grade filtration,
cf. \cite{Quadrat:2013}).  A simpler algorithm for computing $j(M)$ for modules over PBW algebras is given at the end of Section \ref{sec:GKdim}. 
\end{remark}

\section{Gelfand-Kirillov Dimension for Modules over PBW Algebras}\label{sec:GKdimPBW}

In this section we will present an algorithm for computing the Gelfand-Kirillov dimension of a finitely generated module over a PBW algebra. We will include a very brief introduction to filtered algebras and modules, and we will characterize PBW algebras as the filtered algebras having a quantum affine space as associated graded algebra.  The algorithm for the computation of the Gelfand-Kirillov dimension reduces the problem to the determination of the degree of the Hilbert function of a (commutative) monomial ideal. This reduction needs some results on filtered algebras. Such an algorithm was first presented for enveloping algebras of finite dimensional Lie algebras in \cite{Bueso/Castro/Jara:1997}, then for PBW algebras with quadratic relations \cite{Bueso/Castro/Gomez/Lobillo:1998}, and finally in full generality \cite{Bueso/Gomez/Lobillo:2001b}.

\subsection{Filtered Algebras and the Gelfand-Kirillov Dimension}

Filtrations play a relevant role in the study of rings of differential operators (see \cite{Bjork:1979}) and, more generally, in the investigation of properties of non-commutative algebras over a field (a good reference here is \cite{McConnell/Robson:1988}). The degree of growth of some filtrations (the standard finite dimensional filtrations) of an algebra provide a useful invariant called Gelfand-Kirillov dimension. This dimension makes also sense for modules, and may be effectively computed for a large class of algebras.  Some fundamental properties of this invariant
can be found in \cite{Lorenz:1988}, \cite[Chapter 8]{McConnell/Robson:1988} and
\cite{Krause/Lenagan:2000}.

\subsubsection{Growth degree.}
For any function $f: \mathbb{N} \rightarrow [1,+\infty)$, consider its \emph{growth degree} or \emph{degree} defined as
\begin{equation}\label{dimension}
d(f) = \inf \{ \nu : f(s) \leqslant s^{\nu} \hbox{ for } s \gg  0 \}\,.
\end{equation}
Since, given $s$ and $\nu$, one has
\[
f(s) \leqslant s^{\nu} \Leftrightarrow \log f(s)/ \log s \leqslant \nu,
\]
it follows that 
\[
d(f) = \limsup \frac{\log f(s)}{\log s}
\]
which, of course, needs not to be finite. If $f(s)$ coincides with a polynomial function $p(s)$ for $s$ big enough, then $d(f)$ is the usual degree of $p(s)$. 

\subsubsection{Filtrations and Gelfand-Kirillov dimension.}

Let $\field{k}$ denote a field. Given a left module $M$ over a $\field{k}$-algebra $R$, and vector subspaces $V \subseteq R$, $U \subseteq M$, by $VU$ we denote the vector subspace of $M$ spanned by all products of the form $vu$, with $v \in V$, $u \in U$. This gives sense to the expression $V^2$ and, recursively, to the power $V^s$ for $s$ any nonnegative integer. We understand $V^0 = \field{k}$ and $V^1 = V$. 

\begin{definition}
Let $R$ be an algebra over a field $\field{k}$. A \emph{filtration of} $R$ is a family of vector subspaces $FR = \{ F_sR : s \in \N \}$ of $R$ such that
\begin{enumerate}
\item for each $i, j \in \N$, $F_iR \, F_jR \subseteq F_{i+j}R$,
\item $F_iR \subseteq F_jR$ if $i \leqslant j$,  and
\item $\bigcup_{ s  \in \N} F_sR = R$. 
\end{enumerate}
\end{definition}

The filtration $FR$ is said to be \emph{finite} or \emph{finite dimensional} if $F_sR$ is of finite dimension over $ \field{k}$ for every $s \geq 0$. 

A finite filtration is said to be \emph{standard} if $F_0R = \field{k}$ and $F_sR = (F_1R)^s$, for every $s \geq 0$. Observe that if $R$ has a standard filtration, then it is generated as an algebra by $F_1R$ and, henceforth, it is finitely generated as an algebra over $\field{k}$. For standard filtrations on $R$ we often use the notation $F_sR = R_s$.  

If our algebra $R$ is finitely generated as a $\field{k}$-algebra (we say then that $R$ is \emph{affine} over $\field{k}$), then there exists a vector subspace $V$ of finite dimension such that $\bigcup_{s \in \N}  V^s = R$, that is, $R$ has a finite standard filtration given by $R_s = V^s$ for $s \in  \N$ (we assume that $1 \in V$).

\begin{definition}
Given a filtration $FR$ of $R$, and a left $R$-module $M$,  a \emph{filtration} of $M$ is  any family of vector subspaces $\{ F_sM : s \in \N \}$ such that 
\begin{enumerate}
\item for each $i, j \in \N$, $F_iR \, F_jM \subseteq F_{i+j}M$,
\item $F_iM \subseteq F_jM$ if $i \leqslant j$,  and
\item $\bigcup_{s \in \N} F_sM = M$. 
\end{enumerate}
\end{definition}

The following lemma is, of course, very well known. Its easy proof is included here to stress the dependence of the definition of the Gelfand-Kirillov dimension from the standard filtrations (for non standard filtrations things are a little more complicated, since a good behavior with respect to Gelfand-Kirillov dimension depends on the properties of the associated graded algebra, see \cite{McConnell/Stafford:1989} for a deep study of this topic). 

\begin{lemma}\label{GKdimdefined}
Let $\{ R_s : s \in \N \}$ and $\{ R'_s : s \in \N \}$ be finite dimensional standard filtrations of $R$. Then
\[
d(\dim_k R_s) = d(\dim_k R'_s)
\]
\end{lemma}
\begin{proof}
Since $R_1'$ is finite dimensional and $R = \bigcup_{s \in \N} R_{ s}$, there exists $a \in \N$ such that $R'_1 \subseteq R_a$.  Being both filtrations standard, we get $R'_s = {R'_1}^s \subseteq {R_a}^s = {R_1}^{as} = R_{as}$ for every $s \in \N$.  Thus, $d(\dim_k R'_s) \leqslant d(\dim_k R_s)$.  The other inequality follows by symmetry. 
\end{proof}

\begin{definition}
The \emph{Gelfand-Kirillov dimension} of an affine $\field{k}$-algebra is defined as the degree of the function $s \mapsto \dim_{\field{k}}R_s$ for any standard finite dimensional filtration $\{R_s : s \in  \N  \}$ of $R$. 
\end{definition}

An argument similar to that of the proof of Lemma \ref{GKdimdefined} shows that the following definition is mathematically sound. 

\begin{definition}
Let $M$ be a finitely generated left $R$-module, and $M_0$ the finite dimensional vector subspace of $M$ spanned by any finite set of generators of ${}_RM$. Let $\{ R_s : s \in  \N  \}$ be any finite dimensional standard filtration of $R$. The \emph{Gelfand-Kirillov dimension} of $M$ is defined as the growth degree of the function $s \mapsto \dim_{\field{k}}R_sM_0$. 
\end{definition}

\subsection{PBW Algebras and Filtrations}

Let $\field{k}$ be any commutative field, and $R = \field{k}\{ x_1, \dots, x_n; Q, Q', \preceq \}$ be a left PBW ring. If $R$ is a $\field{k}$-algebra, then we will say that $R$ is a \emph{PBW algebra}. This is equivalent to require that the relations $Q'$ are trivial, that is, $x_i a = a x_i$ for every $a \in \field{k}$, $1 \leqslant i \leqslant n$. Obviously, a PBW algebra is also a right PBW ring. 

Although PBW algebras are perfectly defined as before,  our purposes for this section require the following rephrasing of the definition.

\begin{definition}\label{def:PBWalg}
An algebra $R$ over a field $\cuerpo{k}$ is said to be a \emph{PBW algebra}  if there exist  $x_1, \dots, x_n \in R$ such that 
\begin{enumerate}
\item\label{PBWbasis} \textbf{(PBW basis)} The standard monomials $\mono{x}{\alpha} =x_1^{\alpha_1} \cdots x_n^{\alpha_n}$
with  $\aalpha = (\alpha_1, \dots, \alpha_n) \in \BbbN^n$ form a basis of $R$ as a
vector space over $\field{k}$.
\item\label{QRelations} \textbf{(Quantum relations)} There are nonzero scalars $q_{ji} \in \cuerpo{k}$ and polynomials $p_{ji} \in R$ ($1
\leqslant i < j \leqslant n$) such that the relations 
$$Q = \{
x_jx_i = q_{ji}x_ix_j + p_{ji}, \; 1 \leqslant i < j \leqslant n
\}$$ 
hold in $R$. 
\item \textbf{(Bounded relations)} There is an admissible ordering
$\preceq$ on $\mathbb{N}^n$ such that 
$$
\exp{p_{ji}} \prec
\eepsilon_i + \eepsilon_j$$ 
for every $1 \leqslant i < j \leqslant
n$.
\end{enumerate}
We will use the notation $R = \field{k}\{ x_1, \dots, x_n; Q, \preceq \}$. 
\end{definition}

\begin{remark} Our definition of PBW algebra is equivalent to that of \emph{polynomial algebra of solvable type} from \cite{Kandri-Rody/Weispfenning:1988}. These algebras are also known as \emph{$G$-algebras} after \cite{Levandovskyy:2005}, where several fundamental algorithms for ideals and modules were implemented in \textsc{Singular} (see also \cite{Levandovskyy/Schoenemann:2003}, \cite{Levandovskyy:2006}).  Implementations in \textsc{MAS} (Modula-2 Algebra System) were done in \cite{Kredel:1993}. We prefer to keep the name PBW algebras because they are a particular case of PBW rings.  
\end{remark}

\begin{example}
The basic example of PBW algebra is the \emph{$n$-dimensional quantum affine space} $\mathcal{O}_{q}(\field{k}^n)$, where $q = (q_{ij}) \in \field{k}^{n \times n}$ is multiplicatively antisymmetric matrix, that is, $q_{ij} \neq 0$, and $q_{ij} = q_{ji}^{-1}$ for all $1 \leqslant i, j \leqslant n$,  and $q_{ii} = 1$ for all $1 \leqslant i \leqslant n$. This algebra is generated by $n$ indeterminates $x_1, \dots, x_n$ subject to the relations $x_jx_i = q_{ji}x_ix_j$ for all $1 \leqslant i,j \leqslant n$.  Clearly, $\mathcal{O}_{q}(\field{k}^n)$ is a PBW algebra with respect to any admissible order $\preceq$ on $\Nn$. When $q_{ij} = 1$ for all $i, j$, we get the usual commutative polynomial ring $\field{k}[x_1, \dots, x_n]$.
\end{example}

A remarkable characterization of PBW algebras is that they are precisely those filtered algebras with an associated graded ring isomorphic to $\mathcal{O}_q(\field{k}^n)$ for some $n$ (see Theorem \ref{refiltrar} below).   Recall that if an algebra $R$ over $\field{k}$ has a filtration $FR = \{F_sR : s \in \N \}$ then its \emph{associated graded algebra} $\gr(R)$ is defined as the vector space
\[
\gr(R) = \bigoplus_{s \in \N} \frac{F_sR}{F_{s-1}R} \qquad (F_{-1} R = \{ 0 \}), 
\]
endowed with the product defined on \emph{homogeneous elements} 
\[
a + F_{s-1}R \in \frac{F_{s}R}{F_{s-1}R}, \quad a' + F_{s' -1}R \in \frac{F_{s'}R}{F_{s'-1}R} 
\]
by
\[
(a + F_{s-1}R)(a' + F_{s'-1}R) = aa' + F_{s + s' -1}R,
\]
and extended to $\gr(R)$ by linearity.

We need also to fix some notation on degree lexicographical orders. For any $\gordo{u} = (u_1,\dots,u_n) \in \mathbb{R}^n$ with $u_i \geq 0$ for all $i = 1, \dots, n$, and $\aalpha \in \Nn$, write $\abs{\gordo{u}}{\aalpha} = u_1\alpha_1 + \dots + u_n\alpha_n$. By $\preceq_{\gordo{u}}$ we denote the \emph{$\gordo{u}$-weighted lexicographical}  order on $\Nn$, defined by  
\[
\aalpha \preceq_{\gordo{u}} \bbeta \Leftrightarrow \begin{cases} \abs{\gordo{u}}{\aalpha} < \abs{\gordo{u}}{\bbeta} & \\
 or &  \\
 \abs{\gordo{u}}{\aalpha} = \abs{\gordo{u}}{\bbeta} & \text{ and } \aalpha \preceq_{lex} \bbeta,  \end{cases} 
\]
where $\preceq_{lex}$ denotes the lexicographical order on $\Nn$ with $\eepsilon_1 \prec_{lex} \cdots \prec_{lex} \eepsilon_n.$
When $\gordo{u} = (1, \dots, 1)$, the ordering $\preceq_{\gordo{u}}$ is just the \emph{degree lexicographical ordering}. By $\N_+^n$ we denote the set of all vectors $\gordo{w} = (w_1, \dots, w_n) \in \Nn$ with $w_i >0$ for all $i = 1, \dots, n$. 

\begin{theorem}\cite[Theorem 3.14]{Bueso/Gomez/Lobillo:2001a}\label{refiltrar}
The following conditions are equivalent for an algebra $R$ over a field $\field{k}$:
\begin{enumerate}[(a)]
\item\label{filt}  There is a filtration of $R$ such that $\gr(R)$ is isomorphic to $\mathcal{O}_{q}(\field{k}^n)$;
\item\label{filtfin} there is a finite filtration of $R$ such that $\gr(R)$ is isomorphic to $\mathcal{O}_q(\field{k}^n)$;
\item\label{pbw} $R$ is a PBW $\field{k}$-algebra with respect to some admissible ordering $\preceq$ in $\Nn$;
\item\label{pbww} $R$ is a PBW $\field{k}$-algebra with respect to some admissible ordering $\preceq_{\gordo{w}}$ in $\Nn$, for some $\gordo{w} \in \N_+^n$. 
\end{enumerate}
\end{theorem}

The proof of Theorem \ref{refiltrar} is interesting from the computational point of view, we will thus give in the following paragraphs a sketch of it highlighting the effective aspects. 

\subsubsection{From Filtrations to Quantum Relations.} 

Let $R$ be a filtered $\cuerpo{k}$-algebra with
filtration $FR$ such that $\gr(R)$ is 
generated by homogeneous elements $y_1, \dots, y_n$ with
$\deg(y_i) = u_i \geqslant 0$ for $1 \leqslant i \leqslant n$. Put
$\gordo{u} = (u_1, \dots, u_n) \in \BbbN^n$ and, for $1 \leqslant i
< j \leqslant n$, let $0 \neq q_{ji} \in \cuerpo{k}$ such that
$y_jy_i = q_{ji}y_iy_j$.

If $x_1, \dots, x_n \in R$ are such that $y_i = x_i + F_{u_i-1}R$ for $1
\leqslant i \leqslant n$ then a straightforward computation gives for any $s \in \N$:
\[
F_sR = \sum_{\abs{u}{\aalpha} \leqslant
s}\cuerpo{k}\mono{x}{\alpha}\,.
\]
Moreover, the algebra $R$ is generated by $x_1, \dots, x_n$ and there
are polynomials $p_{ji}$ for $1 \leqslant i < j \leqslant n$ such
that
\[
Q \equiv x_jx_i = q_{ji}x_ix_j + p_{ji} \qquad \textrm{with}\;\;
\exp{p_{ji}} \prec_{\gordo{u}} \eepsilon_i + \eepsilon_j\,.
\]
Finally, 
if $\gr(R)$ is a quantum affine space, that is, the standard monomials $y_1^{\alpha_1} \cdots y_n^{\alpha_n}$ are linearly independent over $\field{k}$, then it is easily checked that the monomials $x_1^{\alpha_1} \cdots x_n^{\alpha_n}$ are linearly independent, too. 

Therefore, 
$R = \field{k}\{x_1, \dots, x_n; Q, \preceq_{\gordo{u}} \}$ is a PBW
algebra. This proves \eqref{filt} $\Rightarrow$ \eqref{pbw} in Theorem \ref{refiltrar}.

\subsubsection{Bounding Quantum Relations.}

Let us now consider an algebra $R$ with generators $x_1, \dots,
x_n$ satisfying a set
\begin{equation}\label{qrelations}
Q = \{ x_jx_i = q_{ji}x_ix_j + p_{ji}, \; 1 \leqslant i < j
\leqslant n \}
\end{equation}
of \emph{quantum relations}, for nonzero scalars $q_{ji} \in
\cuerpo{k}$ and polynomials $p_{ji} \in R$. If there is an admissible ordering $\preceq$ on
$\mathbb{N}^n$ such that $\exp{p_{ji}} \prec \eepsilon_i +
\eepsilon_j$, that is, $\max_{\preceq} \Newton(p_{ji}) \prec
\eepsilon_i + \eepsilon_j$, for every $1 \leqslant i < j \leqslant
n$, we say that the quantum relations are
\emph{$\preceq$-bounded.} 

Our next aim is to answer to the question: Given the quantum relations $Q$, it is possible to decide whether they are $\preceq$-bounded  for some $\preceq$?   To this end, 
define, for $1 \leqslant i < j
\leqslant n$, the translated sets $C_{ji} = \Newton(p_{ji}) -
\eepsilon_i - \eepsilon_j$, and the finite set
\begin{equation}\label{transnewton}
C_Q = \bigcup_{1 \leqslant i < j \leqslant n} C_{ji} \cup \{
-\eepsilon_1, \dots, -\eepsilon_n \} \,.
\end{equation}
Let us also define the (open) polyhedron
\begin{equation}\label{feasible}
\Phi_Q = \{ \gordo{w} \in \mathbb{R}^n :
\bilin{\gordo{w}}{\gordo{\gamma}} < 0 \; \forall \; \gordo{\gamma} \in
C_Q \} \,.
\end{equation}
Observe that all points in $\Phi_Q$, if any, have strictly
positive components. By density of the rational numbers in the reals, if $\Phi_Q$ is not empty, then it
contains vectors with strictly positive rational components and,
by multiplying by a suitable positive integer, $\Phi_Q$ contains
vectors with strictly positive integer components. Define, for a weight vector $\gordo{w}$,
\[
\deg_{\gordo{w}}(p_{ji})  = \max \{ \abs{\gordo{w}}{\aalpha} : \aalpha \in \Newton(p_{ji}) \} \,.
\]

The ideas from \cite{Mora/Robbiano:1988} or \cite{Weispfenning:1987} may be used to obtain the following proposition. A direct explicit proof was given in \cite[Proposition 2.1, Theorem 2.3]{Bueso/Gomez/Lobillo:2001c}. 

\begin{proposition}\label{refiltra} \cite[Proposition 2.8]{Bueso/Gomez/Lobillo:2001a} 
Let $R$ be an algebra over a field $\field{k}$ with
generators $x_1, \dots, x_n$ satisfying the set $Q$ of quantum relations
\eqref{qrelations}. There exists an admissible ordering $\preceq$
on $\mathbb{N}^n$ such that $\exp{p_{ji}} \prec \eepsilon_i +
\eepsilon_j$ for every $1 \leqslant i < j \leqslant n$ if and only
if $\Phi_Q$ is not empty. In such a case, $\deg_{\gordo{w}}{p_{ji}}
< w_i + w_j$ for every $\gordo{w} = (w_1, \dots, w_n) \in \Phi_Q$
and every $1 \leqslant i < j \leqslant n$.
\end{proposition}

As a consequence of Proposition \ref{refiltra}, we get \eqref{pbw} $\Rightarrow$ \eqref{pbww} of Theorem \ref{refiltrar}.   A proof of \eqref{pbww} $\Rightarrow$ \eqref{filtfin} is given in Subsection  \ref{sec:GKdim}. 

\subsubsection{Is this Algebra PBW?} Assume known that a given algebra $R$ is generated by finitely many elements $x_1, \dots, x_n$ that satisfy a set $Q$ of quantum relations. According to Proposition \ref{refiltra}, the relations $Q$ will be $\preceq$-bounded for some admissible ordering $\preceq$ if and only if the open polytope $\Phi_Q$ is non empty, and, in such a case, $Q$ will be $\preceq_{\gordo{w}}$-bounded for any $\gordo{w} \in \Phi_Q$. An effective procedure to compute $\gordo{w} \in \Phi_Q$ with integer components, optimal in the sense that the total degree of $\gordo{w}$ is minimal, is to solve the linear programming problem (see \cite[Section 4]{Bueso/Gomez/Lobillo:2001c}):
\begin{equation}\label{lpro}
\begin{array}{l}
\textbf{minimize} \; f(\gordo{w}) = w_1 + \cdots + w_n \\ \textbf{with the
constraints}\\ \Phi_Q \equiv
\begin{cases}
w_i \geq 1 & (i = 1, \dots n), \\
 \langle \gordo{w},\gordo{\gamma} \rangle \leq -1 & (\ggamma \in C_{ Q})) \,.
\end{cases}
\end{array}
\end{equation}

Once an admissible ordering $\preceq_{\gordo{w}}$ is computed, in order to decide whether $R = \field{k}\{x_1, \dots, x_n; Q, \preceq_{\gordo{w}}\}$ is a PBW algebra we need to check if the standard monomials $x_1^{\alpha_1} \cdots x_n^{\alpha_n}$ are linearly independent over $\field{k}$. This can be decided by using Bergman's Diamond Lemma \cite{Bergman:1978}. 

More precisely, if $R$ satisfies a set $Q$ of $\preceq$-bounded quantum relations, then $R = \field{k}\langle x_1, \dots, x_n\rangle/I_Q$, where $I_Q$ is the two-sided ideal generated by $Q$ in the free $\field{k}$-algebra $\field{k}\langle x_1, \dots, x_n \rangle$. Following \cite[Ch. 3]{Bueso/Gomez/Verschoren:2003}, interpret $(Q,\preceq)$ as a reduction system on $\field{k}\langle x_1, \dots, x_n \rangle$, as in \cite[Definition 4.1, Ch. 3]{Bueso/Gomez/Verschoren:2003}. Then we may apply Bergman's Diamond Lemma to $(Q,\preceq)$ (see \cite[Theorem 4.7, Chp. 3]{Bueso/Gomez/Verschoren:2003}) to deduce that $R$ is a PBW algebra if and only if $x_k(q_{ji}x_ix_j + p_{ji} )$ and $(q_{kj}x_jx_k + p_{kj})x_i$ reduce to the same standard polynomial under $(Q,\preceq)$. By applying a suitable one-step reduction, one obtains the non-degeneracy condition from \cite[Lemma 2.1]{Levandovskyy:2005} that appears in the definition of $G$-algebra \cite[Definition 3.2]{Levandovskyy:2005}. Hence, PBW algebras and $G$-algebras from \cite{Levandovskyy:2005} are the same mathematical objects.

\subsection{Computation of the Gelfand-Kirillov Dimension}\label{sec:GKdim}

\subsubsection{Filtering Modules over PBW Algebras.}
Let $M$ be a finitely generated left module over a PBW $\field{k}$-algebra 
\[
R = \field{k}\{x_1,\dots, x_n; Q, \preceq \}\,. 
\]
If we wish to compute the Gelfand-Kirillov  dimension  of $M$, the use of a standard filtration on $R$ is not a good choice in the general case (it only could work in the case that the $p_{ji}$'s are quadratic). Fortunately, $R$ has finite dimensional filtrations that work nicely from the computational point of view. In fact, Proposition \ref{refiltra} provides a weight vector $\gordo{w} = (w_1,\dots,w_n) \in \Phi_Q$ such that 
\[
R = \field{k}\{x_1,\dots,x_n; Q, \preceq_{\gordo{w}}\}
\]
is a PBW algebra. Moreover, we can chose $w_i \geq 1$ and integer for all $i=1, \dots, n$. Now, for each $s \in \N$, define
\[
F_s^{\gordo{w}}R = \{ f \in R : \abs{\gordo{w}}{\exp{f}} \leqslant s \},
\]
where $\exp{f}$ is computed with respect to $\preceq_{\gordo{w}}$. By using that $\exp{fg} = \exp{f} + \exp{g}$ (see Theorem \ref{LPBW})  it is easy to check that $\{ F_s^{\gordo{w}}R : s \in \N \}$ is a filtration of $R$. Moreover, the associated graded algebra $\gr^{\gordo{w}}(R)$ is a quantum affine space $\mathcal{O}_{q}(\field{k}^n)$ of dimension $n$.  The filtration is obviously finite dimensional and $F_0^{\gordo{w}} R= \field{k}$.  Observe that this proves \eqref{pbww} $\Rightarrow$ \eqref{filtfin} in Theorem \ref{refiltrar}.

Given a finite dimensional vector subspace $M_0$ of $M$ that generates it as a left $R$-module, it is not difficult to see that the filtration $F_s^{\gordo{w}}M = F_s^{\gordo{w}}R \cdot M_0$ for $s \in \N$ is a \emph{good filtration}, that is, the associated graded left $\gr^{\gordo{w}}(R)$-module 
\[
\gr^{\gordo{w}}(M) = \bigoplus_{s \in \N} \frac{F_sM}{F_{s-1}M}
\]
 is finitely generated.  It follows from \cite[Proposition 6.5]{McConnell/Robson:1988} that
 \begin{equation}\label{GKdimM}
 \GKdim{M} = d(\dim_k F_s^{\gordo{w}}M) \,. 
 \end{equation}
 Thus, what we should compute is the degree of the $\gordo{w}$-weighted \emph{Hilbert function} of $M$ defined by $HF_M^{\gordo{w}}(s) = \dim_k F_s^{\gordo{w}}M$ for $s \in \N$. 
 
 \subsubsection{The Hilbert Function of a Module.} In order to use \eqref{GKdimM} to give an effective algorithm for computing the Gelfand-Kirillov dimension of $M$, we need a presentation $M = R^m/K$, where $K$ is a left $R$-submodule of a finitely generated free left $R$-module $R^m$. Recall from Section \ref{Buchberger} that if a finite set of generators of $K$ is explicitly given, then we may compute a Gr\"obner basis of $K$ with respect to $\preceq_{\gordo{w}}$ and, in particular, the basis of $\Exp(K)$. Our next aim is to show that the Hilbert function $HF_M^{\gordo{w}}$ of $M$, and its Gelfand-Kirillov dimension, is computable from the basis of $\Exp(K)$. To this end, observe that a $\field{k}$-basis of $M = R^m/K$ is
 \[
 \{ \gordo{x}^{\aalpha}\gordo{e}_i + K : (\aalpha,i) \notin \Exp(K) \}\,.
 \]
 Since $\aalpha \preceq_{\gordo{w}} \bbeta$ implies that $\abs{\gordo{w}}{\aalpha} \leqslant \abs{\gordo{w}}{\bbeta}$, we deduce that a $\field{k}$-basis of $F_sM$ is given by
 \[
 \{ \gordo{x}^{\aalpha} \gordo{e}_i + K : (\aalpha,i) \notin \Exp(K), \abs{\gordo{w}}{\aalpha} \leqslant s \}\,.
 \]
 Therefore,
 \[
 HF_M^{\gordo{w}} (s) = \card \{ (\aalpha,i) \notin \Exp(K), \abs{\gordo{w}}{\aalpha} \leqslant s \},
 \]
  where ``$\card$'' refers to the cardinal of a set. 
 We  thus see  that the $\gordo{w}$-weighted Hilbert function of $M$ depends only on the stable subset $\Exp(K)$ of $\Nnm$ (and on $\gordo{w}$, of course). Henceforth, its degree, which we know equals $\GKdim{M}$, depends ultimately  on  the basis of $\Exp(K)$. 
 
 \subsubsection{The Hilbert Function of a Stable Subset.}
 Let $E$ be a stable subset of $\Nnm$. The \emph{$\gordo{w}$-weighted Hilbert function} of $E$ is defined as 
 \[
 HF_E^{\gordo{w}}(s) = \card \{ (\aalpha,i) \notin E , \abs{\gordo{w}}{\aalpha} \leqslant s \}\,.
  \] 
 
 For every
$i= 1,\dots,m$, the set $E_{i} = \{\gordo{\alpha} \in\Nn; \,
(\gordo{\alpha},i) \in E \}$ is a monoideal of $\Nn$. Moreover,
\begin{equation}\label{disjunion}
E = \bigcup_{i=1}^{m} (\gordo{0},i) + E_i, 
\end{equation}
 Since the latter is a disjoint union, we get that 
\begin{equation}\label{HFsuma}
HF_E^{\gordo{w}} = HF_{E_1}^{\gordo{w}} + \cdots + HF_{E_m}^{\gordo{w}} 
\end{equation}
and
\[
d(HF_E^{\gordo{w}}) = \max \{ d(HF_{E_1}^{\gordo{w}}), \dots , d(HF_{E_m}^{\gordo{w}}) \} \,.
\]
We may thus reduce our problem to the case where $E \subseteq \Nn$ is a monoideal. By the inclusion-exclusion principle, we get that, if $E = E' \cup E''$ for monoideals $E', E'' \subseteq \Nn$, then
\begin{equation}\label{incex}
HF_E^{\gordo{w}} = HF_{E'}^{\gordo{w}} + HF_{E''}^{\gordo{w}} - HF_{E' \cap E''}^{\gordo{w}} \,.
\end{equation}
On the other hand, if $\aalpha_1, \dots, \aalpha_t \in \Nn$ is the basis of $E$, then
\[
E = (\aalpha_1 + \Nn) \cup \dots \cup (\aalpha_t + \Nn),  
\]
and, since $(\aalpha + \Nn) \cap (\bbeta + \Nn) = \aalpha \vee \bbeta + \Nn$, we get from \eqref{incex} that $HF_E^{\gordo{w}}$ is a linear combination, with integer coefficients, of functions of the form $HF_{\aalpha + \Nn}^{\gordo{w}}$, and the same applies, by \eqref{HFsuma}, when $E$ is a stable subset of $\Nn$. 

We know on the other hand (see e.g. \cite[Lemma 2.7]{Bueso/Gomez/Lobillo:2001b}) that $d(HF_E^{\gordo{w}}) = d(HF_E)$, where 
\[
HF_E(s) =  \card \{ (\aalpha,i) \notin E , \abs{}{\aalpha} \leqslant s \}.
\]

Since, for $\aalpha \in \Nn$, we have 
\begin{equation}\label{HFuno}
HF_{\aalpha + \Nn}(s) = \left\{\begin{array}{lcl} \binom{n+s}{s} & \text{ if } s < \abs{}{\aalpha} \\ \binom{n+s}{s} - \binom{n+s - \abs{}{\aalpha}}{s - \abs{}{\aalpha}}  &  \text{ if } s \geq \abs{}{\aalpha} \end{array} \right.
\end{equation}
we get from the previous discussion the following proposition. 

\begin{proposition}\label{HilbertPoly}
Let $E \subseteq \Nnm$ a stable subset, and assume that, in the decomposition \eqref{disjunion}, the basis $\{ \aalpha_1^i, \dots, \aalpha_{t_i}^i \}$ of $E_i$ is given for $i = 1, \dots, m$. Then, for every weight vector $\gordo{w}$ with strictly positive integer components, $d(HF_E^{\gordo{w}}) = d(HF_E)$. Moreover, there exists a polynomial with rational coefficients $HP_E$ such that $HF_E(s) = HP_E(s)$ for $s \geq \max \{ \abs{}{\aalpha_1^1 \vee \cdots \vee \aalpha_{t_1}^1}, \dots, \abs{}{\aalpha_1^{m} \vee \cdots \vee \aalpha_{t_m}^m} \}$
\end{proposition}

\subsubsection{The Effective Computation of the Gelfand-Kirillov Dimension.}
As for the Gelfand-Kirillov dimension of the left $R$-module $M = R^m/K$ concerns, we have the following consequence of Proposition \ref{HilbertPoly}. 

\begin{corollary}
Let $M = R^m/K$ be a finitely generated left module over a PBW algebra $R = \field{k}\{ x_1, \dots, x_n; Q, \preceq \}$. Let $\gordo{w} \in \Phi_Q$ any weight vector with strictly positive integer components, and compute the stable subset $\Exp(K)$ with respect to $\preceq_{\gordo{w}}$. Then $\GKdim{M}$ is the degree of the polynomial $HP_{\Exp(K)}$. Thus, in particular, $\GKdim{R} = n$. 
\end{corollary}

\begin{remark}
According to Proposition \ref{HilbertPoly}, the polynomial $HP_{\Exp(K)}$ may be computed by interpolation from its values at $s_0, s_0+1, \dots, s_0+ n -1$, where $s_0 = \max \{ \abs{}{\aalpha_1^1 \vee \cdots \vee \aalpha_{t_1}^1}, \dots, \abs{}{\aalpha_1^{m} \vee \cdots \vee \aalpha_{t_m}^m} \}$, and $\{ \aalpha_1^i, \dots, \aalpha_{t_i}^i \}$ is the basis of $E_i$ in the decomposition \eqref{disjunion} for $E = \Exp(K)$. Alternatively, one can compute each of the polynomials $HP_{E_i}$, for $i = 1, \dots, m$, and then compute $HP_E = HP_{E_1} + \cdots + HP_{E_m}$. The calculation of $HP_{E_i}$ can be done by interpolation, or recursively from the basis of $E_i$ by using \eqref{incex} and \eqref{HFuno}. 
\end{remark}

Each monoideal $E_i$ in the decomposition \eqref{disjunion} defines a monomial ideal of the commutative polynomial ring $\field{k}[x_1, \dots, x_n]$, and, hence, the Hilbert function $HF_{E_i}$ is, precisely, the Hilbert function of the corresponding monomial ideal. Thus, in order to compute $HP_{E_i}$, and its degree, we may use any algorithm available for monomial ideals. 

For 
$\aalpha \in \BbbN^n$, set
$
\supp (\aalpha) = \{ i \in \{1, \dots, n \} ~|~ \alpha_i \neq 0 \},
$
and define, for $i = 1, \dots, m$,
\[
V(E_i) = \{ \sigma \subseteq \{1,\dots,n \} ~|~ \sigma \cap \supp (\aalpha_k^{i}) \neq
\emptyset \; \forall k = 1, \dots, t_i \} \,.
\]

Then 
\begin{equation}\label{dimensionV}
d(HF_{E_i}) = n - \min \{ \card (\sigma) ~|~ \sigma \in V(E_i) 
\} \,.
\end{equation}

A proof of \eqref{dimensionV}, inspired in the material of
\cite[Section 9.3]{Becker/Weispfenning:1993}, can be seen in \cite[Section 4]
{Bueso/Castro/Gomez/Lobillo:1998}.

In conclusion, an algorithm to compute the Gelfand-Kirillov dimension of a given finitely
generated left $R$-module $M = R^m/K$ over a PBW algebra $R$ is described as follows. Given a set of generators
$\{ \gordo{f}_1, \dots, \gordo{f}_s \}$ for $K$, proceed according to the following steps:

\begin{enumerate}
\item Compute a  weight vector  $\gordo{w} = (w_1, \dots, w_n) \in \Phi_Q$ with $w_i \geq 1$ and integer. 
\item Compute a
Gr\"{o}bner basis $G$ for $K$ with respect to $\preceq_{\gordo{w}}$. 
\item Compute, from $G$, the basis $B$ of the stable subset $\Exp(K)$ of $\Nnm$.
\item Set, for $i = 1, \dots, m$, $B_i = \{ \aalpha \in \Nn : (\aalpha, i) \in B \}$.
\item Set $E_i = B_i + \Nn$, for $i = 1, \dots, m$, and compute $d(HF_{E_i})$.
\item $\GKdim{M} = \max \{ d(HF_{E_1}), \dots, d(HF_{E_m}) \}$.  
\end{enumerate}

\subsubsection{Holonomic Modules.}
The grade number $j(M)$ of a finitely presented left $R$-module was defined in Remark \ref{rem:jM}, and a procedure for its computation when $R$ is a left PBW ring was outlined. In the case that $R$ is a PBW algebra, the computation of $j(M)$ is much simpler, because it reduces to the computation of $\GKdim{M}$. This follows from the formula
\begin{equation}\label{CM}
j(M) + \GKdim{M} = \GKdim{R} 
\end{equation}
which is deduced from some results on algebras that do have a finite dimensional filtration with a ``nice'' associated graded algebra, as every PBW algebra does by Theorem \ref{refiltrar} (See \cite[Theorem 4.1]{Bueso/Gomez/Lobillo:2001b}). An algebra $R$ satisfying \eqref{CM} is said to be \emph{Cohen-Macaulay}.

Since $j(M)$ reaches its maximum for holonomic modules, we could then define holonomic modules over a PBW algebra as those having minimal Gelfand-Kirillov dimension (this minimum does exist because $\GKdim{M}$ is an integer for every finitely generated left $R$-module). A holonomic module $M$ is always of finite length, being the length of $M$ bounded by the multiplicity of $M$, computed from its Hilbert function $HF_M^{\gordo{w}}$ (see \cite[Theorem 2.8]{Gomez/Lenagan:2000}).

\subsubsection{(Re)filtering beyond PBW Algebras.}
The tight relationship between ``quantum relations'' and filtrations discussed in this section may be extended from PBW algebras to the more general  framework of ring extensions.  Let $A$ be any ring, and consider an iterated Ore extension of the form $A[x_1, \sigma_1]\cdots[x_n,\sigma_n]$, where $\sigma_j(x_i) = q_{ji}x_i$ for some $q_{ji} \in A$, for $1 \leq i < j \leq n$. Many good algebraic properties may be lifted from $A$ to $A[x_1, \sigma_1]\cdots[x_n,\sigma_n]$, specially when the elements $q_{ji}$ are units of $A$. Now, the strategy is to assume a ring  
extension $A \subseteq B$, such that $B$ is generated as a ring by $A$ and finitely many elements $x_1, \dots, x_n \in B$. If the relations among $A$ and the generators $x_1, \dots, x_n$ are not too complicated, a suitable filtration on $B$ with associated graded ring isomorphic to the iterated Ore extension $A[x_1, \sigma_1]\cdots[x_n,\sigma_n]$ can be defined, and the nice properties from $A[x_1,\sigma_1] \cdots [x_n, \sigma_n]$ (and, hence, from $A$) may be lifted to $B$. The easiest situation is when we assign degree $0$ to the elements of $A$, and degree $1$ to the generators $x_1, \dots, x_n$.  This gives the notion of a skew PBW extension from \cite{Lezama/Gallego:2011,Lezama/Reyes:2013}, where many interesting properties are lifted from $A$ to the skew PBW extension $B$. 

These skew PBW extensions are linear extensions of $A$, in the sense that the relations among $A$ and the generators take the form $x_iA \subseteq A + Ax_i, x_jx_i - q_{ji}x_ix_j \in A + Ax_1 + \cdots + Ax_n$ (see \cite{Gomez:2001} for the ``nonlinear'' setting). Let us mention that any left PBW ring is a left quantum bounded extension (\emph{extensi\'on cu\'antica acotada por la izquierda}), in the sense of \cite[Definici\'on 6]{Gomez:2001}, of the base division ring $\field{D}$,  but it needs not to be in general a skew PBW extension of $\field{D}$. An interesting problem here is to investigate under which circumstances it is possible to write a left PBW ring as an iterated skew PBW extension (this happens for some examples of PBW algebras, see \cite{Lezama/Reyes:2013}). 

Using these ideas (in the nonlinear case) in conjunction with a suitable generalization of Theorem \ref{refiltrar} (see \cite[Theorem 1]{Gomez/Lobillo:2004}), we deduce from \cite[Corollary 2]{Gomez/Lobillo:2004} that every PBW ring is Auslander-Regular and its Grothendieck group $K_0$ is trivial (see \cite{Bjork:1989} and \cite{McConnell/Robson:1988}, respectively, for these notions).  A suitable refinement of \cite[Theorem 1]{Gomez/Lobillo:2004} lead to prove that the complex quantum enveloping algebra $U_q(C)$ associated to any Cartan matrix $C$ is Auslander-Regular and Cohen-Macaulay \cite[Theorem 3, Theorem 5]{Gomez/Lobillo:2004}. The algebra $U_q(C)$, for a general $C$, seems not to be a PBW algebra nor a skew PBW extension of some ``easier enough'' subring. A detailed discussion of these topics lands beyond the scope of this overview.

\section{Appendix on Computer Algebra Systems \\ (by \emph{V. Levandovskyy})}
In this appendix, we discuss computer algebra systems, which 
provide support for the methods in the overview.

The website ``Oberwolfach References on Mathematical Software" 
\begin{center}
\url{http://orms.mfo.de/} 
\end{center}
is a web-interfaced collection of information and links on general mathematical software. In particular, it follows the {\it ORMS classification scheme} for mathematical software. 

There is a database on specialized software
\begin{center} \url{http://www.ricam.oeaw.ac.at/Groebner-Bases-Implementations/},\end{center} where Gr\"obner-related properties of systems for 
both commutative and non-commutative computations are described
by the authors of systems.

\subsection{Functionality of Systems for $\field{D}[x;\sigma,\delta]$}

To the best of our knowledge, no computer algebra system
provides computations over an arbitrary non-commutative skew field $\field{D}$ directly. 
Notably, the arithmetic operations over the skew field of fractions ${\rm Quot}(R)$ of $R$, 
where $R$ is a PBW algebra, are algorithmic (see \cite{Apel/Lassner:1988} for the case of universal enveloping algebras of Lie algebras, and \cite[Theorem 3.2]{Bueso/alt:2001} for the general case). However, in general 
the very basic arithmetic operations will invoke Gr\"obner bases over $R$.

From now on we denote by $\field{D}$ a (commutative) field. 

Let $A=\field{D}[x;\sigma,\delta]$ be a single Ore extension of $\field{D}$ with $\sigma$ bijective as described in Section \ref{sec:Ore}.
Then 
\begin{itemize}
\item division with rest,
\item extended greatest common right divisor (i.e., ${\rm gcrd}$ together with cofactors of its presentation via the input polynomials) and
\item extended least common left multiple (i.e., ${\rm lclm}$ and corresponding left quotients of the input polynomials)
\end{itemize}
can be computed with the help of packages {\sc OreTools} \cite{OreTools}, 
{\sc Ore$\_$algebra} \cite{Chyzak/Salvy:1998} in computer algebra system {\sc MAPLE} and 
{\sc ore$\_$algebra} \cite{OreAlgebraSage} in computer algebra system {\sc SAGE}.

Let $D=K(t)$ and $\delta(t)\in K[t]$.
Then a Jacobson form of a matrix with entries in $D[X;\sigma,\delta]$
can be computed with the library  {\tt jacobson.lib}
 \cite{Levandovskyy/Schindelar:2012,Jacobsonlib} of {\sc Singular:Plural}, 
by using fraction-free strategy. 

{\sc Maple} packages by Cheng et al. \cite{BCL06,DCL08}
provide two versions of the algorithm for computing an order basis of a polynomial matrix $M$ from an Ore algebra $A$, namely fraction-free version {\sc FFreduce} and a modular version {\sc Modreduce}.

Order bases are used for the computation of the left nullspace of
$M$ and indirectly for the computation of the Popov form of $M$. 
A Jacobson form can be obtained from the Popov form by further computation.

If $\field{D}$ is a differential field and $\delta$ is a derivation on $\field{D}$, the package {\sc Janet} for {\sc Maple} \cite{Robertz} provides the classical algorithm for the computation of a Jacobson normal form.

\subsection{Functionality of Systems for Multivariate Ore Algebras}

Left Gr\"obner bases together with other tools for multivariate Ore algebras
are available from the following packages:

\begin{itemize}
\item {\sc  Ore$\_$algebra} from the {\sc Mgfun} family \cite{Chyzak/Salvy:1998}, 
by F.~Chyzak et al. in computer algebra system {\sc MAPLE},
\item {\sc JanetOre} \cite{Robertz:2007} by D.~Robertz et al. in computer algebra system {\sc MAPLE},
\item {\sc HolonomicFunctions} \cite{HoloFun} by C.~Koutschan in computer algebra system {\sc MATHEMATICA}.
\end{itemize}

These systems can work with the operators, arising from the
following operations: differentiation, shift, Eulerian differentiation, forward difference, $q$-shift, $q$-differentiation (Jackson derivation), commutative multiplication. 

The latter system allows to define general $\field{D}[x,\sigma,\delta]$ as well.

Ore algebras found many applications to certain special functions \cite{Chyzak/Salvy:1998} as well as to algebraic systems and control theory \cite{Chyzak/Quadrat/Robertz:2005}, to name a few examples. In particular, 
collections of packages {\sc Mgfun} by Chyzak et al. (containing  {\sc  Ore$\_$algebra})
 and {\sc RISCErgoSum} from RISC Linz (containing {\sc HolonomicFunctions}) provide rich functionality for manipulations with special functions.

The package {\sc OreModules} \cite {Chyzak/Quadrat/Robertz:2007} for {\sc Maple} together with its subpackages allows to determine many module- and control-theoretic properties of linear systems over the Ore algebras available in the {\sc  Ore$\_$algebra} package.

\subsection{Functionality of Systems for PBW Algebras}

The three most important algorithms, namely 
\begin{itemize}
\item (left) Gr\"obner basis, 
\item first (left) syzygy module, 
\item (left) transformation matrix between a set of generators of a left submodule of a free module of a finite rank and its Gr\"obner basis 
\end{itemize}
are sometimes called \emph{the Gr\"obner trinity}. One can prove that all three objects can be obtained by only one Gr\"obner basis computation of the extended input. In the overview above we have seen their importance.

B.~Buchberger and B.~Sturmfels coined as \emph{Gr\"obner basics} the most
 fundamental applications of Gr\"obner bases, which include elimination of variables, kernel of a module (resp. ring) homomorphism, Hilbert series, various dimensions etc.

As we have indicated in the beginning, at the moment no system supports
skew fields as coefficient domains, hence PBW rings cannot be treated. 
On the contrary, PBW algebras have been adressed by at least three systems.

\subsubsection{General PBW Algebras.}

\begin{itemize}
\item {\sc Felix} by J.~Apel and U.~Klaus \cite{FelixSys} provides Buchberger's algorithm and its 
generalizations, including syzygy computations and basic ideal operations.
\item {\sc MAS} by H.~Kredel and M.~Pesch \cite{MASsys} contains a large library of implemented Gr\"obner basis algorithms, covering most of Gr\"obner basics.
\item \textsc{Singular:Plural} by V.~Levandovskyy et al. \cite{Plural} is a part of \textsc{Singular}, responsible 
for computations with the most general PBW algebras (which are addressed as $G$-algebras in this system)
as well as with the factor algebras of PBW algebras modulo two-sided ideals. It allows to work over any field and use any well-ordering, available in {\sc Singular}. 
Except for Gr\"obner basics, numerous algorithms are implemented in more than 20 {\sc Plural} libraries.
In particular, {\sc Plural} has the only implementation of the computation of Gelfand-Kirillov dimension of finitely presented modules, known to us. 
\end{itemize}

Unfortunately, the development of systems {\sc Felix} and {\sc MAS} has ceased by now. Both systems are still available for download and perform nicely. They, however, do not fully support quantum algebras.

\subsubsection{Special PBW Algebras.}

In a variety of situations one is interested in working with algebras of operators
with variable coefficients, like Weyl or shift algebras with coefficients in $K[x_1,\ldots,x_n]$.

{\sc MAPLE} packages {\sc Janet} resp. {\sc LDA} by D.~Robertz et al. \cite{Robertz,LDA}
compute Gr\"obner and Janet bases of left ideals over rings of linear differential resp. difference operators.

A {\sc Singular} subsystem {\sc SCA} by O.~Motsak \cite{SCA} provides standard and Gr\"obner bases, syzygies and free resolutions as well as Gr\"obner basics for \emph{graded $\mathbb{Z}_2$-commutative algebras}, that is
tensor products over the field $K$ of (local or global) commutative algebras with an exterior algebra.

The system \textsc{Macaulay2} by D.~Grayson and M.~Stillman
\cite{MAc2} includes various Gr\"obner bases-based algorithms for exterior and Weyl algebras.

\subsubsection{$D$-Modules.}
The challenging problems in the realm of algebraic resp. analytic $D$-modules, that is systems
of linear partial differential equations with polynomial resp. power series coefficients attracted 
the attention of computer algebraists since decades.

The experimental system \textsc{Kan/sm1} \cite{KAN} by N.~Takayama et al. provides Gr\"obner basis computations in polynomial rings, rings of differential operators, rings of difference and $q$-difference operators. 
Its functionality for $D$-modules is remarkable, providing implementations for many algorithms from 
the book \cite{SST:2000}.

The package {\tt D-modules.m2} for {\sc Macaulay2} \cite{dmodMac2} provides
very reach functionality for computations with $D$-modules, including polynomial/rational and holonomic 
solutions of systems as well as numerous invariants for singularities.

The system {\sc RISA/ASIR} by M.~Noro et al. \cite{Asir} provides newly implemented functionality, similar to {\sc Kan/sm1} on the higher level of performance. Notably are many implemented algorithms for $D$-modules.

{\sc Singular:Plural} has implementations of many algorithms for $D$-module theory as well \cite{ABLMS}, 
including multivariate Weyl closure of a left ideal, polynomial/rational solutions and various local invariants for singularities.

\subsubsection{Factorization of Non-Commutative Polynomials.}

H.~Melenk and J.~Apel created a package for the computer algebra system \textsc{REDUCE}
\cite{Apel/Melenk:1994}, which provides tools to deal with a big class of non-commutative multivariate polynomial algebras. Among other, it contains an algorithm for factorization of polynomials over supported algebras.

M.~van Hoeij developed an algorithm to factorize a differential operator with rational coefficients \cite{VanHoeij:1997}. This technique was enhanced and extended further, in particular 
to the case of power series coefficients. Nowadays this algorithm is implemented in the {\sc DETools} package of \textsc{Maple} as the standard algorithm for factorization of such operators.

The computer algebra web-service \textsc{ALLTYPES} is based on computer algebra system \textsc{REDUCE}. 
It is only accessible as web-service and features the algorithm for factoring differential operators due to F.~Schwarz and D.~Grigoriev \cite{GrigorievSchwartz:2004}.

In {\sc Singular:Plural} there is a library {\tt ncfactor.lib} by A.~Heinle \cite{Ncfactorlib}, which provides algorithms for factorization of polynomials over univariate Weyl and shift algebras. Moreover, $\mathbb{Z}$-graded
polynomials over $q$-Weyl algebras can be factorized as well. Notably, there is an experimental
implementation of algorithms for the factorization of polynomials over multivariate algebras as above. 

\subsection{Further Systems}

Here, we briefly mention some other projects, having relevance to the topics of the overview.

The {\sc Homalg} project \cite{Barakat/Robertz:2008} is a multi-author multi-package open source software project for constructive homological algebra.

The package \textsc{ISOLDE}  \cite{ISOLDE} for {\sc Maple} contains symbolic algorithms for solving systems of ordinary linear differential equations, and more generally linear functional matrix equations. Some commands of
the \textsc{ISOLDE} package have been extended to handle algebraic integrable
connections \cite{integrable}.

The package \textsc{LinearFunctionalSystems} for {\sc Maple} provide, among other, implementation of algorithms
for finding polynomial/rational/power series solutions of a linear ($q$-) difference system with polynomial coefficients. These algorithms are based on the implementation of the EG-elimination algorithm by S.~Abramov. 
See also the package \textsc{LRETools} for {\sc Maple} for functions, manipulating and finding certain types of solutions of linear recurrence equations and the package \textsc{QDifferenceEquations} for finding 
polynomial/rational or $q$-hypergeometric solutions for a linear $q$-difference equation with polynomial coefficients.

There are several packages {\sc Maple} by Y.~Cha \cite{Cha} in particular, computing closed-form solutions for second-order homogeneous linear ordinary difference operators with rational coefficients. Also, homomorphisms between two linear ordinary difference operators can be computed.


\begin{thebibliography}{}


\bibitem{OreTools}
Abramov,  S. A., Le, H. Q., Li, Z.: 
 Oretools: a computer algebra library for univariate {O}re polynomial
  rings.
 Technical report, 2003.
 Technical Report CS-2003-12. University of Waterloo.
 
 
\bibitem{Adams/Loustaunau:1994}
Adams, W., Loustaunau, P.:
An introduction to Gr\"obner Bases, AMS, Providence RI, 1994.


\bibitem{Anderson/Fuller:1992}
Anderson, F.W., Fuller, K.R.:
Rings and categories of modules, 2nd ed.
Springer-Verlag, New York, 1992.

\bibitem{ABLMS}
Andres, D., Brickenstein, M., Levandovskyy, V., Mart{\'i}n-Morales, J., Sch{\"o}nemann, H.:
Constructive $D$-Module Theory with \textsc{SINGULAR}, Math. Comput. Sci. \textbf{4}:2-3 (2010), 359--383.


\bibitem{Apel:1992}
Apel, J.: 
A relationship between Gr\"obner bases of ideals and vector modules of G-algebras, Contemporary Mathematics \textbf{21} (1992), part 2, 195--204.


\bibitem{FelixSys}
Apel, J.,  Klaus, U.: 
 {FELIX}, a special computer algebra system for the computation in
  commutative and non-commutative rings and modules, 1998.
 URL: \url{http://felix.hgb-leipzig.de/}.

 
\bibitem{Apel/Lassner:1988}
Apel, J., Lassner, W.: 
An extension of Buchberger's algorithm and calculations in enveloping fields of Lie algebras,
J. Symbolic Comput. \textbf{6} (1988), 361--370.

\bibitem{Apel/Melenk:1994}
Apel, J., Melenk, H.:   
{\sc REDUCE} package {\sc NCPOLY}: Computation in non-commutative polynomial ideals. Preprint, Konrad-Zuse-Zentrum Berlin (ZIB), 1994.

\bibitem{Barakat/Robertz:2008}
Barakat, M., Robertz, D.:  homalg: a meta-package for homological algebra, J. Algebra Appl. \textbf{7} (2008), 299-317.
URL: \url{http://homalg.math.rwth-aachen.de}.

\bibitem{Barkatou:1999} 
Barkatou, M. A;  Rational solutions of matrix difference equation. problem of
equivalence and factorisation, in Proceedings of ISSAC '99, ACM Press, pp.  277--282, 
1999.

\bibitem{Barkatou:2007}
 Barkatou, M.A.:
Factoring systems of linear functional equations using eigenrings.
Latest Advances in Symbolic Algorithms, Proc. of the Waterloo
  Workshop, Ontario, Canada, I. Kotsireas and E. Zima (Eds.),
  World Scientific, pp. 22--42, 2007.

\bibitem{integrable}
Barkatou, M. A., Cluzeau, T., El Bacha, C., Weil, J.-A.:
Computing closed form solutions of integrable connections,
Proceedings of the International Symposium on Symbolic and Algebraic
Computations (ISSAC '12), ACM Press, pp. 43--50, 2012. URL: \url{http://www.ensil.unilim.fr/~cluzeau/PDS.html}. 


\bibitem{Barkatou/Pflugel:1998} 
Barkatou, P.A., Pfl\"ugel, E.: On the Equivalence
Problem of Linear Differential Systems and its Application
for Factoring Completely Reducible Systems. In
Proceedings of ISSAC '98, ACM Press, pp. 268--275,  
1998.

\bibitem{ISOLDE}
Barkatou, M.A.,  Pfl\"ugel, E.: The ISOLDE package. A SourceForge Open Source project, 2006. 
URL: \url{http://isolde.sourceforge.net}


\bibitem{Becker/Weispfenning:1993}
Becker, T. Weispfenning, V.: {Gr\"obner} bases. {A} computational
  approach to commutative algebra, Springer-Verlag, 1993.
  

  \bibitem{BCL06}
Beckermann, B.,  Cheng, H., Labahn. G.:
 Fraction-free row reduction of matrices of {O}re polynomials.
J. Symbolic Comput. \textbf{41} (2006), 513--543. URL: \url{http://www.cs.uleth.ca/~cheng/software/}

  
\bibitem{Bergman:1978}
Bergman, G.: The diamond lemma for ring theory, Adv. Math. \textbf{29} (1978), 178--218. 
  
\bibitem{Bjork:1979}
Bj\"ork, J. E.: 
Rings of differential operators, 
Mathematical Library, 21, North-Holland, Amsterdam, 1979. 


\bibitem{Bjork:1989}
Bj\"ork, J.-E.: The Auslander condition on Noetherian rings, In: M.-P. Malliavin (ed.), {S\'eminaire d'{A}lg\`ebre {P}aul {D}ubreil et {M}arie-{P}aul {M}alliavin, 39\`eme {A}nn\'ee ({P}aris, 1987/1988)}, Lecture Notes in Math. 1404, Springer,
New York, 1989, pp. 137--173.


\bibitem{Robertz}
 Blinkov, Y.~A., Cid, C.~F., Gerdt, V.~P.,  Plesken, W.,  Robertz, D. 
 The {MAPLE} package ``{J}anet'': {II}. {L}inear {P}artial {D}ifferential
  {E}quations.
 In: Proceedings of the 6th International Workshop on Computer
  Algebra in Scientific Computing, pp. 41--54, 2003.
 URL: \url{http://wwwb.math.rwth-aachen.de/Janet}.
 
 
\bibitem{Bronstein/Petkovsek:1991}
Bronstein, M., Petkov\v sek, M.:
An introduction  to pseudo-linear algebra,
Theoret. Comput. Sci. \textbf{157} (1996), 3--33.

\bibitem{Bueso/Castro/Gomez/Lobillo:1998}
Bueso, J.L., Castro, F. J., G\'omez-Torrecillas, J., Lobillo, F. J.: 
An introduction to effective calculus in quantum groups, Rings,
  {Hopf} algebras and {Brauer} groups. (S.~Caenepeel and A.~Verschoren, eds.)
  Marcel Dekker, 1998, pp.~55--83.
  
  
\bibitem{Bueso/alt:2001} 
  Bueso, J.L., Castro, F. J., G\'omez-Torrecillas, J., Lobillo, F. J.: 
Primality Test in iterated Ore extensions, Comm. Algebra \textbf{29} (2001), 1357--1371.

  
\bibitem{Bueso/Castro/Jara:1997}
Bueso, J.L., Castro, F.J., Jara, P.: 
The effective computation of the Gelfand-Kirillov dimension, Proc. Edinburgh Math. Soc. \textbf{40} (1997), 111--117. 
  
\bibitem{Bueso/Gomez/Lobillo:2001a}
Bueso, J. L., G\'omez-Torrecillas, J., Lobillo, F.J.: 
Homological Computations in PBW Modules,
Alg. Repr. Theory \textbf{4} (2001), 201--218.

\bibitem{Bueso/Gomez/Lobillo:2001b}
Bueso, J. L., G\'omez-Torrecillas, J., Lobillo, F.J.: 
Computing the  Gelfand-Kirillov  dimension, II. 
In: Ring Theory and Algebraic Geometry, (A. Granja, J.A. Hermida, A. Verschoren, eds.) Marcel Dekker, New York, 2001, pp. 33--57. 

\bibitem{Bueso/Gomez/Lobillo:2001c}
Bueso, J. L., G\'omez-Torrecillas, J., Lobillo, F.J.:
Re-filtering and exactness of the Gelfand-Kirillov dimension, Bull. Sci. Math. \textbf{125} (2001), 689--715.

\bibitem{Bueso/Gomez/Verschoren:2003}
Bueso, J. L., G\'omez-Torrecillas, J., Verschoren, A.:
Algorithmic methods in Non-Commutative Algebra. Applications to quantum groups. 
Kluwer Academic Publishers, Dordrecht, 2003. 

\bibitem{Caruso/LeBorgne} 
Caruso, X., Le Borgne, J.; Some algorithms for skew polynomials over finite fields. URL: \url{http://arxiv.org/abs/1212.3582}

\bibitem{Castro:1984}
Castro, F.J.:
Th\'eor\`eme de division pour les op\'erateurs diff\'erentielles et calcul des multiplicit\'es. Th\`ese $3^{eme}$ cycle, Univ. Paris VII, 1984.

\bibitem{Cha} Cha, Y.:  Packages {\sc tausqsols}, { \sc solver}, {\sc  Hom}, 2012. URL: \url{https://sites.google.com/site/yongjaecha/code}.


\bibitem{Churchill/Kovacic:2002}
Churchill, R.C., Kovacic, J.J.: 
Cyclic vectors. In: Proceedings of Differential algebra
and related topics (Newark, NJ, 2000), 191--218,
World Sci. Publ., River Edge, NJ, 2002.


\bibitem{Chyzak/Quadrat/Robertz:2005}
Chyzak, F., Quadrat, D., Robertz, D.: Effective algorithms for parametrizing linear control systems over Ore algebras, Appl. Algebra Eng. Comm. Comput. \textbf{16} (2005), 319--376.

\bibitem{Chyzak/Quadrat/Robertz:2007}
Chyzak, F.,  Quadrat, A., Robertz, D.: OreModules: A symbolic package for the study of multidimensional linear systems, in Applications of Time Delay Systems, Lecture Notes in Control and Inform. Sci., Vol. 352 (Springer, Berlin, 2007), pp. 233--264. URL: \url{http://wwwb.math.rwth-aachen.de/OreModules}

\bibitem{Chyzak/Salvy:1998}
Chyzak, F., Salvy, B.: 
Non-commutative elimination in Ore algebras proves multivariate identities, J. Symbolic Comput. \textbf{26} (1998), 187--227. 
URL: \url{http://algo.inria.fr/chyzak/mgfun.html} 


\bibitem{Cluzeau/Quadrat:2008}
Cluzeau, T., Quadrat, A.: Factoring and decomposing a class of linear functional systems, Linear Algebra Appl. \textbf{428} (2008), 324--381.

\bibitem{Cluzeau/Quadrat:2009}
Cluzeau, T., Quadrat, A.: OreMorphisms: a homological algebraic package for factoring, reducing and decomposing linear functional systems, Topics in time delay systems, Lecture Notes in Control and Inform. Sci., vol. 388, Springer, Berlin, 2009, pp. 179--194.

\bibitem{Cohn:1971} 
Cohn, P.M.: Free rings and their relations, Academic Press, London, 1971. 



\bibitem{DCL08}
Davies, P.,  Cheng, H.,  Labahn, G: 
 Computing {P}opov form of general {O}re polynomial matrices.
 In: Proceedings of the Milestones in Computer Algebra (MICA)
  Conference, pp. 149--156, 2008. URL: \url{http://www.cs.uleth.ca/~cheng/software/}
  

\bibitem{Dixmier:1977} 
Dixmier, J.: 
Enveloping algebras, North-Holland, Amsterdam, 1977.


\bibitem{Foldenauer:2012}
Foldenauer, A. C.:
Gr\"obner Bases for Bimodules and its
Applications: Jacobson normal form in centerless Ore extensions and
Gr\"obner Basis theory for Bimodules in $G$-algebras. Diploma thesis, Univ. Aachen, 2012.


\bibitem{Galligo:1983}
Galligo, A.: Algorithmes de calcul de base standards, preprint, 1983. 


\bibitem{LDA}
 Gerdt, V.~P.,  Robertz, D.:
 A {M}aple package for computing {G}r\"obner bases for linear
  recurrence relations.
 Nuclear Instruments and Methods in Physics Research Section A,
  \textbf{559} (2006), 215--219.
 URL: \url{http://arxiv.org/abs/cs/0509070}.
 
\bibitem{Giesbrecht:1998}
Giesbrecht, M.:  Factoring in skew-polynomial rings over finite fields,  J. Symbolic Comput. \textbf{26} (1998), 463--486. 


\bibitem{Giesbrecht/Heinle:2012}
Giesbrecht, M.,  Heinle, A.:  A polynomial-time algorithm for the Jacobson form of a matrix of Ore polynomials. In Proc. Computer Algebra in Scientific Computing. pp. 117--128, Lecture Notes in Computer Science, vol. 7442, Springer, 2012. 


\bibitem{Giesbrecht/Zhang:2003}
Giesbrecht, M., Zhang, Y.:  Factoring and Decomposing Ore Polynomials over $\mathbb{F}_q(t)$. Proceedings of the 2003 International symposium on Symbolic and algebraic computation (ISSAC '03), ACM Press, pp. 127--134, 2003. 

\bibitem{Gluesing/Schmale:2004}
Gluesing-Luerssen, H.,  Schmale, W.:  On cyclic convolutional codes.  Acta Appl. Math. \textbf{82} (2004) 183--237.

\bibitem{Gomez:1999} 
G\'omez-Torrecillas, J.: 
Gelfand-Kirillov dimension of multi-filtered algebras, Proc. Edinburgh Math. Soc. \textbf{52} (1999), 155--168.

\bibitem{Gomez:2001}
G\'omez-Torrecillas, J.: 
Regularidad de las \'algebras envolventes cuantizadas, Actas del Encuentro de Matem\'aticos Andaluces, Vol. 2, pp. 493--500, ISBN
8447206394, Sevilla, 2001

\bibitem{Gomez/Lenagan:2000}
G\'omez-Torrecillas, J., Lenagan, T. H.:
Poincar\'e series of multi-filtered algebras and partitivity, 
J. London Math. Soc. \textbf{62} (2000),  370--380.

\bibitem{Gomez/Lobillo:2004}
G\'omez-Torrecillas, J., Lobillo, F.J.:
Auslander-Regular and Cohen-Macaulay
Quantum Groups, Algebras Repr. Theory \textbf{7} (2004), 35--42. 


\bibitem{Gomez/alt:unp}
G\'omez-Torrecillas, J., Lobillo, F. J., Navarro, G.: 
Computing the bound of an Ore polynomial. Applications to factorization. URL: \url{http://arxiv.org/abs/1307.5529}


\bibitem{MAc2}
Grayson, D, Stillman, M.:
 {Macaulay} 2, a software system for research in algebraic geometry,
  2013.
 URL: \url{http://www.math.uiuc.edu/Macaulay2}.
 

\bibitem{Plural}
Greuel, G.-M., Levandovskyy, V., Motsak, O., Sch{\"o}nemann, H.:
 {\textsc{Plural}. A \textsc{Singular} 3.1 subsystem for computations
  with non-commutative polynomial algebras. Centre for Computer Algebra, TU
  Kaiserslautern.}, 2010.
 URL: \url{http://www.singular.uni-kl.de}.

\bibitem{SCA}
Greuel, G.-M., Motsak, O., Sch{\"o}nemann, H.:
{\sc Singular:SCA}. {A} {\sc Singular} {3.1} subsystem for
computations with graded commutative algebras, 2011. 
 URL: \url{http://www.singular.uni-kl.de}.

\bibitem{GrigorievSchwartz:2004}
Grigoriev, D., Schwarz, F.: Factoring and solving linear partial differential equations, Computing \textbf{73} (2004), 179--197. URL: \url{http://www.alltypes.de/}.


\bibitem{Hazewinkel/alt:2004} 
Hazewinkel, M., Gubareni, N., Kirichenko, V.V.:
Algebras, Rings and Modules, volume 1.
Kluwer Academic Publishers, Dordrecht, 2004.


\bibitem{Ncfactorlib}
Heinle, A., Levandovskyy V.:
 \texttt{ncfactor.lib}. A \textsc{Singular} 3.1 a library for factorization in some non-commutative algebras, 2013.
 URL: \url{http://www.singular.uni-kl.de}.
 

\bibitem{Hilton/Stammbach:1971}
Hilton, P.J., Stammbach, U.: 
A course in Homological Algebra, Springer-Verlag, New York, 1971.

\bibitem{Jacobson:1943}
Jacobson, N.:
The Theory of Rings, 
AMS, Providence, RI, 1943.

 \bibitem{Jacobson:1980}
Jacobson, N.:
Basic Algebra. II. W. H. Freeman and Co., San Francisco, Calif., 1980. 

\bibitem{Jacobson:1996}
Jacobson, N.:
Finite-Dimensional Division Algebras over Fields,
Springer-Velag, New York, 1996. 

\bibitem{Kandri-Rody/Weispfenning:1988}
Kandri-Rody, A., Weispfenning, V.: 
Non-commutative Gr\"obner bases in algebras of solvable type,
J. Symb. Comp. \textbf{6} (1990), 231-248.



\bibitem{OreAlgebraSage}
Kauers, M.,   Jaroschek, M., Johansson, F.:
 Ore polynomials in {S}age., 2013.
 URL: \url{http://arxiv.org/abs/1306.4263}.
 


\bibitem{HoloFun}
Koutschan, C.:
 HolonomicFunctions (user's guide).
 Technical report, 2010.
 Technical report no. 10-01 in RISC Report Series, University of Linz,
  Austria.
 URL:
  \url{http://www.risc.jku.at/research/combinat/software/ergosum/RISC/HolonomicFunctions.html}.
  


\bibitem{Krause/Lenagan:2000}
 Krause, G. R., Lenagan, T. H.: 
Growth of Algebras and Gelfand-Kirillov Dimension, Revised Ed. 
A.M.S., Rhode Island, 2000.

\bibitem{Kredel:1993}
Kredel, H.:
Solvable Polynomial Rings, 
Verlag Shaker, Aachen, 1993.


\bibitem{MASsys}
Kredel, H.,  Pesch, M.:
 {MAS}, modula-2 algebra system, 1998.
 URL: \url{http://krum.rz.uni-mannheim.de/mas.html}.
 
 
\bibitem{Lassner:1985}
Lassner, W.: An extension of Buchberger's algorithm and calculations in enveloping fields of Lie algebras, Proc. EUROCAL `85, Linz 1985, Lect. Notes Comp. Sci. 204, pp. 99-115, 1985.

\bibitem{Lejeune:1984}
Lejeune-Jalabert, M.: Effectivit\'e des calculs polynomiaux, Cours de D.E.A., Univ. Gr\'enoble, 1984--1985.

\bibitem{Leroy:1995} 
Leroy, A.: Pseudo linear transformations and evaluation in Ore extensions. Bull. Belg. Math. Soc. \textbf{2} (1995), 321--347.

\bibitem{Levandovskyy:2005}
Levandovskyy, V.: Non-commutative Computer Algebra for Polynomial Algebras: Gr\"obner Bases, Applications and Implementation, Ph. D. Thesis, Univ. Kaiserslautern, 2005. URL: \url{http://kluedo.ub.uni-kl.de/volltexte/2005/1883/}

\bibitem{Levandovskyy:2006}
Levandovskyy, V.: PLURAL, a non-commutative extension of SINGULAR: past, present and future.  Mathematical software, ICMS  2006,  LNCS, Springer, Berlin, 2006, pp. 144--157.

\bibitem{Levandovskyy/Schindelar:2012}
Levandovskyy, V., Schindelar, K.:  Fraction-free algorithm for the computation of diagonal forms matrices over Ore domains using Gr\"obner bases,  J. Symbolic Comput. \textbf{47} (2012), 1214--1232.

\bibitem{Levandovskyy/Schoenemann:2003}
Levandovskyy, V., Sch\"onemann, H.: PLURAL ---  a computer  algebra system for noncommutative polynomial algebras, Proceedings of the 2003 International Symposium on Symbolic and Algebraic Computation (New York), ACM, 2003, pp. 176--183.


\bibitem{Lezama/Gallego:2011}
Lezama, O., Gallego, C.:
Gr\"obner bases for ideals of sigma-PBW extensions, Comm. Algebra \textbf{39} (2011), 50--75.

\bibitem{Lezama/Reyes:2013}
Lezama, O., Reyes, A.: 
Some homological properties of skew PBW extensions,
\url{http://arxiv.org/abs/1310.6639}


\bibitem{Lobillo:1998}
Lobillo, F.J.:
M\'etodos Algebraicos y Efectivos en Grupos Cu\'anticos, Tesis doctoral, Universidad de Granada, 1998. 

\bibitem{Lorenz:1988}  
Lorenz, M.: 
Gelfand-Kirillov dimension and Poincar\'e series, 
Cuadernos de \'Algebra 7, Universidad de Granada, 1988. 

\bibitem{McConnell/Robson:1988} 
McConnell, J.C.,  Robson, J.C.: 
 Noncommutative noetherian rings, Wiley Interscience, New York, 1988. 

 
 \bibitem{McConnell/Stafford:1989}
 McConnell, J.C., Stafford, J.:
 Gelfand-Kirillov dimension and associated graded modules, J. Algebra \textbf{125} (1989), 197--214.
 
 \bibitem{Mora:1994}
 Mora, T.: 
 An introduction to commutative and non-commutative Gr\"obner bases, Theoret. Comput. Sci. \textbf{134} (1994), 131--173.
 
 \bibitem{Mora/Robbiano:1988}
 Mora, T., Robbiano, L.: 
 The Gr\"obner fan of an ideal,  J. Symbolic Comput. \textbf{6} (1988), 183--208.
 
 
 \bibitem{Asir}
Noro, M.,  Shimoyama, T.,  Takeshima, T.:
 {Risa/Asir}, an open source general computer algebra system, 2012.
 URL: \url{http://www.math.kobe-u.ac.jp/Asir}.


\bibitem{Ore:1933}
Ore, \O.: Theory of non-commutative polynomials, Ann. Math. \textbf{34} (1933), 480--508.



\bibitem{Quadrat:2013}
Quadrat, A.: 
Grade Filtration of Linear Functional Systems,
Acta Appl. Math. \textbf{127} (2013), 27--86.


\bibitem{Robbiano:1986}
Robbiano, L.: On the theory of graded structures,  J. Symbolic Comput. \textbf{2} (1986), 139--186.


\bibitem{Robertz:2007}
Robertz, D.: 
Janet Bases and Applications
in: M. Rosenkranz, D. Wang (editors),
Gr\"obner Bases in Symbolic Analysis,
de Gruyter, Berlin, 2007, pp. 139--168.
URL: \url{http://wwwb.math.rwth-aachen.de/Janet/janetore.html}.

 
\bibitem{Ronyai:1987}
R\'{o}nyai, L.:
 Simple algebras are difficult, in: Proceedings of the Nineteenth
  Annual ACM Symposium on Theory of Computing, STOC '87, ACM, 1987, pp. 398--408.
  

\bibitem{SST:2000}
Saito, M., Sturmfels, B., Takayama, N.:
Gr\"obner deformations of hypergeometric differential equations. Berlin: Springer, 2000.

    
\bibitem{Jacobsonlib}
Schindelar, K.,  Levandovskyy, V.: 
 A \textsc{Singular} 3.1 library with algorithms for {S}mith and
  {J}acobson normal forms \texttt{jacobson.lib}., 2009.
 URL: \url{http://www.singular.uni-kl.de}.
 
\bibitem{Schreyer:1980}
Schreyer, F.: 
Die Berechnung von Syzygien mit dem verallgemeinerten Weierstrachen Divisionsatz, Diplomatbeit, Universit\"at Hamburg, 1980. 
  
\bibitem{Singer:1996} 
Singer, M.F.: Testing reducibility of linear differential operators: a group theoretic perspective. Appl. Algebra Eng. Commun. Comput. \textbf{7} (1996), 77--104.

\bibitem{Stenstrom:1975}
Strenstr\"om, B.: Rings of Quotients, Springer-Verlag, Berlin, 1975.

\bibitem{KAN}
Takayama, N.:
\textsc{kan/sm1}, a Gr{\"o}bner engine for the ring of differential and difference operators, 2003.
 URL: \url{http://www.math.kobe-u.ac.jp/KAN/index.html}

\bibitem{dmodMac2}
Tsai, H., Leykin, A.:
 $D$-modules package for {Macaulay} 2 -- algorithms for {D}--modules,
  2006.
 URL: \url{http://people.math.gatech.edu/~aleykin3/Dmodules/}.

\bibitem{VanDerPut/Singer:2003} 
Van der Put, M.,  Singer, M. F.:  Galois Theory of Linear Differential Equations, Springer-Verlag, New York,  2003. 

\bibitem{VanHoeij:1997}
Van Hoeij, M.:
Factorization of Differential Operators with Rational Functions Coefficients, 
 J. Symbolic Comput. \textbf{24} (1997), 537--561

\bibitem{Weispfenning:1987}
Weispfenning, V.: 
Constructing universal Gr\"obner bases, in: Proceedings of AAECC 5, Lecture Notes in Comput. Sci., vol. 356, Springer, 1987, pp. 408--417.
\end{thebibliography}
\end{document}